\documentclass{amsart}
\usepackage[latin1]{inputenc}
\usepackage{amssymb}
\usepackage{latexsym}
\usepackage{amscd}
\usepackage{longtable}
\usepackage{amsmath}
\usepackage{mathrsfs}
\usepackage{array}
\usepackage[all]{xy}
\usepackage{a4}

\usepackage{todonotes}

\newtheorem{defn0}{Definition}[section]
\newtheorem{prop0}[defn0]{Proposition}
\newtheorem{thm0}[defn0]{Theorem}
\newtheorem{lemma0}[defn0]{Lemma}
\newtheorem{claim0}[defn0]{Claim}
\newtheorem{corollary0}[defn0]{Corollary}
\newtheorem{example0}[defn0]{Example}
\newtheorem{remark0}[defn0]{Remark}
\newtheorem{assumption0}[defn0]{Assumption}
\newtheorem{conjecture0}[defn0]{Conjecture}
\newtheorem{notation0}[defn0]{Notation}
\newtheorem{question0}[defn0]{Question}

\newenvironment{definition}{\begin{defn0}\rm}{\end{defn0}}
\newenvironment{proposition}{\begin{prop0}}{\end{prop0}}
\newenvironment{theorem}{\begin{thm0}}{\end{thm0}}
\newenvironment{lemma}{\begin{lemma0}}{\end{lemma0}}

\newenvironment{corollary}{\begin{corollary0}}{\end{corollary0}}

\newenvironment{remark}{\begin{remark0}\rm}{\end{remark0}}
\newenvironment{assumption}{\begin{assumption0}\rm}{\end{assumption0}}
\newenvironment{conjecture}{\begin{conjecture0}}{\end{conjecture0}}

\def \mint {\times \hskip -1.1em \int}

\newcommand{\Gal}{{\mathrm {Gal}}}

\newcommand{\disc}{{\mathrm {disc }}}
\newcommand{\ord}{\mathrm{ord}}

\newcommand{\Norm}{\mathrm{N}}

\newcommand{\M}{\mathrm{M}}

\newcommand{\Jac}{\mathrm{Jac}}

\newcommand{\Ind}{{\mathrm{Ind}}}
\newcommand{\PGL}{{\mathrm{PGL}}}

\newcommand{\GL}{{\mathrm{GL}}}
\newcommand{\End}{{\mathrm{End}}}

\newcommand{\SO}{{\mathrm {SO}}}
\newcommand{\rO}{{\mathrm {O}}}

\newcommand{\ind}{{\mathrm {ind}}}
\newcommand{\df}{{\mathrm {def}}}

\newcommand{\Z}{{\mathbb Z}}
\newcommand{\A}{{\mathbb A}}
\newcommand{\Q}{{\mathbb Q}}
\newcommand{\C}{{\mathbb C}}
\newcommand{\R}{{\mathbb R}}

\newcommand{\N}{{\mathbb N}}
\newcommand{\G}{{\mathbb G}}

\newcommand{\PP}{{\mathbb P}}

\newcommand{\cA}{{\mathcal A}}
\newcommand{\cQ}{{\mathcal Q}}

\newcommand{\cN}{{\mathcal N}}

\newcommand{\cD}{{\mathcal D}}

\newcommand{\cF}{{\mathcal F}}

\newcommand{\cX}{{\mathcal X}}
\newcommand{\cY}{{\mathcal Y}}
\newcommand{\cH}{{\mathcal H}}

\newcommand{\cG}{{\mathcal G}}
\newcommand{\dG}{{\mathfrak g}}
\newcommand{\cV}{{\mathcal V}}
\newcommand{\cT}{{\mathcal T}}

\newcommand{\cI}{{\mathcal I}}
\newcommand{\cL}{{\mathcal L}}
\newcommand{\dL}{{\mathfrak L}}

\newcommand{\cP}{{\mathcal P}}
\newcommand{\cO}{{\mathcal O}}
\newcommand{\cW}{{\mathcal W}}

\newcommand{\dH}{{\mathfrak H}}

\newcommand{\Div}{{\mathrm {Div}}}
\newcommand{\Coind}{{\mathrm {Coind}}}

\newcommand{\rank}{{\mathrm rank}}

\newcommand{\Hom}{{\mathrm {Hom}}}
\newcommand{\Dist}{{\mathrm {Dist}}}
\newcommand{\Meas}{{\mathrm {Meas}}}

\title{Anticyclotomic $p$-adic $L$-functions and \\ the Exceptional Zero Phenomenon}
\author{Santiago Molina
}

\begin{document}
%%%%%%%%%%%%%%%%%%%%%%%%%%%%%%%%%%%%%%%%%%%%%

\address{Departament de Matem\`atica Aplicada\\
Universitat Polit\`ecnica de Catalunya}
\email{santiago.molina@upc.edu}

\maketitle
%%%%%%%%%%%%%%%%%%%%%%%%%%%%%%%%%%%%%%%%%%%%%
%\vskip 1cm

\begin{abstract} 
Let $A$ be a modular elliptic curve over a totally real field $F$, and let $K/F$ be a totally imaginary quadratic extension.
In the event of exceptional zero phenomenon, we prove a formula for the derivative of the multivariable anticyclotomic $p$-adic $L$-function attached to $(A,K)$, in terms of the Hasse-Weil $L$-function and certain $p$-adic periods attached to the respective automorphic forms. 
Our methods are based on a new construction of the anticyclotomic $p$-adic $L$-function by means of the corresponding automorphic representation.
\end{abstract}

\tableofcontents

\section{Introduction}

Let $F$ be a totally real field of degree $d$ and let $A$ be a modular elliptic curve defined over $F$ (although our results apply for general modular abelian varieties). One of the central research topics in Modern Number Theory is the relation between the arithmetic of $A$ and the analysis of the (Hasse-Weil) $L$-function $L(A,s)$ attached to $A$.  The $L$-function $L(A,s)$ and all its twists $L(A,\psi,s)$, where $\psi$ is a finite order character of the Galois group $\Gal(F^{ab}/F)$ of the maximal abelian extension $F^{ab}$ of $F$, are $\C$-valued functions that satisfy a certain symmetric functional equations relating their values at $s$ and $2-s$. The well-known \emph{Birch and Swinnerton-Dyer conjecture} predicts that the rank of the Mordel-Weil group $A(F)$ coincides with the order of vanishing of $L(A,1)$. Later generalizations by Mazur and Tate predict that the rank of the $\psi$-isotypical component $A(F)[\psi]=A(F^{ab})\otimes_{\Z[\Gal(F^{ab}/F)]}\C(\psi)$ agrees with the order of vanishing of $L(A,\psi,1)$.

If $A$ has either ordinary good or bad multiplicative reduction at all places above $p$, we obtain a better understanding of the arithmetic of $A/F$ if we replace the complex analysis of $L(A,s)$ by the a $p$-adic analysis of its $p$-adic avatar $L_p(A,s)$, the (cyclotomic) $p$-adic $L$-function of $A$. This is a $\C_p$-valued function that \emph{interpolates} the critical values $L(A,\psi,1)$, for any finite order character $\psi$ of the Galois group $\cG_p\simeq\Z_p$ of the cyclotomic $\Z_p$-extension $F_p^{cyc}$ of $F$ unramified outside $p$ and $\infty$. The function $L_p(A,s)$ is defined as
\[
L_p(A,s)=\int_{\cG_p}\exp_p(s\ell_{cyc}(\gamma))d\mu_p(\gamma),\qquad s\in\C_p,
\]
where $\ell_{cyc}:\cG_p\rightarrow\Z_p$ is a the $p$-adic logarithm of the cyclotomic character and $\mu_p$ is a certain (cyclotomic) $p$-adic measure attached to $A$. By $L_p(A,s)$ \emph{interpolates} $L(A,\psi,1)$, we mean that the measure $\mu_p$ satisfies
\[
\int_{\cG_p}\psi(\gamma)d\mu_p(\gamma)=e_p(A,\psi)L(A,\psi,1),
\]
for all finite order characters $\psi:\cG_p\rightarrow\C^\times$, where
$e_p(A,\psi)$ is the Euler factor at $p$ which is non-zero for almost all $\psi$. Observe that the $p$-adic $L$-function is univocally characterized by the $\C_p$-valued measure $\mu_p$.
A $p$-adic analog of the Birch and Swinnerton-Dyer conjecture was stated in \cite{MTT} and \cite{Dis}.

Let $K/F$ be a totally imaginary quadratic  extension.  Some remarkable achievements towards the Birch and Swinnerton-Dyer conjecture have been obtained by means of the rich theory of \emph{Heegner points} associated with $K/F$. This encourages us to consider $A$ as an elliptic curve defined over $K$.
Note that in this setting we can consider \emph{anticyclotomic} abelian extensions of $K$ which are linearly disjoint from $F_p^{cyc}K$. Indeed,
for any prime ideal $\cP$ of $F$ dividing $p$, let
$K_\cP^{a}$ be the maximal abelian extension of $K$ which is unramified outside $\cP$ and $\infty$ and so that the complex conjugation $\tau\in\Gal(K/F)$ acts on $\cG_{K,\cP}:=\Gal(K_\cP^{a}/E)$ by $-1$.
Up to torsion, $\cG_{K,\cP}$ is isomorphic to $\Z_p^r$, where $r=[F_\cP:\Q_p]$.
Motivated by the cyclotomic theory, one may ask if there is an analogous construction of $p$-adic $L$-functions attached to such anticyclotomic $\Z_p^r$-extensions.

The behavior of the local functional equation outside $\cP$ provides a dichotomy in our scenario: on the one hand the \emph{definite 
case}, where the number of finite places $v$ outside $\cP$ with the sign of the local functional equation (i.e. the local root number) $\omega_v(A/K)=-1$ is even; on the other hand the \emph{indefinite case}, where that number of finite places is odd.

Assume that $A$ has either ordinary good or multiplicative reduction modulo $\cP$.
Our starting point is the construction of $\C_p$-valued measures of $\cG_{K,\cP}$ attached to $A$, with good interpolation properties.
In the literature, we can find constructions of such measures in some particular cases: for $F=\Q$ and $K/\Q$ not ramified at $p$, we have the work of Bertolini-Darmon (see \cite{B-D1} and \cite{B-D4}); for arbitrary $F$, we have the work of Van Order also under certain restrictions for the ramification of $K/F$ (see \cite{V-Ord1} and \cite{V-Ord2}). 
In this paper, we provide alternative constructions valid for arbitrary totally imaginary quadratic extensions $K/F$. %Note the novelty of these constructions, even for $F=\Q$, when $E/\Q$ ramifies at $p$. 
We denote by $\mu_{K,\cP}^{\df}$ and $\mu_{K,\cP}^{\ind}$ the corresponding measures in the definite and indefinite situations, respectively. 
In analogy with the cyclotomic setting, our definite $p$-adic measure $\mu_{K,\cP}^{\df}$ interpolates the critical values $L(A/K,\chi,1)$ for any finite order character $\chi$ of $\cG_{K,\cP}$ (Theorem \ref{intprop1}).
In the indefinite case, the corresponding $p$-adic measure $\mu_{K,\cP}^{\ind}$ interpolates $p$-adic logarithms of certain Heegner points whose height is given by the derivative $L'(A/K,\chi,1)$ (Theorem \ref{intprop2}). In both scenarios (and also including the cyclotomic setting), the interpolation property involves an Euler factor $e_\cP(A/K,\chi)$ that is not identically zero. We remark that our construction of $\mu_{K,\cP}^{\df}$ admits generalizations to $p$-adic measures (seen as Stickelberger elements) attached to arbitrary quadratic extensions $K/F$ (see Bergunde and Gehrmann \cite{B-G} and \cite{lennart-felix}), indeed, we explain in \S \ref{globdistmeas} that the degenerate case of $K=F\times F$ corresponds to Spiess construction of the cyclotomic $p$-adic L-function \cite{Spiess}.

The set of $\C_p$-valued measures $\Meas(\cG_{K,\cP},\C_p)$ of $\cG_{K,\cP}$ has a natural structure of $\C_p$-algebra called the \emph{Iwasawa algebra}. The morphism $\deg:\Meas(\cG_{K,\cP},\C_p)\rightarrow\C_p$, that maps any measure to the image of the constant map 1, is indeed a $\C_p$-algebra morphism. The ideal $\cI=\ker(\deg)$ is called \emph{augmentation ideal} and the maximum power $\cI^r$ in which a measure $\mu$ lies is called \emph{order of vanishing} of $\mu$. Observe that, in the cyclotomic setting, $L_p(A,0)=\deg(\mu_p)$ and the order of vanishing of $L_p(A,s)$ at $s=0$ coincides with the order of vanishing of $\mu_p$.

In this paper, we mean by anticyclotomic $p$-adic $L$-functions the elements $L^{\df}_\cP(A/K)$ and $L^{\ind}_\cP(A/K)$ of the Iwasawa algebra of $\cG_{K,\cP}$ defined by $\mu_{K,\cP}^{\df}$ and $\mu_{K,\cP}^{\ind}$, respectively. In the cyclotomic and definite anticyclotomic setting, the interpolation property relates the image $\deg( L^{\df}_\cP(A/K))$
 with the critical value of the classical $L$-functions at $s=1$ (analogously, in the indefinite anticyclotomic case the classical $L$-function is replaced by its derivative). This suggests that the order of vanishing of such $p$-adic $L$-functions coincides with the order of vanishing of the classical $L$-functions at $s=1$ (resp. their derivatives in the indefinite anticyclotomic case).
Nevertheless, a surprising phenomenon appears if $A$ has split multiplicative reduction: the Euler factor $e_\cP(A/K,1)$ vanishes and we observe a zero of the $p$-adic $L$-function, even when the classical $L$-function (or its derivative) is not zero. 
These extra zeros coming from the vanishing of the Euler factors are called \emph{exceptional zeros}.

A first approach to understand this exceptional zero phenomenon is to compute the derivatives of the corresponding $p$-adic $L$-functions or, analogously in terms of Iwasawa algebras, their classes in the respective $\C_p$-vector spaces $\cI^r/\cI^{r+1}$. Many authors have contributed to this research line: 
\begin{itemize}
\item In the cyclotomic setting, the order of vanishing of $L_p(A,s)$ at $s=0$ is at least $m$, the number of places above $p$ where $A$ has split multiplicative reduction \cite{Spiess}. Moreover, a formula that expresses $\frac{d^m}{ds^m}L_p(A,s)\mid_{s=0}$ as a product of $L(A,1)$ with what are known as (geometric) $\cL$-invariants $\cL_\cP(A)$ ($\cP\mid p$ is a prime of split multiplicative reduction) was conjectured by Hida. This formula was established by Greenberg and Stevens in \cite{G-S} for $F=\Q$, by Mok in \cite{Mok} for arbitrary totally real fields and $m=1$, and finally by Spie{\ss} in \cite{Spiess} for arbitrary $m$ and $F$ under some mild assumptions. In fact, Spie{\ss} proves an analogous formula with the $\cL$-invariants $\cL_\cP(A)$ replaced by certain automorphic $\cL$-invariants $\cL_\cP(\pi)$, where $\pi$ is the automorphic cuspform attached to $A$.

\item In the definite anticyclotomic setting, there have only been results for $F=\Q$, prior to the writing of this paper. Assuming that the quadratic extension $K/\Q$ splits at $p$, Bertolini and Darmon proved in \cite{B-D1} a similar formula for the image of $L^{\df}_p(A/K)$ in $\cI/\cI^2\simeq\C_p$ in terms of $L(A/K,1)$ and the same $\cL$-invariant $\cL_p(A)$ that appeared in the cyclotomic setting. If $p$ is inert in $K$, Bertolini and Darmon proved in \cite{B-D3} an analogous formula involving the $p$-adic logarithm of a Heegner point (see also \cite{B-D4}). The case $K/\Q$ ramified at $p$ is still an open problem even for $F=\Q$. During the development of this paper, L. Gehrmann and F. Bergunde in \cite{B-G} and \cite{lennart-felix} have obtained similar results in the definite setting for general $F$ and arbitrary quadratic extension $K/F$ using Stickelberger elements. See also \cite{Hung1} for a similar work in this setting.

\item In the indefinite anticyclotomic setting, there are only results for the case $F=\Q$. Bertolini and Darmon proved in \cite{B-D2} that the image of $L^{\ind}_p(A/K)$ in $\cI/\cI^2\simeq\C_p$ can be expressed in terms of the critical value $L(A/K,1)$ (see also \cite{B-D4}). If $K/\Q$ splits, Castell\`a recently proved in \cite{Cas} a formula for the derivative of $L^{\ind}_p(A/K)$ in terms of the logarithm of a Heegner point and the $\cL$-invariant $\cL_p(A)$.
The case $K/\Q$ ramified at $p$ has never been considered. 
\end{itemize} 

In order to explain our main result, let us introduce $\pi$, the automorphic representation of $\GL_2(\A_F)$ associated with $A$. The (twisted) Hasse-Weil $L$-function $L(A/K,\chi,s)$ coincides with the Rankin-Selberg $L$-function $L(s-1/2,\pi_K,\chi)$, where $\pi_K$ is the extension of the representation $\pi$ to $\GL_2(\A_K)$. Actually, we will define our anticyclotomic measures by means of a certain Jacquet-Langlands lift $\pi^{JL}$ of $\pi$. For this reason we denote them by $L_\cP^{\df}(\pi_K)$ and $L_\cP^{\ind}(\pi_K)$ when considered as elements of the Iwasawa algebra.
The main result of this paper establishes a formula for the classes of the $p$-adic $L$-functions in $\cI/\cI^2$, when $\cP$ splits in $K$:
\begin{theorem}
Suppose that $\cP$ is a Steinberg prime of $\pi$, then both $L_\cP^{\df}(\pi_K),L_\cP^{\ind}(\pi_K)\in \cI$.
Let us denote by $\nabla L_\cP^{\df}(\pi_K)$ and $\nabla L_\cP^{\ind}(\pi_K)$ their classes in $\cI/\cI^2$. If we also assume that $\cP$ splits in $K$, then we have that
\begin{eqnarray*}
\nabla L_\cP^{\df}(\pi_K)&=&\underline{\cL_\cP^{\df}}(\pi)\left( C_K C(\pi_\cP)\frac{L(1/2,\pi_K,1)}{L(1,\pi,ad)}\right)^{1/2},\\
\nabla L_\cP^{\ind}(\pi_K)&=&\underline{\cL_\cP^\ind}(\pi)\log_p(P_T),
\end{eqnarray*}
where $C_K$ is a non-zero constant, $C(\pi_\cP)$ is a non-zero Euler factor, $P_T$ is a Heegner point with explicit height depending on $L'(1/2,\pi_K,1)$, and $\underline{\cL_\cP^\df}(\pi),\underline{\cL_\cP^{\ind}}(\pi)\in\cI/\cI^2$ are automorphic $\cL$-invariants defined in terms of the cohomology of ($S_p$-)arithmetic groups associated with Jacquet-Langlands lifts of $\pi$.
\end{theorem}
In Theorem \ref{main-Res} one can found a more precise statement of this result where we compute the Euler factor $C(\pi_\cP)$ and the height of $P_T$. The definition of $\underline{\cL_\cP^\df}(\pi)$ and $\underline{\cL_\cP^{\ind}}(\pi)\in\cI/\cI^2$ can be found in \S\ref{Ap2} and \S \ref{CaseSplit}. Notice that the condition of  $\cP$ being a Steinberg prime of $\pi$ is equivalent to $A$ having split multiplicative reduction at $\cP$. 
 
Hence, our formulas generalize the results for $F=\Q$ by Bertolini and Darmon \cite{B-D1} and by Castell\`a \cite{Cas} in the definite and indefinite settings, respectively. We remark that our proof in the definite case is an adaption of Spie{\ss}' proof of the exceptional zero conjecture for Hilbert modular forms \cite{Spiess}. Nevertheless, the results in the indefinite setting are much more exciting. The indefinite formula was conjectured by Bertolini and Darmon in \cite[Conjecture 4.6]{BDMT} for $F = \Q$ and $A(K)$ of rank 1, with the spirit of the $p$-adic BSD conjecture introduced in \cite{MTT}, but replacing the automorphic $\cL$-invariant vector $\underline{\cL_\cP^{\ind}}(\pi)$ by the corresponding geometric $\cL$-invariant. 
Moreover, the techniques used to prove the indefinite formula differ completely from Castell\`a's work, who exploited the theory of big Heegner points and an adaptation of Greenberg-Stevens proof of the cyclotomic exceptional zero conjecture for $F=\Q$ (see \cite{G-S}). We have decided to expose both the definite and indefinite settings together in this paper in order to remark the analogies between both constructions of the $p$-adic measures and both proofs of the exceptional zero conjecture. 

Bearing in mind the case $F=\Q$, we expect the $\cL$-invariant vectors $\underline{\cL_\cP^{\df}}(\pi)$ and $\underline{\cL_\cP^{\ind}}(\pi)$ to be equal and related to the geometry of $A/F_\cP$. 
 In \S \ref{GeoLinv}, we show that $\underline{\cL_\cP^{\df}}(\pi)$ is given by the geometric $\cL$-invariants. 
 It is a work in progress to prove that $\underline{\cL_\cP^{\ind}}(\pi)=\underline{\cL_\cP^{\df}}(\pi)$. Such equality, predicted by Conjecture \ref{conjLinv}, holds for $F=\Q$ thanks to the work by Greenberg-Stevens \cite{G-S} and Longo-Rotger-Vigni \cite{L-R-V}, and holds for general $F$, under certain technical conditions, thanks to the work of L. Gehrmann \cite{lennart-L-invariants}.

Finally we would like to remark that, similarly as in the definite setting one can generalize the costruction of the $p$-adic measure and prove a similar exceptional zero result when the quadratic extension $K/F$ is arbitrary (see \cite{B-G} and \cite{lennart-felix}), we strongly believe that one can also define an analogous $p$-adic distribution in the indefinite setting with values in Stark-Heegner points. Moreover, the techniques developed in \cite{GMM} would provide a analogous exceptional zero result.

 \subsection*{Outline of the proof} The fact that $A$ has good ordinary or multiplicative reduction modulo $\cP$ implies that $\pi_\cP$ is the irreducible quotient of the representation induced by the character of the Borel subgroup given by some $\alpha\in \bar\Q\cap\cO_{\C_p}^\times$. In \S \ref{locthe} we define a $\bar\Q$-representation, with representation space $V_\alpha^{\bar\Q}$, whose scalar extension by $\C$ is isomorphic to $\pi_\cP$, and a $K_\cP^\times$-equivariant morphism $\delta_T:C_c(K_\cP^\times/F_\cP^\times,\bar\Q)\rightarrow V_{\alpha}^{\bar\Q}$. Such a morphism $\delta_T$ realizes $\Hom(V_\alpha^{\bar\Q},\bar\Q)$ as a space of distributions of $K_\cP^\times/F_\cP^\times$. 
 
 In the definite setting, we consider $G_D$ the multiplicative group modulo the center of the totally definite quaternion algebra $D$ that splits at $\cP$ and such that $\Hom_{K_v^\times}(\pi_{v}^{D},\C)\neq 0$, for any place $v\neq\cP$, where $\pi^{D}$ is the Jacquet-Langlands lift of $\pi$ to $G_D$. It is easy to prove that a well chosen generator of $\pi^{D}$ provides an element 
 \[
 \phi\in H^0(G_D(F),\cA_D^\cP(V_\alpha^{\bar\Q},\bar\Q)^{U^\cP}),\qquad \cA_D^\cP(V_\alpha^{\bar\Q},\bar\Q)=C(G_D(\hat F^\cP),\Hom(V_\alpha^{\bar\Q},\bar\Q))
 \]
 where $U^\cP\subseteq G_D(\hat F^\cP)$ is certain open compact subgroup. By Class Field Theory, any locally constant function in $g\in C(\cG_{K,\cP},\bar\Q)$ can be seen as an element $\partial g\in H_0(K^\times,C_c(\hat K^\times/\hat F^\times,\bar\Q))$.
 In \S \ref{globdistmeas}, we define the distribution $\mu_{K,\cP}^{\df}$ by
 \[
 \int_{\cG_{K,\cP}}g(\gamma)d\mu_{K,\cP}^{\df}(\gamma):={\rm res}(\phi)\cap\partial g,
 \]
 where the cap-product corresponds to the natural pairing between $C_c(\hat K^\times/\hat F^\times,\bar\Q)$ and $\cA_D^\cP(V_\alpha^{\bar\Q},\bar\Q)^{U^\cP}$, once considered $\Hom(V_\alpha^{\bar\Q},\bar\Q)$ as a space of distributions.
 
 In the indefinite setting, we choose our quaternion algebra $B$ in such a way that $B$ splits at a single archimedean place and $\Hom_{K_v^\times}(\pi_v^{B},\C)\neq 0$, for any place $v\neq\cP\infty$. In \S \ref{ShimCM}, we explain how to obtain an element $\Phi_T\in H^0(K^\times,\cA_B^\cP(V_\alpha^{\bar\Q},A(K^{ab})\otimes\bar\Q)^{U^\cP})$ from the image through the Abel-Jacobi map of divisors supported on Heegner points of the Shimura curve attached to $G_B$ and some $U\subset G_B(\hat F)$. Applying the same procedure as before, we can define a distribution with values in $A(K^{ab})\otimes\bar\Q$. Since we want a distribution with values in $\C_p$, we compose $\Phi_T$ with the formal group logarithm of $A$ obtaining $\log\phi\in H^0(K^\times,\cA_B^\cP(V_\alpha^{\bar\Q},\C_p)^{U^\cP})$. Similarly as before,
 \[
 \int_{\cG_{K,\cP}}g(\gamma)d\mu_{K,\cP}^{\ind}(\gamma):=\log\phi\cap\partial g.
 \]
 
 The interpolation properties of $\mu_{K,\cP}^{\df}$ and $\mu_{K,\cP}^{\ind}$ are provided by Waldspurger and Gross-Zagier-Zhang formulas, respectively. The explicit computations of the Euler factors are given in \S \ref{locthe}. Analyzing the behavior of such Euler factors when $A$ has split multiplicative reduction, we observe the exceptional zero phenomenon in both definite and indefinite situations. 
 
 In \S \ref{padicmeas}, we prove that $\mu_{K,\cP}^{\df}$ and $\mu_{K,\cP}^{\ind}$ are in fact $p$-adic measures. Thus, by Corollary \ref{charI/I2}, their images $\nabla L_\cP^\df(\pi_K)$ and $\nabla L_\cP^{\ind}(\pi_K)$ in $\cI/\cI^2$ are characterized by the integrals $\int_{\cG_{K,\cP}}\ell(\gamma)d\mu_{K,\cP}^{\bullet}(\gamma)$ ($\bullet=\df,\ind$), for any $\ell\in \Hom_{\Z_p}(\cG_{K,\cP},\Z_p)$. Note that if $F=\Q$ the $\Z_p$-rank of $\cG_{K,\cP}$ is 1, and it is enough to consider $\ell=\log_p$ given by Iwasawa $p$-adic logarithm. In general such $\Z_p$-rank is $[F_\cP:\Q_p]$, hence there may be different $\ell$ to consider. 
 
 In \S \ref{CaseSplit}, we show that $\partial(\ell)=\vartheta\cap z_{\ell_\cP}$, for some fixed $\vartheta\in H_1(K^\times,C_c((\hat K^\cP)^\times/ (\hat F^\cP)^\times,\Z))$ and certain $ z_{\ell_\cP}\in H^1(K^\times,C_c(K_\cP^\times/F_\cP^\times,\Z_p))$ that only depends on $\ell_\cP$, the restriction of $\ell$ on $K_\cP^\times/F_\cP^\times$ via the Artin map.
 By the cohomological interpretation of the measures $\mu_{K,\cP}^{\bullet}$,
 \[
 \int_{\cG_{K,\cP}}\ell(\gamma)d\mu_{K,\cP}^{\bullet}(\gamma)=\kappa_\bullet\cap\partial(\ell)=(\kappa_\bullet\cap z_{\ell_\cP})\cap\vartheta,\qquad \bullet=\df,\ind,
 \]
 where $\kappa_{\df}={\rm res}(\phi)$ and $\kappa_{\ind}=\log\phi$. We prove in \S \ref{Ap2} that there exists $\cL(\pi^{JL},\ell_\cP)\in \C_p$ such that $\kappa_\bullet\cap z_{\ell_\cP}=\cL(\pi^{JL},\ell_\cP)(\kappa_\bullet\cap z_{\ord_\cP})$, for some fixed $ z_{\ord_\cP}\in H^1(K^\times,C_c(K_\cP^\times/F_\cP^\times,\Z))$. Hence, if we define the automorphic $\cL$-invariant $\underline{\cL_\cP^{\bullet}}(\pi)$ to be the element in $\cI/\cI^2$ that maps any $\ell$ to $\cL(\pi^{JL},\ell_\cP)$, then it is clear that 
 \[
 \nabla L_\cP^\bullet(\pi_K)=\underline{\cL_\cP^\bullet}(\pi)(\kappa_\bullet\cap z_{\ord_\cP})\cap\vartheta.
 \]
 It is a tedious but straightforward computation to show that
 \[
 (\kappa_\bullet\cap z_{\ord_\cP})\cap\vartheta=\left\{\begin{array}{lc}\left( C_K C(\pi_\cP)\frac{L(1/2,\pi_K,1)}{L(1,\pi,ad)}\right)^{1/2},&\bullet=\df,\\
 \log_p(P_T),&\bullet=\ind,\end{array}\right.
 \]
 for some Heegner point $P_T$ with Neron-Tate height $|P_T|^2=C_K C(\pi_\cP)\frac{L'(1/2,\pi_K,1)}{L(1,\pi,ad)}$.

 %The techniques used to prove these results are generalizations of the work by Spie{\ss} in \cite{Spiess}, where the cyclotomic measure is defined by means of a certain group cohomology cocycle. In our setting, we construct our anticyclotomic measures by means of certain ($S_\cP$)-arithmetic cocycles associated with the automorphic representations $\pi^{JL}$ and the torus attached to the extension $E/F$. In fact, our construction is given in such generality that we recover the Spie{\ss} $p$-adic measure if we allow $E/F$ to be the trivial extension $E=F\times F$. In this direction, our work helps us to understand the link between the cyclotomic and anticyclotomic framework.

%Finally, let us remark that the techniques and results described in this paper could also contribute to research on the exceptional zero phenomenon whenever $\cP$ does not split in $E$. A future research line is to compute the classes $\nabla L_\cP^I(\pi,E)$ and $\nabla L_\cP^{II}(\pi,E)$ in case $\cP$ is inert in $E$, generalizing the results of Bertolini and Darmon for $F=\Q$, or $\cP$ ramified in $E$, which is an open problem even for $F=\Q$.

\subsection{Notation}

For any field $L$, write $\cO_L$ for its integer ring.

We fix embeddings $\bar\Q\hookrightarrow\C$ and $\bar\Q\hookrightarrow\C_p$, and write $\bar\cO=\cO_{\C_p}\cap\bar\Q$.

Throughout this paper $F$ will be a totally real number field.
Let $\A_F$ and $\hat F$ be its rings of adeles and finite adeles, respectively. For any set of non-archimedean primes $S$, write $\hat F^S=\hat F\cap \prod_{v\not\in S}F_v$, where $F_v$ is the localization of $F$ at a finite place $v$. We write $\hat F^\cP$ instead of $\hat F^{\{\cP\}}$. We denote by $d^\times x_v$ the Haar measure normalized as in \cite{Zhang}. The product of $d^\times x_v$ defines a Haar measure on $\A_F^\times/F^\times$. 

Throughout this paper $K/F$ will be an imaginary quadratic extension. We will denote by $T$ the algebraic group given by the quotient of $K^\times$ and $F^\times$, namely, $T(R)=(K\otimes_F R)^\times/R^\times$.
If $v$ is a finite place of $F$, we denote by $d^\times t_v$ the Haar measure on $T(F_v)$ given by the quotient measure.
The product of $d^\times t_v$ defines a Haar measure $d^\times t$ on $T(\hat F)$.

For any quaternion algebra $B/F$, write $G_B$ for the multiplicative group modulo its center, namely, $G_B(R)=(B\otimes_F R)^\times/R^\times$,  for any $F$-algebra $R$.

%We denote by $\ast$ the usual action of $\GL_2(F_v)$ on the projective line $\PP^1(F_v)$, for any place $v$ of $F$. % by:
%\[
%\GL_2(F)\times\PP^1(F)\longrightarrow\PP^1(F),\quad
%\left(\left(\begin{array}{cc}
%a&b\\
%c&d
%\end{array}
%\right),\tau\right)\longmapsto g\ast\tau:=\frac{a\tau+b}{c\tau+d}.
%\]  

Given two topological spaces $\cX$ and $\cY$, we denote by $C_c(\cX,\cY)$ the space of continuous and compactly supported functions between $\cX$ and $\cY$.
We denote it by $C_c(\cX,\cY)_0$ when $\cY$ is considered with the discrete topology. If $\cY$ is a normed space write $C_\diamond(\cX,\cY)$ for the completion of $C_c(\cX,\cY)$ with the supremum norm $\|\cdot\|_\infty$.

\subsubsection*{Acknowledgements.}
The author would like to thank Lennart Gehrmann and Michael Spie{\ss} for their comments and discussions
 throughout the development of this paper. The author also thanks the referee for carefully reading the paper and providing many helpful
comments.

The author is supported in part by DGICYT Grant MTM2015-63829-P.
This project has received funding from the European Research Council
(ERC) under the European Union's Horizon 2020 research and innovation
programme (grant agreement No. 682152).

\section{Measures and Iwasawa algebras}

\subsection{Distributions and measures}

Let $\cG$ be a totally disconnected locally compact topological group, and let $R$ be a topological Hausdorff ring.
 For any $R$-module $M$,  an $M$-valued distribution on $\cG$ is a homomorphism $\mu:C_c(\cG,\Z)_0\rightarrow M$. Hence, it extends to a $R$-linear map 
\[
C_c(\cG,R)_0\longrightarrow M, \quad f\longmapsto\int_\cG f(\gamma)d\mu(\gamma).
\]
We shall denote the $R$-module of $M$-valued distributions by $\Dist(\cG,M)$.

Let $(V,\|\;\|)$ be a Banach space over a $p$-adic field $R=L$. We say that $\mu\in\Dist(\cG,V)$ is a measure if it is continuous with respect to the supremum norm, hence it extends to a continuous functional
\[
C_\diamond(\cG,L)\longrightarrow V;\quad f\longmapsto\int_\cG fd\mu.
\]
We will denote by $\Meas(\cG,V)$ the space of $V$-valued measures on $\cG$.

An $\cO_L$-submodule $M$ of a $L$-vector space $V$ is a \emph{lattice} if $\cup_{a\in L^\times}aM=V$ and $\cap_{a\in L^\times}aM=\{0\}$. For a given lattice $M\subseteq V$ the function $p_M(v):=\inf_{v\in aM}|a|$ is a norm on $V$. Any other norm $\|\;\|$ on $V$ is equivalent to $p_M$ if and only if $M$ is open and bounded in $(V,\|\;\|)$. 
%A lattice $L\subseteq V$ is \emph{complete} if $V$ is complete with respect to $p_L$, namely, $(V,p_L)$ is a Banach space.
For any $M$ open and bounded lattice on a Banach space $(V,\|\;\|)$, the space $\Meas(\cG,V)$ is the image of the canonical inclusion 
\begin{equation}\label{defMeas}
\Dist(\cG,M)\otimes_{\cO_L}L\longrightarrow\Dist(\cG,V).
\end{equation}

\subsection{Iwasawa algebras}
 
Let $\cG$ now be a commutative pro-$p$ group. Since $\cO_L$ is a lattice in the Banach space $(L,|\;|)$, the $L$-vector space $\Meas(\cG,L)$ coincides with the tensor product $\Dist(\cG,\cO_L)\otimes_{\cO_L}L$.

The usual convolution product 
%\[
%\ast:\Dist(\cG,\cO_L)\times\Dist(\cG,\cO_L)\longrightarrow\Dist(\cG,\cO_L)\qquad (\mu_1,\mu_2)\longmapsto \mu_1\ast\mu_2,
%\]
%where
%\[
%\int_\cG f(\gamma)d(\mu_1\ast\mu_2)(\gamma):=\int_\cG\int_\cG f(\alpha\cdot\beta)d\mu_1(\alpha)d\mu_2(\beta).
%\]
 endows $\Dist(\cG,\cO_L)$ with structure of $\cO_L$-algebra. %, where the unit is given by the Dirac measure $\int_\cG fd1:=f(1)$ and the structural ring homomorphism corresponds to
%\[
%s:\cO_K\longrightarrow\Dist(\cG,\cO_K),\quad \alpha\longmapsto \alpha d1.
%\] 
It is isomorphic to $\cO_L[[\cG]]=\varprojlim_\cH \cO_L[\cG/\cH]$,
where $\cH$ runs over open compact subgroups of $\cG$. Hence, we deduce that $\Meas(\cG,L)$ is equipped with the $L$-algebra structure $\Lambda_L:=\cO_L[[\cG]]\otimes_{\cO_L}L$, called the \emph{Iwasawa algebra} of $\cG$ with coefficients in $L$. 

Observe that the natural map 
\[
d:\cG\longrightarrow\Lambda_L^\ast;\qquad \int_\cG fdg:=f(g),
\]
is a group homomorphism and the $L$-module homomorphism 
\[
\deg:\Lambda_L\longrightarrow L;\qquad\mu\longmapsto\int_\cG d\mu,
\]%\todo{Unificar la definicion de $\deg$ (para $K$ y $\cO_K$)}
is indeed an $L$-algebra homomorphism.

%Let us give a different description of the $\cO_K$-algebra $\cO_K[[\cG]]$. We claim that the Iwasawa algebra is the inverse limit of group algebras
%\[
%\cO_K[[\cG]]=\varprojlim_\cH \cO_K[\cG/\cH],
%\]
%where $\cH$ runs over open compact subgroups of $\cG$. Indeed, given $\mu\in \Dist(\cG,\cO_K)$ we define the element in $\varprojlim_\cH \cO_K[\cG/\cH]$ provided by the expression 
%\[
%\sum_{g\cH\in\cG/\cH}\left(\int_\cG 1_{g\cH}d\mu\right)g\cH\in \cO_K[\cG/\cH]
%\]
%where $1_{g\cH}$ is the characteristic function of $g\cH$. We leave to the reader the verification that this defines an isomorphism of $\cO_K$-algebras. It is easy to check that $\deg$ corresponds to the natural degree homomorphism.

Let $\cI:=\ker(\deg)$ be the augmentation ideal. Since $\deg$ is surjective, $\cI/\cI^2$ has a natural structure of $\Lambda_L/\cI\simeq L$-vector space. Moreover, since $\cG$ is a pro-$p$ group, it is a $\Z_p$-module and the map
\[
\varphi:\cG\longrightarrow\cI/\cI^2;\qquad g\longmapsto dg-d1,
\]
 is a $\Z_p$-module homomorphism. %Indeed,
%\[
%\varphi(g_1+g_2)=d(g_1+g_2)-d1=\varphi(g_1)+\varphi(g_2)+(dg_1-d1)\ast (dg_2-d1).
%\]
The following result describes $\cI/\cI^2$ as the tensor product
$\cG\otimes_{\Z_p}L$.
\begin{proposition}\label{HurThm}
Assume that $\cG$ is a $\Z_p$-module of finite rank,
then the map $\varphi$ defines an isomorphism of $L$-modules
\[
\varphi:\cG\otimes_{\Z_p}L\longrightarrow\cI/\cI^2.
\]
\end{proposition}
\begin{proof}
Let us consider the dual group $(\cG\otimes_{\Z_p}L)^\vee:=\Hom_{\Z_p}(\cG,L)$ and the canonical $L$-module morphism $\iota:\cG\otimes_{\Z_p}L\rightarrow ((\cG\otimes_{\Z_p}L)^\vee)^\vee$. We define
\[
\psi:\cI\longrightarrow ((\cG\otimes_{\Z_p}L)^\vee)^\vee;\qquad \mu\longmapsto \left(\ell\mapsto\int_\cG\ell d\mu\right),
\]
where $\ell\in\Hom_{\Z_p}(\cG,L)$ is seen as a continuous function in $C(\cG,L)$. We check that $\psi(\cI^2)=0$; Indeed, if $\mu_1,\mu_2\in\cI$,
\[
\int_{\cG}\ell d(\mu_1\ast\mu_2)=\int_\cG\int_\cG \ell(\alpha+\beta)d\mu_1(\alpha)d\mu_2(\beta)=\deg(\mu_2)\int_\cG\ell d\mu_1+\deg(\mu_1)\int_\cG\ell d\mu_2=0.
\]
Thus we have a $L$-module morphism $\psi:\cI/\cI^2\rightarrow ((\cG\otimes_{\Z_p}L)^\vee)^\vee$ satisfying $\psi\circ\varphi=\iota$.

Since $\cG$ is a $\Z_p$-module of finite rank, $\iota$ is an isomorphism. Thus, in order to prove our result, it is enough to show that $\varphi$ is surjective.

Given the description of $\cO_L[[\cG]]$ as an inverse limit of group algebras, we observe that at each finite level
\[
\cI_\cH:=\ker\left(\deg:\cO_L[\cG/\cH]\rightarrow\cO_L\right)=\left\{\sum_{g\cH\in\cG/\cH}\alpha_{g\cH}(g\cH-\cH),\;\;\alpha_{g\cH}\in\cO_L\right\}.
\]
This implies that the corresponding morphism $\varphi_\cH:\cG/\cH\otimes_{\Z_p}\cO_L\rightarrow\cI_\cH/\cI_\cH^2$ is surjective. We conclude that $\varphi$ is surjective an the result follows.
\end{proof}

\begin{corollary}\label{charI/I2}
Assume that $\cG$ is a (free) $\Z_p$-module of finite rank and let $\cG^\vee:=\Hom_{\Z_p}(\cG,\Z_p)$. Then the map
\[
\psi:\cI/\cI^2\longrightarrow \Hom_{\Z_p}(\cG^\vee,L),\qquad\mu\longmapsto \left(\ell\mapsto\int_\cG\ell d\mu\right),
\]
is a $L$-module isomorphism.
\end{corollary}
\begin{proof}
Follows directly from the proof of Proposition \ref{HurThm}.
\end{proof}

\section{Local theory}\label{locthe}

\subsection{Universal unramified principal series}\label{Loc-Re}

Let $\cP\subset \cO_F$ be a prime ideal with uniformizer $\varpi$, residual characteristic $q$, and valuation $\nu$ ($\nu(\varpi)=1$). %Let $G=A^\times$, where $A$ is a quaternion algebra over $F$, let $Z\simeq F^\times$ be its center, and let $U\subseteq G$ be a compact open subgroup.

%Let $G=\GL_2(F)$, and
%let us consider the following subgroups of $G$
%\[
%K=\GL_2(\cO), \;\; B=\left\{\left(\begin{array}{cc}
%					  u_1 & x\\
%					  & u_2
%					\end{array}\right)\in G\right\},\;\; Z=\left\{\left(\begin{array}{cc}
%					  z & \\
%					  & z
%					\end{array}\right)\in G\right\}.
%\]
%Hence $Z$ is the center of $G$, and $K$ is a maximal compact subgroup. For any ideal $\mathfrak{c}\subset \cO$ write $K_0(\mathfrak{c})\subset K$ for the subgroup of triangular matrices modulo $\mathfrak{c}$.

Let $R$ be a topological Hausdorff ring and assume $\alpha\in R^\times$. Let us consider the unramified character $\mu_\alpha:F_\cP^\times\rightarrow R^\times$, 
$\mu_\alpha(x)=\alpha^{\nu(x)}$. This provides the induced $R$-representation
\begin{equation}\label{indalpha}
\bar V_\alpha^R=\left\{\phi:\GL_2(F_\cP)\rightarrow R\mbox{ continuous }:\;\phi\left(\left(\begin{array}{cc}
                                                               t_1&x\\&t_2
                                                              \end{array}
\right)g\right)=\mu_\alpha(t_2/t_1)\phi(g)\right\}.
\end{equation}
If $\alpha=\pm1$, the representation $\bar V_\alpha^R$ is reducible, since there is an invariant subspace of rank 1 over $R$ generated by $\phi_0(g)=\alpha^{\nu(\det g)}$. 
We denote by $(\pi_\alpha^R,V_\alpha^R)$ the quotient representation of $\bar V_\alpha^R$ by this rank 1 subrepresentation.
In the case $\alpha\neq\pm1$, we write $(\pi_\alpha^R,V_\alpha^R)$ for the induced representation $V_\alpha^R=\bar V_\alpha^R$. 
We denote by $\bar V_{\alpha, 0}^R$, $V_{\alpha, 0}^R$ and $\pi_{\alpha, 0}^R$ the corresponding representations when $R$ considered with the discrete topology. We call $V_{1,0}^\C$ \emph{the Steinberg representation} (the subindex $0$ will be dropped soon because $\C$ will be always considered with the trivial topology).

For any ideal $\mathfrak{c}\subset \cO_{F_\cP}$ write $K_0(\mathfrak{c})\subset \GL_2(\cO_{F_\cP})$ for the subgroup of triangular matrices modulo $\mathfrak{c}$, and let $Z$ be the centre of $\GL_2(F_\cP)$.
By Iwasawa decomposition, the constant map $\phi_0(k)=1$, for all $k\in K_0(1)$, defines an element of $\bar V_{\alpha,0}^R$ fixed by $K_0(1)$ and $Z$. Moreover, if $\alpha\neq\pm1$, the module $(V_\alpha^R)^{K_0(1)}$ is the free $R$-module $R\phi_0$. Similarly, if $\alpha=\pm 1$ the class of the function  
$\phi_1(k)=1_{K_0(\cP)}(k)$, for all $k\in K_0(1)$, generates $(V_\alpha^R)^{K_0(\cP)}$ freely. In both cases, $\rank_R(V_\alpha^R)^{U}=1$, where $U=K_0(1)$ or $K_0(\cP)$.

Let us consider 
\[
\Ind_{UZ}^{\GL_2}(1_R)=\{\phi\in C(\GL_2(F_\cP),R):\;\phi(UZg)=\phi(g), \mbox{\small compactly supported mod }UZ\},
\] 
and the Hecke algebra $\cH_U^R:=\Ind_{UZ}^{\GL_2}(1_R)^U$. Note that, by Frobenius reciprocity,
\[
\cH_U^R=\Hom_{UZ}(1_R,\Ind_{UZ}^{\GL_2}(1_R))=\End_{\GL_2}(\Ind_{UZ}^{\GL_2}(1_R)).
\]

Let $g_{\varpi}=\left(\begin{array}{cc}\varpi&\\&1\end{array}\right)$. We consider the element of the Hecke algebra 
$\cT_\cP\in\cH_K^R=(\Ind_{UZ}^{{\GL_2}}1_R)^{U}$ attached to 
$1_{UZg_{\omega}U}$, for $U=K_0(1)$ or $K_0(\varpi)$.

Again by Frobenius reciprocity, elements $\phi_i\in (V_\alpha^R)^U$ provide ${\GL_2(F_\cP)}$-module homomorphisms
\[
\varphi_{i}:\Ind^{\GL_2}_{ZU}1_R/(\cT_\cP-a_\cP)\Ind^{\GL_2}_{ZU}1_R\longrightarrow V_\alpha^R,
\]
where $a_\cP=\alpha+q\alpha^{-1}$ if $i=0$, and $a_\cP=\alpha=\pm 1$ if $i=1$. 

\begin{lemma}{\cite[Theorem 20]{B-L}}\label{lemlattices}
Assume $R$ is a ring endowed with the discrete topology. 
Then $\varphi_0$ is injective and $\varphi_1$ is an isomorphism. Moreover, if $R$ is a field then $\varphi_0$ is also bijective.
\end{lemma}

\begin{remark}
We emphasize that $\varphi_0$ is not an isomorphism in general. If $R$ is a domain and $L$ its fraction field, we notice that 
\[
V_\alpha^R\otimes_R L=V_\alpha^L\simeq \Ind^{\GL_2}_{ZU}1_L/(\cT_\cP-a_\cP)\Ind^{\GL_2}_{ZU}1_L=\Ind^{\GL_2}_{ZU}1_R/(\cT_\cP-a_\cP)\Ind^{\GL_2}_{ZU}1_R\otimes_RL,
\]
where $U=K_0(1)$.
Hence, we have two distinguished but generally distinct ${\GL_2(F_\cP)}$-stable $R$-modules in $V_\alpha^L$, namely $\Lambda=V_\alpha^R$ and $\Lambda'=\Ind^{\GL_2}_{ZU}1_R/(\cT_\cP-a_\cP)\Ind^{\GL_2}_{ZU}1_R$, satisfying $\Lambda'\subseteq \Lambda$.
\end{remark}

\subsection{Local distributions attached to a torus in $\GL_2$}\label{LocDisTor}

Fix $K/F$ an imaginary quadratic extension, and assume that we have a fixed embedding $K_\cP\hookrightarrow\GL_2(F_\cP)$ such that its image does not lie in the subgroup $P_\cP\subset\GL_2(F_\cP)$ of upper triangular matrices. 
%Let $T^2$ be a 2-dimensional torus in $G=\GL_2(F)$.
%Then the embedding $T^2\hookrightarrow G$ gives rise to the composition 
%\[
%\varphi_T:T:=T^2/Z\longrightarrow B\backslash G\stackrel{\varphi}{\longrightarrow}\PP^1(F)
%\]
%Note that $T^2$ acts on $F^2$ by means of the embedding $T^2\hookrightarrow G$. 
%We say that $v\in F^2$ is an \emph{eigenvector of $T^2$} if, for all $t\in T^2$, we have $tv=k_tv$, for some $k_t\in F$.
%\begin{definition}\label{deftor}
%We say that $T^2$ is \emph{split} if there is an element in $F^2$ that is an eigenvector for the action of all elements of $T^2$. We say that $T^2$ is \emph{non-split} otherwise.
%\end{definition}

%\begin{lemma}
%Assume that $T^2\not\subset B$, then we have that $\varphi_T$ is injective. Moreover, we have the following possibilities:
%\[
%\varphi_T(T)=\left\{\begin{array}{ll}
%\PP^1,& T^2\; \mbox{non-split}\\
%\PP^1\setminus\{P_1,P_2\},& T^2\; \mbox{split}
%\end{array}\right.
%\]
%for some points $P_1,P_2\in\PP^1$.
%\end{lemma}

Let $R$ be a topological Hausdorff ring, and let us consider the action of $T(F_\cP)$ on $C(T(F_\cP),R)$, given by $t\ast f(x):=f(t^{-1}x)$.
Since $P_\cP\cap K_\cP^\times=Z$, the natural map $T(F_\cP)\rightarrow \GL_2(F_\cP)/Z\rightarrow P_\cP\backslash \GL_2(F_\cP)\simeq\PP^1(F_\cP)$ is injective. 
By abuse of notation, we denote by $\mu_\alpha:P_\cP\rightarrow R^\times$ the continuous central character associated with $\mu_\alpha$ as in \eqref{indalpha}. We can construct the following well defined $T(F_\cP)$-equivariant morphism
\[
\bar\delta_T^\alpha: C_\diamond(T(F_\cP),R)\longrightarrow\bar V_\alpha^R;\quad \bar\delta_T^\alpha(f)(g)=\left\{
\begin{array}{ll}
\mu_\alpha(b)f(t^{-1}),&g=bt\in P_\cP K_\cP^\times\\
0,&g\not\in P_\cP K_\cP^\times.
\end{array}\right.
\]
%which is well-defined, since $B\cap T=Z$. 
The composition of $\bar\delta_T^\alpha$ with the natural projection gives rise to a morphism
\begin{equation}\label{deltaT}
\delta_T: C_\diamond(T(F_\cP),R)\longrightarrow(\pi_\alpha^R,V_\alpha^R).
\end{equation}
%Since both $\Ind_B^G(\mu)$ and $\pi_\alpha^R$ have trivial central character, restriction to $T^2$ provides a natural action of the group $T=T^2/Z$. It is clear that both morphisms $\bar\delta^\mu_T$ and $\delta_T$ are $T$-equivariant.

\begin{remark}
Note that if $K_\cP^\times$ is non-split %we can restrict $\delta_T$ to the subgroup $C_c(T,R)$ of $C_\diamond(T,R)$. Moreover, 
$C_\diamond(T(F_\cP),R)=C(T(F_\cP),R)$, and $\bar\delta_T$ is bijective.
\end{remark}

\begin{lemma}\label{bounddelta}
Assume that $R$ is a domain endowed with the discrete topology and $\alpha\neq\pm 1$. Then there exists $\lambda\in R$ such that
\[
\lambda\left({\rm Im}(\delta_T)\right)\subseteq \Ind^{\GL_2}_{ZU}1_R/(\cT_\cP-a_\cP)\Ind^{G}_{ZU}1_R,
\]
where $U=K_0(1)=\GL_2(\cO_{F_\cP})$.
\end{lemma}
\begin{proof}\footnote{Due to L. Gehrmann}
If we prove that there exists $H\subseteq T(F_\cP)$ an small enough open compact neighborhood of 1 such that ${\rm Im}(\delta_T)=R[\GL_2(F_\cP)]\delta_T(1_H)$, where $1_H$ is the characteristic function of $H$, then the result will automatically follow. Indeed, we can choose $\lambda\in R$ to be such that $\lambda\delta_T(1_H)\in \Ind^{\G_2}_{ZU}1_R/(\cT_\cP-a_\cP)\Ind^{\GL_2}_{ZU}1_R$.

Notice that the continuous equivariant map 
\[
\varphi:\GL_2(F_\cP)\longrightarrow  \PP^1(F_\cP), \qquad\left(\begin{array}{cc}a&b\\c&d\end{array}\right)\longmapsto -\frac{d}{c}.
\]
provides an injection of $T(F_\cP)$ in $\PP^1(F_\cP)$ that sends $1$ to $\infty$. The open subsets $U_n=\{x\in\PP^1(F_\cP),\;\ord(x)<-n\}$ ($n\geq 0$) form a neighborhood basis for $\infty$ and their preimages give rise to neighborhood basis $\{H_n\}_n$ of $1\in T(F_\cP)$. We check that $g_\varpi=\left(\begin{array}{cc}\varpi&\\&1\end{array}\right)$
satisfies that $g_\varpi^{-1} U_n=U_{n+1}$. Thus,
\[
\varphi(H_{n+1})=g_\varpi^{-1} U_n=\varphi(H_n g_\varpi),\qquad \mbox{hence}\quad H_{n+1}\subseteq P_\cP H_n g_\varpi.
\] 
Let $n_0\in\N$ be big enough that, for any $n\geq n_0$, $H_{n}\subset K_0(1)\cap g_\varpi^{-1} K_0(1)g_\varpi$. This implies that, 
\[
H_{n+1}\subseteq P_\cP H_n g_\varpi\cap g_\varpi^{-1} K_0(1)g_\varpi=g_\varpi^{-1}(P_\cP\cap K_0(1))H_n g_\varpi.
\]
It is clear that $\delta_T(1_{H_{n+1}})(g)=\pi_\alpha(g_\varpi^{-1})\delta_T(1_{H_{n}})(g)=0$ if $\varphi(g)\not \in U_{n+1}$. Assume that $\varphi(g) \in U_{n+1}$, then $g=bh_{n+1}$ for some $b\in P_\cP$ and $h_{n+1}\in H_{n+1}$. The above claim implies that $h_{n+1}=g_\varpi^{-1}b'h_n g_\varpi$, where $b'\in P_\cP\cap K_0(1)$ and $h_n\in H_n$. We compute
\begin{eqnarray*}
\pi_\alpha(g_\varpi^{-1})\delta_T(1_{H_{n+1}})(g)&=&\delta_T(1_{H_{n}})(bh_{n+1}g_\varpi^{-1})=\delta_T(1_{H_{n}})(bg_\varpi^{-1}b'h_n)=\mu_\alpha(g_\varpi)^{-1}\mu_\alpha(b)\\
\delta_T(1_{H_{n+1}})(g)&=&\delta_T(1_{H_{n+1}})(bh_{n+1})=\mu_\alpha(b)
\end{eqnarray*}
This implies that $\delta_T(1_{H_{n+1}})=\mu_\alpha(g_\varpi)\pi_\alpha(g_\varpi^{-1})\delta_T(1_{H_n})\in R[\GL_2(F_\cP)]\delta_T(1_{H_n})$. By induction, we deduce that $\delta_T(1_{H_{n}})\in R[\GL_2(F_\cP)]\delta_T(1_{H_{n_0}})$, for any $n>0$. Since $1_{H_n}$ generate the $T(F_\cP)$-module $C_c(T(F_\cP),R)$ and $\delta_T$ is $T(F_\cP)$-equivariant, the results follows. 
\end{proof}

\subsection{Torus and inner products}\label{TorInn}

%For this section, let $F$ be a local field (either archimedean or nonarchimedean).
For any (finite or indefinite) place $v$ of $F$,
let $D_v$ be the quaternion division algebra over $F_v$. %, and let $T^2$ be a two-dimensional torus as above. 
We fix an embedding $K_v^\times\hookrightarrow D_v^\times$ whenever it exists ($K_v$ is non-split). Let $\chi_v:T(F_v)\rightarrow\C^\times$ be any continuous character.
\begin{proposition}[Saito-Tunnel \cite{Sa}, \cite{Tu}]\label{Saito-Tunnel}
Let $\pi_v$ be a representation of $\GL_2(F_v)$ with central character.
 \begin{itemize}
  \item If either $\pi_v$ is principal or $K_v^\times$ is split, then ${\rm dim}(\Hom_{T(F_v)}(\pi_v\otimes\chi_v,\C))=1$.
  
  \item If $\pi_v$ is discrete and $K_v^\times$ is non-split, then 
    \[
      {\rm dim}(\Hom_{T(F_v)}(\pi_v\otimes\chi_v,\C))+{\rm dim}(\Hom_{T(F_v)}(\pi_v^{JL}\otimes\chi_v,\C))=1.
    \]
  where $\pi_v^{JL}$ is the Jacquet-Langlands correspondence of $\pi_v$ on $D_v^\times$.
 \end{itemize}
\end{proposition}

Assume that $\pi_v$ (and thus $\pi_v^{JL}$) is unitarizable, namely, there is an invariant hermitian inner product $\langle\;,\;\rangle$ on $\pi_v$ (and $\pi_v^{JL}$). If $\dim(\Hom_{T(F_v)}(\Pi_v\otimes\chi_v,\C))\neq 0$, where $\Pi_v=\pi_v$ or $\pi_v^{JL}$, we can consider the following distinguished element of $\Hom_{T(F_v)}(\Pi_v\otimes\chi_v,\C)\otimes \Hom_{T(F_v)}(\Pi_v\otimes\chi_v^{-1},\C)$:
\[
\beta_{\Pi_v,\chi_v}(u, w):=\int_{T(F_v)}\langle\Pi_v(t)u,w\rangle\chi_v(t_v)d^\times t_v,\qquad u, w\in\Pi_v.
\] 
\begin{proposition}[Waldspurger \cite{Walds}]\label{propWaldloc}
Given $\beta_{\pi_v,\chi_v}$ or $\beta_{\pi_v^{JL},\chi_v}$, we have
\begin{itemize}
\item [(i)] If $\dim(\Hom_{T(F_v)}(\pi_v\otimes\chi_v,\C))\neq 0$, then $\beta_{\pi_v,\chi_v}\neq 0$;
\item [(ii)] If $\dim(\Hom_{T(F_v)}(\pi_v\otimes\chi_v,\C))= 0$, then $\beta_{\pi_v^{JL},\chi_v}\neq 0$;
\item [(iii)] If $\pi_v$ is spherical, $\chi$ unramified, and $w\in \pi_v^{K_0(1)}$, $\langle w,w\rangle=1$% and $d t$ is the Haar measure on $T/Z$ such that the volume of the maximal open compact subgroup is 1
, then
\[
\beta_{\pi_v,\chi_v}(w, w)=\frac{\xi_v(2)L(1/2,\pi_v,\chi_v)}{L(1,\eta_v)L(1,\pi_v,ad)},
\]
where $\eta_v$ is the quadratic character attached to $T(F_v)$.
\end{itemize}
\end{proposition}

Due to this proposition, we can normalize the above pairing as follows:
\[
\alpha_{\pi_v,\chi_v}:=\frac{L(1,\eta_v)L(1,\pi_v,ad)}{\xi_v(2)L(1/2,\pi_v,\chi_v)}\beta_{\pi_v,\chi_v}.
\]

\subsection{Steinberg representations and intertwining operators}\label{Steinberg}

In this section, we fix a prime $\cP$, and we consider the field $\C$ endowed with discrete topology. Assume that $\alpha=\pm 1$, and write $\alpha_s:=q^s\alpha$. % and $\mu_s:=\mu_{\alpha_s}^{-1}\otimes\mu_{\alpha_s}$, for any $s\in\C$.
We observe that the representation $V_{\alpha_1}^\C=V_{\alpha_1,0}^{\C}$ admits an infinite dimensional subrepresentation $(\hat V_{\alpha_1}^\C,\hat \pi_{\alpha_1}^\C)$
\[
\hat V_{\alpha_1}^\C=\left\{\phi\in V_{\alpha_1}^\C:\;\int_{K_0(1)}\phi(g)dg=0\right\}\subset V_{\alpha_1}^\C,
\]
by Corollary \ref{coras} with $M=K_0(1)$, $g_0=1$ and $h(g)=\alpha^{\nu(\det(g))}\phi(g)$ (in this case $P_\cP M=P_\cP K_0(1)=\GL_2(F_\cP)$ and $dg$ is any Haar measure of $\GL_2(F_\cP)$). 

As above, we fix an embedding $K_\cP^\times\hookrightarrow\GL_2(F_\cP)$ whose image does not lie in $P_\cP$, hence $K_\cP^\times\cap P_\cP=Z$. By Corollary \ref{coras2} (comparison) we have that
\[
\hat V_{\alpha_1}^\C=\left\{\phi\in V_{\alpha_1}^\C:\;\int_{T(F_\cP)}\alpha^{\nu(\det(t))}\phi(t)d^\times t=0\right\}\subset V_{\alpha_1}^\C.
\]
Here the expression $\alpha^{\nu(\det(t))}$ is well-defined for $t\in T(F_\cP)$, since $\nu(\det(z))\in 2\Z$, for $z\in Z$.

Let $\phi=\bar\delta_T^\alpha(f)\in \bar V_{\alpha}^\C$,  for some $f\in C_\diamond(T(F_\cP),\C)$. Then, we can consider its \emph{flat section} (\cite[\S 4.5]{Bump}) $\phi_s:=\bar\delta_T^{\alpha_s}(f)\in \bar V_{\alpha_s}^\C$. The integral 
\begin{eqnarray*}
I(\phi,s,g):=\int_{N_\cP} \phi_s(n\omega g)dn,&
\mbox{where}&
N_\cP:=\left\{\left(\begin{array}{cc}
					 1  & \\
					x  & 1
					\end{array}\right):\; x\in F_\cP\right\}\simeq F_\cP,\\
					&&\omega:=\left(\begin{array}{cc}
					   & -1\\
					 1 & 
					\end{array}\right).
\end{eqnarray*}
converges absolutely for ${\rm Re }(s)>1/2$ and admits analytic continuation to all $s\in \C$ \cite[Proposition 4.5.6, Proposition 4.5.7]{Bump}. By abuse of notation, we also denote its analytic continuation by $I(\phi,s,g)$. 

Recall that the map 
\begin{equation}\label{defvarphi}
\varphi:\GL_2(F_\cP)\longrightarrow P_\cP\backslash \GL_2(F_\cP)\longrightarrow \PP^1(F_\cP), \qquad\left(\begin{array}{cc}a&b\\c&d\end{array}\right)\longmapsto -\frac{d}{c}.
\end{equation}
provides an injection of $T(F_\cP)$ in $\PP^1(F_\cP)$.
We have that $\varphi(N_\cP\omega)=\PP^1(F_\cP)\setminus\{\infty\}=F_\cP$. Hence, for any $t\in T(F_\cP)$ but (possibly) one, there is a unique $n\in N_\cP$, such that $n\omega\in P_\cP\bar t^{-1}$, for any preimage $\bar t\in K_\cP^\times$ of $t$. We denote this fact by $n\omega\in P_\cP t^{-1}$.
\begin{definition}
We consider the function $\theta_T(s)(t)$, defined for almost all $t\in T(F_\cP)$ by $\theta_T(s)(t):=\mu_{\alpha_{s-1}}(n\omega t)$, where $n\omega\in P_\cP t^{-1}$ and $n\in N_\cP$. 
\end{definition}
\begin{remark}
The expression $\mu_{\alpha_{s-1}}(n\omega t)$ is well-defined, since $n\omega \bar t\in P_\cP$, for any preimage $\bar t\in K_\cP^\times$ of $t$, and $\mu_{\alpha_{s-1}}(n\omega \bar t)$ does not depend on the given preimage.
\end{remark}
If $t\in T(F_\cP)$ and $n\in N_\cP$, let $y\in T(F_\cP)$ such that $n\omega\in P_\cP y^{-1}$ (this can be done for all $n$ but maybe finitely many). Thus,
\begin{eqnarray*}
\phi_s(n \omega t)&=&\phi_s(n\omega yy^{-1}t)=\kappa(n\omega y)\mu_{\alpha_{s-1}}(n\omega y)f(t^{-1}y)\\
&=&\kappa(n\omega y)\theta_T(s)(y)(t\ast f)(y),
\end{eqnarray*}
where $\kappa:P_\cP\rightarrow\R$ is the modular quasi character defined in Appendix \S\ref{locint}. Hence, if we define $h_t\in C(\GL_2(F_\cP),\C)$ by $h_t(by):=\kappa(b)\theta_T(s)(y^{-1})(t\ast f)(y^{-1})$, for $b\in P_\cP$ and $y\in K_\cP^\times$, we have that $\phi_s(n \omega t)=h_t(n\omega)$. We compute,
\begin{eqnarray}
I(\phi,s,t)&=&\int_{N_\cP} \phi_s(n \omega t)dn=\int_{N_\cP} h_t(n \omega)dn=C_T\int_{T(F_\cP)} h_t(y)d^\times y\\
&=&%C_T\int_{T(F_\cP)}\theta_T(s)(y)(t\ast f)(y)d^\times y=
C_T\int_{T(F_\cP)}\theta_T(s)(y)f(t^{-1}y)d^\times y,\label{eqntT}
\end{eqnarray}
by Comparison (Corollary \ref{coras2}).

We define $\Lambda(\phi)(g):=I(\phi,0,g)$. The following result proves that the expression $\Lambda(\phi)$ provides a well-defined intertwining operator.
\begin{proposition}\label{propintert}
We have that:
\begin{enumerate}
\item [(i)] $I(\phi,s,bg)=\mu_{\alpha_{1-s}}(b)I(\phi,s,g)$, thus, $\Lambda(\phi)\in V_{\alpha_1}^\C$.

\item[(ii)] The image of the morphism 
\[
\Lambda:\bar V_\alpha^\C\longrightarrow V_{\alpha_1}^\C,\quad \phi\longmapsto\Lambda(\phi)
\]
lies in fact in $\hat V_{\alpha_1}^\C$.

\item[(iii)] Let $\phi_0\in \bar V_{\alpha}^\C$, defined by $\phi_0(g)=\alpha^{\nu(\det g)}$, then $\Lambda(\phi_0)=0$.

\item [(iv)] The intertwining operator $\Lambda$ induces an isomorphism 
\[
\Lambda:V_\alpha^\C\stackrel{\simeq}{\longrightarrow} \hat V_{\alpha_1}^\C. 
\]
\end{enumerate}
\end{proposition}
\begin{proof}
Part $(i)$ follows from a direct computation. Part $(ii)$ follows from the fact that $\Lambda^{-1}(\hat V_{\alpha_1}^\C)$ must be an infinite dimensional subrepresentation of $\bar V_\alpha^\C$, hence it must be $\bar V_\alpha^\C$ itself. Notice that, if $\Lambda(\phi_0)\neq 0$, it generates a 1-dimensional subrepresentation in $V_{\alpha_1}^\C$. Since $V_{\alpha_1}^\C$ has no 1-dimensional subrepresentations, part $(iii)$ follows.
Finally, part $(iv)$ follows from the fact that $V_\alpha^\C$ is irreducible and $\Lambda$ is non-zero.
\end{proof}

\subsection{Local pairings}\label{locpair}
Throughout this section we will choose the Haar measure of $T(F_\cP)$ so that the image of $\cO_{K_\cP}^\times$ has volume 1. 
In the previous sections, we have defined a $T(F_\cP)$-equivariant morphism $\delta_T: C_\diamond(T(F_\cP),\C)\longrightarrow(\pi_\cP,V_\cP)$, for any representation $\pi_\cP:=\pi_{\alpha,0}^\C$ associated with $\alpha\in \C^\times$. Moreover, we have defined the pairing $\beta_{\pi_\cP,\chi_\cP}:V_\cP\times V_\cP\rightarrow\C$. The aim of this section is to compute $\beta_{\pi_\cP,\chi_\cP}(\delta_T(f_1),\delta_T(f_2))$, for $f_1,f_2\in C_\diamond(T(F_\cP),\C)$.

We assume that either $\alpha=\pm 1$ or $|\alpha|^2=q$, the cardinal of the residue field of $F$. We know that the representation is unitarizable, namely, there is an invariant hermitian inner product 
$\langle\cdot,\cdot\rangle: V_\cP\times V_\cP\rightarrow\C$. First, let us fix the inner product:
\begin{definition}
If either $\alpha=\pm 1$ or $|\alpha|^2=q$, we denote by $\langle\cdot,\cdot\rangle$ the hermitian inner product on $(V_\cP,\pi_\cP):=(V_{\alpha,0}^\C,\pi_{\alpha,0}^\C)$ given by
$$
\langle \vartheta_1,\vartheta_2\rangle=\left\{\begin{array}{ll}
					  \int_{K_0(1)} \vartheta_1(g)\overline{\vartheta_2(g)}dg, & |\alpha|^2=q,\\
					 \int_{K_0(1)} \vartheta_1(g)\overline{\Lambda(\vartheta_2)(g)}dg, & \alpha=\pm 1,
					\end{array}\right.
$$
for any $\vartheta_1,\vartheta_2\in V_{\alpha,0}^\C$. It is $G$-invariant, by Corollary \ref{coras} of the Appendix with $M/Z_M=K$. 
\end{definition}
\begin{remark}
Note that, in the case $\alpha=\pm 1$, the integral $\int_{K_0(1)} \vartheta_1(g)\overline{\Lambda(\vartheta_2)(g)}dg$ does not depend on the representative of $\vartheta_1$ and $\vartheta_2$ in $\bar V_{\alpha,0}^\C$, since $\Lambda(\phi_0)=0$ by Proposition \ref{propintert} (iii) and
\[
 \int_{K_0(1)} \phi_0(g)\overline{\Lambda(\vartheta_2)(g)}dg= \int_{K_0(1)} \overline{\Lambda(\vartheta_2)(g)}dg=0,
\]
by Proposition \ref{propintert} (ii).
\end{remark}

The following result computes the pairing $\langle\cdot,\cdot\rangle$ in terms of the torus $T(F_\cP)$. 
\begin{proposition}\label{innerprod}
%Assume that $T^2\not\subset B$ is a 2-dimensional torus, $T=T^2/Z$ and $\vartheta_1,\vartheta_2\in V$. Then, 
There exists a constant $c_T$ such that  
\[
c_T^{-1}\langle \vartheta_1,\vartheta_2\rangle=\left\{\begin{array}{ll}
					  \int_{T(F_\cP)} \vartheta_1(t)\overline{\vartheta_2(t)}d^\times t, & |\alpha|^2=q,\\
					 \int_{T(F_\cP)} \vartheta_1(t)\overline{\Lambda(\vartheta_2)(t)}d^\times t, & \alpha=\pm 1,
					\end{array}\right.
\]
for all $\vartheta_1,\vartheta_2\in V_\cP$.
\end{proposition}
\begin{proof}
Note that, $T(F_\cP)\cap P_\cP=Z$ because $K_\cP^\times\not\subset P_\cP$. Since $|\alpha|^2=q$ or $1$, we have that
\[
h(g):=\left\{\begin{array}{ll}
					  \vartheta_1(g)\overline{\vartheta_2(g)}, & |\alpha|^2=q,\\
					  \vartheta_1(g)\overline{\Lambda(\vartheta_2)(g)}, & \alpha=\pm 1,
					\end{array}\right. 
\]
satisfies $h(bg)=\kappa(b)h(g)$, where $\kappa$ is the modular quasicharacter. 
The result follows now from Corollary \ref{coras2} of the Appendix with $M_1/Z_{M_1}=K_0(1)$, $M_2/Z_{M_2}=T(F_\cP)$.
\end{proof}

%The following proposition computes the period $\beta_{\pi,\chi}(\delta_T(f_1),\delta_T(f_2))$ for $f_1,f_2\in C_\diamond(T,\C)$.
\begin{proposition}\label{prop-sc-prod} 
Assume that $K_\cP^\times\not\subset P_\cP$, and let $f_1,f_2\in C_\diamond(T(F_\cP),\C)$. Then, there is a non-zero constant $c_T$, depending on $T(F_\cP)$, such that $\beta_{\pi_\cP,\chi_\cP}(\delta_T(f_1),\delta_T(f_2))$ equals to:
 \[
\left\{\begin{array}{ll}
					  c_T\left(\int_{T(F_\cP)}f_1(t)\chi_\cP^{-1}(t)d^\times t\right)\overline{\left(\int_{T(F_\cP)}f_2(t)\chi_\cP^{-1}(t)d^\times t\right)}, & |\alpha|^2=q,\\
					 c_T\left(\int_{T(F_\cP)}f_1(t)\chi_\cP^{-1}(t)d^\times t\right)\overline{\left(\int_{T(F_\cP)}\Lambda(\delta_T(f_2))(t)\chi_\cP(t)d^\times t\right)}, & \alpha=\pm 1.
					\end{array}\right.
\] 
\end{proposition}
\begin{proof}
Let $f_2^*:=f_2$ if $|\alpha|^2=q$ and $f_2^*(x):=\Lambda(\delta_T(f_2))(x^{-1})$ if $\alpha=\pm 1$.
Using Proposition \ref{innerprod}, we obtain,
\begin{eqnarray*}
\int_{T(F_\cP)}\langle \pi_\cP(t)\delta_T(f_1),\delta_T(f_2)\rangle\chi_\cP(t)d^\times t=
c_T\int\int_{T(F_\cP)}f_1(x^{-1}t^{-1})\overline{f_2^*(x^{-1})}\chi_\cP(t)d^\times xd^\times t \\
%&=&c_T\int_{T}\int_{T}f_1(y)\overline{f_2^\ast(x^{-1})}\chi(yx)dydx \quad (y=x^{-1}t^{-1})\\
%&=&c_T\int_{T}\int_{T}f_1(y)\overline{f_2^\ast(x^{-1})}\chi(y)\overline{\chi(x^{-1})}dydx \quad (\chi^{-1}=\bar\chi)\\
=c_T\left(\int_{T(F_\cP)}f_1(y)\chi_\cP^{-1}(y)d^\times y\right)\overline{\left(\int_{T(F_\cP)}f_2^*(y)\chi_\cP^{-1}(y)d^\times y\right)}, %\quad (y=x^{-1})
\end{eqnarray*}
and the result follows (we are allowed to exchange the order of the integrals, since $f_i\in C_\diamond(T(F_\cP),\C)$).
\end{proof}

Let us assume for the rest of the section that $\alpha=\pm 1$. Proposition \ref{prop-sc-prod} implies that, in order to compute $\beta_{\pi_\cP,\chi_\cP}(\delta_T(f_1),\delta_T(f_2))$, one has to describe the integral $\int_{T(F_\cP)}\Lambda(\delta_T(f))(t)\chi_\cP(t)d^\times t$ in terms of $f\in C_\diamond(T(F_\cP),\C)$. By Equation \eqref{eqntT}
\begin{eqnarray*}
\int_{T(F_\cP)}I(\delta_T(f),s,t)\chi_\cP(t)d^\times t=C_T\int_{T(F_\cP)}\int_{T(F_\cP)}\theta_T(s)(y)f(t^{-1}y)\chi_\cP(t)d^\times yd^\times t\\
%&=&C_T\int_T\int_T\theta_T(s)(y)f(x)\chi(x^{-1}y)dtdy\quad(x=t^{-1}y)\\
=C_T\left(\int_{T(F_\cP)}\theta_T(s)(t)\chi_\cP(t)d^\times t\right)\left(\int_{T(F_\cP)}f(t)\chi_\cP^{-1}(t)d^\times t\right).
\end{eqnarray*}
Again, the integral $I_T(\chi_\cP,s):=\int_{T(F_\cP)}\theta_T(s)(t)\chi_\cP(t)d^\times t$ converges absolutely for ${\rm Re }(s)>1/2$ and admits analytic continuation to all $s\in \C$. By abuse of notation, denote also by $I_T(\chi_\cP,s)$ its analytic continuation. We conclude
\[
\int_{T(F_\cP)}\Lambda(\delta_T(f))(t)\chi_\cP(t)d^\times t%=\int_{T(F_\cP)}I(\delta_T(f),0,t)\chi_\cP(t)d^\times t
=C_TI_T(\chi_\cP,0)\int_{T(F_\cP)}f(t)\chi_\cP^{-1}(t)d^\times t.
\]

\begin{remark}
If we assume that $\chi_\cP(t)=\alpha^{\nu(\det(t))}$, for all $\phi=\delta_T(f)\in V_\cP$ 
\[
0=\int_{T(F_\cP)}\alpha^{\nu(\det(t))}\Lambda(\phi)(t)d^\times t=C_TI_T(\chi_\cP,0)\int_{T(F_\cP)}\alpha^{\nu(\det(t))}f(t)d^\times t
\]
by Proposition \ref{propintert} (ii). This implies that $I_T(\chi_\cP,0)=0$ in this case. We call this the \emph{Excepcional Zero Phenomenon}.
\end{remark}

%Let $E/F$ be the quadratic extension, such that $T^2=E^\times$. Then $T^2$ is split if, and only if, $E=F\times F$. If $T^2$ is non-split, we say that $T^2$ is inert if the extension $E/F$ is inert, and we say that $T^2$ ramifies if $E/F$ ramifies. Let $\cO_E$ be the integer ring of $E$. 
Let $\eta$ be the quadratic character associated with $T(F_\cP)$. Hence $L(s,\eta)=(1-q^{-s})^{-1}$ if $K_\cP/F_\cP$ is split, $L(s,\eta)=(1+q^{-s})^{-1}$ if $K_\cP/F_\cP$ is inert, or $L(s,\eta)=1$ if $K_\cP/F_\cP$ is ramified. Recall the local Riemann zeta function $\zeta_\cP(s)=(1-q^{-s})^{-1}$. Note that the maximal compact subgroup $H_0$ of $T(F_\cP)$ admits a natural composition series $H_0\supseteq H_1\supseteq\cdots\supseteq H_n\supseteq\cdots$ such that $[H_{n}:H_{n+1}]=q$, for all $n>0$. Given $\chi_\cP$, its conductor $n_\chi$ is the integer such that $\chi_\cP\mid_{H_{n_\chi}}=1$ and $\chi_\cP\mid_{K_{n_\chi-1}}\neq 1$.
\begin{theorem}\label{teocompIT}
There is an integer $n_T$, depending on the embedding $K_\cP^\times\hookrightarrow\GL_2(F_\cP)$, such that the analytic continuation $I_T(\chi_\cP,s)$ is given by
\[
I_T(\chi_\cP,s)=\left\{\begin{array}{ll}
q^{(2-2s)n_T}L(1,\eta)\frac{\zeta_\cP(2s-1)}{\zeta_\cP(2-2s)}q^{(1-2s)n_\chi},&\chi_\cP\mid_{\cO_{K_\cP^\times}}\neq 1,\\
q^{(2-2s)n_T}L(1,\eta)\frac{\zeta_\cP(2s-1)}{\zeta_\cP(2-2s)}\frac{L(-s+1/2,\pi_\cP,\chi_\cP)}{L(s-1/2,\pi_\cP,\chi_\cP)},&\chi_\cP\mid_{\cO_{K_\cP^\times}}= 1.
\end{array}\right.
\]
where $n_\chi$ is the conductor of $\chi_\cP$ and, for an unramified character $\chi_\cP$,
\[
L(s,\pi_\cP,\chi_\cP)=\left\{\begin{array}{ll}
(1-\alpha\chi_\cP(\varpi)q^{-s-1/2})^{-1}(1-\alpha\chi_\cP(\varpi)^{-1}q^{-s-1/2})^{-1},& K_\cP/F_\cP\;\mbox{splits}\\
(1-q^{-2s-1})^{-1},& K_\cP/F_\cP\;\mbox{inert}\\
(1-\alpha\chi(\varpi_K)q^{-s-1/2})^{-1},& K_\cP/F_\cP\;\mbox{ramifies},
\end{array}\right.
\]
with $\varpi_K$ a uniformizer of $\cO_{K_\cP}$ in the ramified case.
\end{theorem}
\begin{proof}
Let us assume that the embedding is given by
\[
K_\cP^\times\hookrightarrow\GL_2(F_\cP);\qquad t\longmapsto\left(\begin{array}{cc}a(t)&b(t)\\c(t)&d(t)\end{array}\right).
\]
Hence we compute
\[
\theta_T(s)(t)=\mu_{s-1}\left(\left(\begin{array}{cc}
					 1  & \\
					\frac{a(t)}{c(t)}  & 1
					\end{array}\right)\omega \left(\begin{array}{cc}
					 a(t)  & b(t) \\
					c(t)  & d(t)
					\end{array}\right)\right)=\alpha^{\nu(\det(t))}\left|\frac{\det(t)}{c(t)^2}\right|^{1-s}.
\]

{\bf Assume that $K_\cP$ is split}. Thus $T(F_\cP)\simeq F_\cP^\times$. Since any split torus is conjugated to the group of diagonal matrices, we compute that $\det(x)=x$ and $c(x)=C(x-1)$, for some $C\in F_\cP$. Since $K_\cP^\times\not\subset P_\cP$, we have $C\neq 0$. Hence we write $n_T=\ord(C)>-\infty$. We write $\cO_\cP=\cO_{F_\cP}$ and compute
\begin{eqnarray*}
I_T(\chi_\cP,s)&=&\int_{T(F_\cP)}\theta_T(s)(t)\chi_\cP(t)d^\times t=\int_{F_\cP^\times}\chi_\cP(x)\alpha^{\nu( x)}\left|\frac{x}{C^2(x-1)^2}\right|^{1-s}d^\times x\\
&=&q^{(2-2s)n_T}\left(\int_{F_\cP\setminus\cO_\cP}\alpha^{\nu( x)}\chi_\cP(x)|x|^{s-1}d^{\times}x+\int_{\cO_\cP^\times}\chi_\cP(x)|x-1|^{2s-2}d^\times x+\right.\\
&&\left.+\int_{\cO_\cP\setminus\cO_\cP^\times}\alpha^{\nu( x)}\chi_\cP(x)|x|^{1-s}d^{\times}x\right).
\end{eqnarray*}
Two of these integrals can be calculated easily:
\begin{eqnarray*}
\int_{F_\cP\setminus\cO_\cP}\alpha^{\nu (x)}\chi_\cP(x)|x|^{s-1}d^{\times}x%&=&\sum_{n>0}\left(\frac{q^{(s-1)}\alpha}{\chi(\varpi)}\right)^{n}\int_{\cO^\times}\chi(x)d^\times x
&=&\frac{q^{s-1}\alpha\chi_\cP(\varpi)^{-1}}{1-q^{s-1}\alpha\chi_\cP(\varpi)^{-1}}\int_{\cO_\cP^\times}\chi_\cP(x)d^\times x.
\end{eqnarray*}
\begin{eqnarray*}
\int_{\cO_\cP\setminus\cO_\cP^\times}\alpha^{\nu( x)}\chi_\cP(x)|x|^{1-s}d^{\times}x%&=&\sum_{n>0}q^{(s-1)n}\alpha^n\chi(\varpi)^{n}\int_{\cO^\times}\chi(x)d^\times x\\
&=&\frac{q^{s-1}\alpha\chi_\cP(\varpi)}{1-q^{s-1}\alpha\chi_\cP(\varpi)}\int_{\cO_\cP^\times}\chi_\cP(x)d^\times x.
\end{eqnarray*}
For any $n>0$, we have that $H_n:=1+\varpi^n\cO_\cP$. %Let $n_\chi$ be the conductor of $\chi$, namely, $\chi\mid_{\cO_{n_\chi}^\times}=1$, but $\chi\mid_{\cO_{n_\chi-1}^\times}\neq 1$. 
By orthogonality $\int_{H_n}\chi_\cP(x)d^\times x=0$ if $n<n_\chi$, and $\int_{H_n}\chi_\cP(x)d^\times x={\rm vol}(H_n)$ otherwise. %Write also $U_n:=\cO_n^\times\setminus\cO_{n+1}^\times$. 
Since $[\cO^\times:\cO_n^\times]=q^{n-1}(q-1)$ for $n>0$, and the Haar measure satisfies ${\rm vol}(\cO_\cP^\times)=1$, we obtain that ${\rm vol}(H_n)=(q-1)^{-1}q^{1-n}$ ($n>0$),  ${\rm vol}(H_n\setminus H_{n+1})=q^{-n}$ ($n>0$) and ${\rm vol}(\cO_\cP^\times\setminus H_1)=(q-1)^{-1}(q-2)$.
Thus,
\begin{eqnarray*}
\int_{\cO_\cP^\times}\chi_\cP(x)|x-1|^{2s-2}d^\times x&=&\sum_{n\geq0}q^{(2-2s)n}\int_{H_n\setminus H_{n+1}}\chi_\cP(x)d^\times x\\
%&=&\left\{\begin{array}{ll}
%\frac{-q^{(1-2s)(n_\chi-1)}}{q-1}+\sum_{n\geq n_\chi}q^{(1-2s)n},&n_\chi>0,\\
%\frac{q-2}{q-1}+\sum_{n>0}q^{(1-2s)n},&n_\chi=0
%\end{array}\right.\\
&=&\left\{\begin{array}{ll}
\frac{q^{(1-2s)n_\chi+1}-q^{(1-2s)(n_\chi-1)}}{(1-q^{1-2s})(q-1)},&n_\chi>0,\\
\frac{q-2+q^{1-2s}}{(q-1)(1-q^{1-2s})},&n_\chi=0.
\end{array}\right.
\end{eqnarray*}
We conclude
\[
I_T(\chi_\cP,s)=\left\{\begin{array}{ll}
\frac{q^{(2-2s)n_T}}{(1-q^{1-2s})(q-1)}(q^{(1-2s)n_\chi+1}-q^{(1-2s)(n_\chi-1)}),&n_\chi>0,\\
\frac{q^{(2-2s)n_T}(1-q^{2s-2})}{(1-q^{-1})(1-q^{1-2s})}\left(\frac{(1-\alpha\chi_\cP(\varpi)q^{-s})(1-\alpha\chi_\cP(\varpi)^{-1}q^{-s})}{(1-\alpha\chi_\cP(\varpi)q^{s-1})(1-\alpha\chi_\cP(\varpi)^{-1}q^{s-1})}\right),&n_\chi=0.
\end{array}\right.
\]

{\bf Assume that  $K_\cP$ is inert}. Thus $K_\cP^\times\simeq \varpi^\Z\cO_{K_\cP}^\times$. %, where $\cO_E$ is the integer ring of $E$. 
We have $T(F_\cP)=\cO_{K_\cP}^\times/\cO_\cP^\times$ and $\cO_{K_\cP}=\cO_\cP+\beta\cO_\cP$, for some $\beta\in \cO_{K_\cP}^\times$. Let $n_T:=\nu(c(\beta))$. In this case $H_n:=(\cO_\cP+\varpi^n\beta\cO_\cP)^\times/\cO_\cP^\times$, for every $n\in \N$, we compute
\begin{eqnarray*}
I_T(\chi_\cP,s)&=&\int_{T(F_\cP)}\theta_T(s)(t)\chi_{F_\cP}(t)d^\times t=\int_{\cO_{K_\cP}^\times/\cO_\cP^\times}\chi_\cP(t)|c(t)|^{2s-2}d^\times t\\
&=&\sum_{n\geq 0}\int_{H_n\setminus H_{n+1}}\chi_\cP(t)q^{(n_T+n)(2-2s)}d^\times t.
\end{eqnarray*}
%Set $U_n:=(\cO_{n}^\times/\cO^\times)\setminus(\cO_{n+1}^\times/\cO^\times)$. If $\chi$ has conductor $n_\chi$, namely, it factorizes by $\cO_E^\times/\cO_{n_\chi}^\times$, we obtain 
This implies that
\[
I_T(\chi_\cP,s)=\sum_{n_\chi> n\geq 0}q^{(n_T+n)(2-2s)}\int_{H_n\setminus H_{n+1}}\chi_\cP(t)d^\times t+\sum_{n\geq n_\chi}q^{(n_T+n)(2-2s)}{\rm vol}(H_n\setminus H_{n+1})
\]
When the character $\chi\mid_{H_n}$ is not trivial ($n<n_\chi$), we have that
$\int_{H_n}\chi_\cP(t)d^\times t=0$, by orthogonality.
This implies that, if $n<n_\chi-1$,
\[
\int_{H_n\setminus H_{n+1}}\chi_\cP(t)d^\times t=\int_{H_n}\chi_\cP(t)d^\times t-\int_{H_{n+1}}\chi_\cP(t)d^\times t=0.
\]
On the other side, if $n=n_\chi-1$,
\[
\int_{H_{n_\chi-1}\setminus H_{n_\chi}}\chi_\cP(t)d^\times t=\int_{H_{n_\chi-1}}\chi_\cP(t)d^\times t-\int_{H_{n_\chi}}\chi_\cP(t)d^\times t=-{\rm vol}(H_{n_\chi}).
\]
Note that $T(F_\cP)$ is compact and ${\rm vol}(T(F_\cP))=1$. Moreover, %since $[\cO_{K_\cP}^\times:\cO_n^\times]=(q+1)q^{n-1}$ for $n>1$, 
we compute that ${\rm vol}(H_n)=(q+1)^{-1}q^{1-n}$ whenever $n>0$. Thus ${\rm vol}(H_n\setminus H_{n+1})=(q+1)^{-1}q^{-n}(q-1)$ if $n>0$, and ${\rm vol}(H_0\setminus H_1)=q(q+1)^{-1}$.
With all these computations we obtain the value of $I_T(\chi_\cP,s)$:
\[
I_T(\chi_\cP,s)=\left\{\begin{array}{ll}
\frac{q^{(2-2s)n_T}}{q+1}\left(-q^{(1-2s)(n_\chi-1)}+(q-1)\sum_{n\geq n_\chi}q^{(1-2s)n}\right),&\chi\neq 1\\
\frac{q^{(2-2s)n_T}}{q+1}\left(q+(q-1)\sum_{n>0}q^{(1-2s)n}\right),&\chi= 1.
\end{array}\right.
\]
We conclude that
\[
I_T(\chi_\cP,s)=\left\{\begin{array}{ll}
\frac{q^{(2s-2)n_T}}{(q+1)(1-q^{1-2s})}(q^{(1-2s)n_\chi+1}-q^{(1-2s)(n_\chi-1)}),&\chi\neq 1\\
\frac{q^{(2-2s)n_T}}{(1+q^{-1})(1-q^{1-2s})}(1-q^{-2s}),&\chi= 1.
\end{array}\right.
\]

{\bf Assume that  $K_\cP$ ramifies}. Thus $K_\cP^\times\simeq \varpi_K^\Z\cO_{K_\cP}^\times$, where $\varpi_K^2\in\varpi\cO_\cP$. This implies $T(F_\cP)=\varpi_K^{\Z/2\Z}\times\cO_{K_\cP}^\times/\cO_\cP^\times$ and $|\det(\varpi_K)|=q$. Note that $\cO_{K_\cP}=\cO_\cP+\varpi_K\cO_{\cP}$, and let $n_T:=\ord(c(\varpi_K))$. 
In this case $H_n=(\cO_\cP+\varpi^n\varpi_K\cO_\cP)^\times/\cO_\cP^\times$.
%Let us consider now the subrings $\cO_n:=\cO+\varpi^n\varpi_K\cO$, and let $U_n:=(\cO_{n}^\times/\cO^\times)\setminus(\cO_{n+1}^\times/\cO^\times)$, for $n\in\N$.  We compute
\begin{eqnarray*}
I_T(\chi_\cP,s)=\int_{T(F_\cP)}\theta_T(s)(t)\chi_\cP(t)d^\times t=\int_{T(F_\cP)}\alpha^{\nu(\det(t))}\chi_\cP(t)\left|\frac{\det(t)}{c(t)^2}\right|^{1-s}d^\times t\\
=\int_{\varpi_K\cO_{K_\cP}^\times/\cO_\cP^\times}\alpha q^{s-1}\chi_\cP(t)\left| c(t) \right|^{2s-2}d^\times t+\int_{\cO_{K_\cP}^\times/\cO_\cP^\times}\chi_\cP(t)\left|c(t)\right|^{2s-2}d^\times t\\
=q^{(2n_T-1)(1-s)}\alpha\chi_\cP(\varpi_K)\int_{\cO_{K_\cP}^\times/\cO_\cP^\times}\chi_\cP(t)d^\times t+\sum_{n\geq 0}\int_{H_n\setminus H_{n+1}}\chi_\cP(t)q^{(n_T+n)(2-2s)}d^\times t
\end{eqnarray*}
Again, %we denote by $n_\chi\geq 1$ the  conductor of $\chi$, namely, $\chi$ factorizes through $\cO_E^\times/\cO_{n_\chi-1}^\times$ (we reserve $n_\chi=0$ for the case $\chi=1$). 
by orthogonality,
$\int_{H_n}\chi_\cP(t)d^\times t=0$ if $n<n_\chi-1$.
Moreover, ${\rm vol}(H_n)=q^{-n}$ and ${\rm vol}(H_n\setminus H_{n+1})=q^{-n-1}(q-1)$, for $n\geq 0$. %, since $[\cO_E^\times:\cO_n^\times]=q^n$. 
Some computations analogous to the previous cases yield the formula
\begin{eqnarray*}
I_T(\chi_\cP,s)&=&\left\{\begin{array}{ll}
q^{(2-2s)n_T}\left(-q^{(1-2s)(n_\chi-1)}+\sum_{n\geq n_\chi-1}q^{(1-2s)n-1}(q-1)\right),&n_\chi>1\\
q^{(2-2s)n_T}\left(q^{s-1}\alpha\chi_\cP(\varpi_K)+\sum_{n\geq 0}q^{(1-2s)n-1}(q-1)\right),&n_\chi=1
\end{array}\right.\\
&=&\left\{\begin{array}{ll}
\frac{q^{(2-2s)n_T}}{q(1-q^{1-2s})}\left(q^{(1-2s)n_\chi+1}-q^{(1-2s)(n_\chi-1)}\right),&n_\chi>1\\
\frac{q^{(2-2s)n_T}}{1-q^{1-2s}}(1-\alpha\chi_\cP(\varpi_K)q^{-s})(1+\alpha\chi_\cP(\varpi_K)q^{s-1}),&n_\chi=1,
\end{array}\right.
\end{eqnarray*}
since $(\alpha\chi_\cP(\varpi_K))^2=1$.
\end{proof}

\begin{corollary}
The analytic continuation $I_T(\chi_\cP,s)$ satisfies  $I_T(\chi_\cP,0)=0$ if, and only if, $\chi_\cP(t)=\alpha^{\nu(\det(t))}$.
\end{corollary}
\begin{proof}
This follows directly from the above result, observing that $\zeta_\cP(-1)\neq 0$, $\zeta_\cP(2)^{-1}\neq 0$, $L(1/2,\pi_\cP,\chi_\cP)\neq 0$ and $L(-1/2,\pi_\cP,\chi_\cP)^{-1}=0$, if, and only if, $\chi_\cP(t)=\alpha^{\nu(\det(t))}$.
\end{proof}

\begin{corollary}\label{Corcalclocal}
Assume that $K_\cP^\times\not\subset P_\cP$, and let $f_1,f_2\in C_\diamond(T(F_\cP),\C)$. Then, there exists a non-zero constant $K_T$, depending on $T(F_\cP)$ and $\pi_\cP$, such that
\[
\alpha_{\pi_\cP,\chi_\cP}(\delta_T(f_1),\delta_T(f_2))=K_Te_\cP(\pi_\cP,\chi_\cP)\left(\int f_1(t)\chi_\cP^{-1}(t)d^\times t\right)\overline{\left(\int f_2(t)\chi_\cP^{-1}(t)d^\times t\right)},
\]
 \begin{eqnarray*}
(K_T,e_\cP(\pi_\cP,\chi_\cP))=
\left\{\begin{array}{ll}
					  \left(c_T\frac{L(1,\eta)}{\xi_\cP(2)},\frac{L(1,\pi_\cP,ad)}{L(1/2,\pi_\cP,\chi_\cP)}\right) & |\alpha|^2=q,\\
					 \left(c_T\bar C_T\frac{L(1,\eta)^2\xi_\cP(-1)q^{2n_T}}{\xi_\cP(2)},\frac{1}{L(-1/2,\pi_\cP,\chi_\cP)}\right), & \alpha=\pm 1,\;\chi_\cP\mid_{\cO_{K_\cP}^\times}=1,\\
					\left(c_T\bar C_T\frac{L(1,\eta)^2\xi_\cP(-1)q^{2n_T}}{\xi_\cP(2)},\frac{q^{n_\chi}}{L(1/2,\pi_\cP,\chi_\cP)}\right), & \alpha=\pm 1,\;\chi_\cP\mid_{\cO_{K_\cP}^\times}\neq1,
					\end{array}\right.
\end{eqnarray*}
where $n_\chi$ is the conductor of $\chi_\cP$.
\end{corollary}

\section{Cohomology of automorphic forms and Shimura curves}\label{Ap2}

%Let $F$ be a totally real number field. Let $G$ be the algebraic group associated with the multiplicative group of a quaternion algebra over $F$, that can be either totally definite or split at a single archimedean place $\sigma$. 
For any quaternion algebra $B/F$, we write $G_B(F)^+\subset G_B(F)$ for the subgroup of elements of positive norm.
%Let $(\Pi, V_\Pi)$ be an automorphic representation of $G(\A_F)$ of weight 2 with a trivial central character. 

Let $S$ be a finite set of nonarchimedean places. We shall usually consider $G_B(F)$ in $G_B(\hat F^S)$ by means of the natural monomorphism. %we define $\cM_T$ to be the free group generated by elements of $G(F)\slash T^2(F)$.
Given any ring $R$, let $N,M$ be a $R[G(F)]$-module  and a $R$-module respectively. We define $\cA_B^S(N,M)$ to be the module of functions $f:G_B(\hat F^S)\longrightarrow \Hom_R(N,M)$ such that there exists an open compact subgroup $U\subseteq G_B(\hat F^S)$ such that $f(\cdot\; U)=f(\cdot)$.
Note that $\cA_B^S(N,M)$ is equipped with commuting $G_B(F)$- and $G_B(\hat F^S)$-actions: 
\[
\begin{array}{cc}
(\gamma\cdot f)(g)=\gamma\left( f(\gamma^{-1} g)\right),&\gamma\in G_B(F), \\
(h\cdot f)(g)=f(g h)),&h\in G_B(\hat F^S);\\
\end{array}
\]
where $g\in G(\hat F^S)$, $f\in \cA_B^S(N,M)$ and we are considering the usual action of $G_B(F)$ on $\Hom_R(N,M)$. We write $\cA_B^\cQ(N,M)$ instead of $\cA_B^{\{\cQ\}}(N,M)$ and $\cA_B(N,M)$ instead of $\cA_B^\emptyset(N,M)$. Similarly, we define $\cA_B^S(M):=\cA_B^S(R,M)$, where $R$ is endowed with trivial $G_B(F)$-action.
Note that, if $M$ and $N$ are ${\bf C}$-vector spaces for some field ${\bf C}$, then $H^q(H,\cA_B^S(N,M))$ is a smooth $G_B(\hat F^S)$-representation over ${\bf C}$, for any subgroup $H\subseteq G_B(F)$.
\begin{remark}\label{AvsHom}
For any $G_B(F)$-module $M'$ and any $G_B(F_S)$-representation $M$ over $R$, we have an isomorphism of $(G_B(F),G_B(\hat F^S))$-representations:
\begin{equation}\label{eqiso}
\begin{array}{ccc}
\phi:\Hom_{G_B(F_S)}(M,\cA_B(M',N)) &\longrightarrow& \cA_B^S( M\otimes_R M',N)\\
\varphi&\longmapsto&\phi(\varphi)(g)(m\otimes m')=\varphi(m)(g)(m'),
\end{array}
\end{equation}
with inverse $\phi^{-1}(f)(m)(g_S,g)(m'):= f(g)(g_S m\otimes m')$, for all $g\in G_B(\hat F^S)$, $m\in M$, $m'\in M'$ and $g_S\in G_B(F_S)$. 
\end{remark}

\begin{lemma}\label{lemmaext}
Assume that $M$ is an $R$-module and $R\rightarrow R'$ a flat ring homomorphism. Then the canonical map
\[
H^q(G_B(F),\cA_B^S(M)\otimes_R R')\longrightarrow H^q(G_B(F),\cA_B^S(M\otimes_R R'))
\]
is an isomorphism for all $q\geq 0$.
\end{lemma}
\begin{proof}
This result can be found essentially in \cite[Corollary 4.7]{Spiess}. %let us reproduce its proof in our setting. It is enough to prove that $H^q(G(F),\cA_f^S(M)^{U}\otimes_R R')\simeq H^q(G(F),\cA_f^S(M\otimes_R R')^{U})$, for any compact open subgroup $U$ of $G(F^S)$. Note that, for any $R$-module $N$, $\cA_f^S(N)^U=\Coind_U^{G(F^S)}N$. Hence it is enough to prove that the functor $N\mapsto H^q(G(F),\Coind_U^{G(F^S)}N)$ commutes with direct limits (any module is the direct limit of free modules of finite rank). By the Strong Approximation Theorem, there are only finitely many double cosets $G(F)g U$ in $G(F^S)$. If $\{g_i\}_{i=1}^n\subset G(F^S)$ is a system of representatives then
%\[
%H^q(G(F),\Coind_U^{G(F^S)}N)=\bigoplus_{i=1}^n H^q(\Gamma_{g_i},N),\qquad\Gamma_{g_i}=G(F)\cap g_i^{-1}Ug_i.
%\]
%Since the group $\Gamma_{g_i}$ is $S$-arithmetic, hence of type (VFL), the functor $N\mapsto H^q(\Gamma_{g_i},N)$ commutes with direct limits (see \cite{Se}, p.101) and the result follows.
\end{proof}

\subsection{Abel-Jacobi map on Shimura curves}

In this section, we assume that $B/F$ is a quaternion algebra that splits at a single place $\sigma\mid\infty$. 
Let us consider the $\C$-vector space $\cA(\C)$ of functions $f:G_B(F_\sigma)\times G_B(\hat F)\longrightarrow\C$ such that:
\begin{itemize}
\item There is an open compact subgroup $U\subseteq G_B(\hat F)$ such that $f(\cdot\; U)=f(\cdot)$.

\item $f\mid_{G_B(F_\sigma)}\in C^\infty(\PGL_2(\R),\C)$, under a fixed identification $G_B(F_\sigma)\simeq\PGL_2(\R)$.

\item Fixing $O_\sigma$, a maximal compact subgroup of $G_B(F_\sigma)$ isomorphic to the image of $\rO(2)$, we assume that any $f\in\cA(\C)$ is $O_\sigma$-finite, namely, its right translates by elements of $O_\sigma$ span a finite-dimensional vector space. 

\item We assume that any $f\in\cA(\C)$ must be ${\mathcal{Z}}$-finite, where ${\mathcal{Z}}$ is the centre of the universal enveloping algebra of the Lie algebra of $G_B(F_\sigma)$.
\end{itemize}
Write $\rho$ for the action of $G_B(F_\sigma)\times G_B(\hat F)$ given by right translation, $(\cA(\C),\rho)$ defines a smooth $G_B(\hat F)$-representation and a $(\dG_\sigma,O_\sigma)$-module, where $\dG_\sigma$ is the Lie algebra of $G_B(F_\sigma)$. Moreover, $\cA(\C)$ is also equipped with the $G_B(F)$-action:
\[
(h\cdot f)(g_\sigma,  g)=f(h^{-1} g_\sigma,h^{-1} g),\quad h\in G_B(F),\;g\in G_B(\hat F), \; g_\sigma\in G_B(F_\sigma), \;f\in \cA(\C).
\]
Let us fix an isomorphism $G_B(F_\sigma)\simeq\PGL_2(\R)$ that maps $O_\sigma$ to the image of $\rO(2)$. Let $V$ be a $(\dG_\sigma,O_\sigma)$-module. In analogy with Remark \ref{AvsHom}, we define
\[
\cA^\sigma(V,\C):=\Hom_{(\dG_\sigma,O_\sigma)}(V,\cA(\C)),
\]
endowed with the natural $G_B(F)$- and $G_B(\hat F)$-actions. Let us consider 1-dimensional $\PGL_2(\R)$-representations $\C(\pm 1)$ of Appendix 2 \S \ref{DiscSer}. Note that we have the isomorphism of $(G_B(F), G_B(\hat F))$-modules
\[
s^\pm:\cA^\sigma(\C(\pm 1),\C)\longrightarrow\cA_B(\C(\pm 1),\C);\qquad s^\pm(\phi)(g_f)(z)=\phi(z)(1,g_f).
\]

\subsubsection{Cohomology of a Shimura curve}\label{CohShi}
 
For any open compact subgroup $U\subset G_B(\hat F)$, let us consider the Shimura curve $X_U$, whose set of non-cuspidal points $Y_U(\C)$ is in correspondence with the double coset space: 
\[
Y_U(\C)=G_B(F)^+\backslash \left(\mathfrak{H}\times G_B(\hat F)/ U\right)\subseteq X_U(\C),
\]
where $\mathfrak{H}$ is the Poincar\'e upper half-plane.
It is well known that the space of holomorphic forms $\Omega^1_{Y_U}$ of $Y_U$ can be identified with $H^0(G_B(F),\cA^\sigma(\cD,\C)^U)$ by means of the morphism
\[
H^0(G_B(F),\cA^\sigma(\cD,\C)^U)\longrightarrow\Omega^1_{Y_U};\qquad \phi\mapsto \omega_\phi(g_\infty i,g_f):=\frac{1}{2\pi i}\phi(f_2)f_2^{-1}(g_\infty,g_f)d\tau,
\]
where $\cD$ is the discrete series representation of weight 2, $f_2\in \cD$ is a generator (see Appendix 2 \S \ref{DiscSer}), $g_f\in G_B(\hat F)$, $g_\infty\in G_B(F_\sigma)^+$, and $\tau=g_\infty i\in \mathfrak{H}$.
Note that $\phi(f_2)f_2^{-1}$ is a function on $G_B(F_\sigma)^+/\SO(2)\times G_B(\hat F)=\mathfrak{H}\times G_B(\hat F)$. It defines a differential form on $Y_U$ because $\phi$ is $G(F)$-invariant. Moreover, $\omega_\phi$ is holomorphic since $Lf_2=0$ (see \eqref{rel1}) and $\phi$ is a morphism of $(\dG_\sigma,O_\sigma)$-modules. %Similarly, we could have identified such cohomology group with the set $\bar\Omega^1_{Y_U}$ of anti-holomorphic forms of $Y_U/\C$, mapping $\phi$ to the anti-holomorphic form $\bar\omega_\phi:=\frac{1}{2\pi i}\phi(f_{-2})f_{-2}^{-1}d\bar\tau$.

We claim that \eqref{ext-seq-GK1} and \eqref{ext-seq-GK2} provide exact sequences
\begin{eqnarray}
0\longrightarrow \cA_B(\C)\stackrel{\iota^+}{\longrightarrow} \cA^\sigma(I^+,\C)\stackrel{{\rm pr}^+}{\longrightarrow}\cA^\sigma(\cD,\C)\longrightarrow 0;\label{exseqfund1}\\
0\longrightarrow \cA_B(\C(-1),\C)\stackrel{\iota^-}{\longrightarrow} \cA^\sigma(I^-,\C)\stackrel{{\rm pr}^-}{\longrightarrow}\cA^\sigma(\cD,\C)\longrightarrow 0.\label{exseqfund2}
\end{eqnarray}
Indeed, since $\Hom_{(\dG_\sigma,O_\sigma)}(\cdot,\cA(\C))$ is left exact, we only have to check that ${\rm pr}^\pm$ is surjective. We note that $I^{\pm}$ is generated as a $(\dG_\sigma,O_\sigma)$-module by $f_0$, and $Rf_0=f_2$, $Lf_0=f_{-2}$ by \eqref{rel1}. Hence, any $\varphi^\pm\in \cA^\sigma(I^\pm,\C)$ is characterized by $\varphi^\pm(f_0)\in\cA(\C)$. Given $\phi\in\cA^{\sigma}(\cD,\C)$, to find a pre-image $\varphi^\pm\in({\rm pr}^\pm)^{-1}(\phi)$ is equivalent to find a $\SO(2)$-invariant $h=\varphi^\pm(f_0)\in\cA(\C)$, such that $Rh=\phi(f_2)$ and $Lh=\phi(f_{-2})$. Since,
\[
Rh=iye^{2i\theta}\left(\frac{\partial}{\partial x}-i\frac{\partial}{\partial y}\right)h=2i f_2\frac{\partial}{\partial \tau}h,\qquad Lh=-2i f_{-2}\frac{\partial}{\partial \bar\tau}h,
\]
we deduce that any pre-image $\varphi^\pm\in({\rm pr}^\pm)^{-1}\phi$ is characterized by
\begin{equation*}
d\varphi^\pm(f_0)=\frac{1}{2i}\phi(f_2)f_2^{-1}d\tau-\frac{1}{2i}\phi(f_{-2})f_{-2}^{-1}d\bar\tau%=\pi(\omega_\phi-\bar\omega_\phi).
\end{equation*} 
Hence, since any closed differential 1-form in $\mathfrak{H}$ is exact, the claim follows.

The difference between exact sequences \eqref{exseqfund1} and \eqref{exseqfund2} is the action of complex conjugation. We know that complex conjugation acts on $Y_U(\C)$ by sending $(\tau,g_f)$ to $(\gamma\bar\tau,\gamma g_f)$, for any $\tau\in\mathfrak{H}$, $g_f\in G_B(\hat F)$ and $\gamma\in G_B(F)\setminus G_B(F)^+$. Given $g_\infty i=\tau\in\mathfrak{H}$,
\[
\bar\omega_\phi(\tau,g_f):=\omega_\phi(\gamma\bar\tau,\gamma g_f)=\frac{1}{2\pi i}\phi(f_2)f_2^{-1}(\gamma g_\infty\omega,\gamma g_f)d\gamma\bar\tau=\frac{-1}{2\pi i}\phi(\omega f_2)f_{-2}^{-1}(g_\infty,g_f)d\bar\tau,
\]
since $f_2(\omega g_\infty\omega)=f_{-2}(g_\infty)$, for all $g_\infty\in \PGL_2(\R)^+$.
 As it is shown in Appendix 2 \S \ref{DiscSer}, $\phi(\omega f_2)=\pm \phi(f_{-2})$ depending whether we have chosen exact sequence \eqref{exseqfund1} or \eqref{exseqfund2}. We deduce that
\begin{equation}\label{primitive}
d\varphi^\pm(f_0)=\pi(\omega_\phi\pm\bar\omega_\phi).
\end{equation} 

Since $\cA_B(\C(-1),\C)$ and $\cA_B(\C)$ are isomorphic as $(G_B(F)^+,G_B(\hat F))$-modules, the composition of the connection morphisms of exact sequences \eqref{exseqfund1} and \eqref{exseqfund2} with the natural restriction map provide morphisms in cohomology
\[
\partial^{\pm}: H^0(G_B(F),\cA^\sigma(\cD,\C)^U)\longrightarrow H^1(G_B(F)^+,\cA_B(\C)^U).
\]
We can identify $H^1(G_B(F)^+,\cA_B(\C)^U)$ with the singular cohomology of $X_U$ with coefficients in $\C$. Moreover, once we interpret $H^0(G_B(F),\cA^\sigma(\cD,\C)^U)$ as the holomorphic differentials of $Y_U$, the morphisms $\partial^+$ and $\partial^-$ correspond to 
\[
C^\pm:\Omega_{Y_U}^1\rightarrow H^1(X_U,\C);\qquad\omega_\phi\longmapsto\left(c\mapsto\int_{c}(\omega_\phi\pm\bar\omega_\phi)\right),
\]
by \eqref{primitive}.

\begin{remark}\label{CCinHOMOL}
Note that $H^1(X_U,\Q)=H^1(X_U,\Q)^+\oplus H^1(X_U,\Q)^-$, where $H^1(X_U,\Q)^\pm$ is the subspace where complex conjugation acts by $\pm 1$, respectively. Then it is clear that $H^1(X_U,\C)^\pm\simeq H^1(G_B(F),\cA_B(\C(\pm 1),\C)^U)$ and $C^\pm(\Omega_{X_U}^1)=H^1(X_U,\C)^\pm$.
\end{remark}

%For any $G(F)$-module $N$ and any $G(F_S)$-module $M$, we write $\cA_f^S(N,M)$ for the module of functions $f:N\times G(F^S)/(F^S)^\times\rightarrow M$ such that $f$ is linear on $N$, and there exists an open compact subgroup $U\subseteq G(F^S)$ such that $f(\cdot,\cdot U)=f(\cdot,\cdot)$. Again, $\cA_f^S(N,M)$ is equipped with commuting $G(F)$- and $G(F^S)$-actions: 
%\[
%\begin{array}{cc}
%(h\cdot f)(n,g^S)=\pi_S(h) (f(h^{-1} n,\pi^S(h^{-1}) g^S)),&h\in G(F), \\
%(h^S\cdot f)(n,g^S)=f(n,g^S h^S)),&h^S\in G(F^S);\\
%\end{array}
%\]
%where $g^S\in G(F^S)$, $n\in N$ and $f\in \cA_f^S(N,M)$. We write $\cA_f^\cQ(N,M)$ instead of $\cA_f^{\{\cQ\}}(N,M)$ and $\cA_f(N,M)$ instead of $\cA_f^\emptyset(N,M)$. Notice that $\cA_f^S(\Z,M)=\cA_f^S(M)$.

Let $H\subseteq \mathfrak{H}$ % (or $H\subseteq (\C\setminus\R)\cup(\Q\cup\infty)$ if $G=\GL_2$) 
be a $G_B(F)^+$-subset. 
Write $\Delta_H=\Z[H]$, equipped with the natural degree morphism $\deg:\Delta_H\rightarrow\Z$. If $\Delta_H^0$ is the kernel of $\deg$, then both $\Delta_H$ and $\Delta_H^0$ are $G_B(F)^+$-modules. Moreover, we have the exact sequence
\begin{equation}\label{exseqDelta}
0\longrightarrow\cA_B^S(M)\stackrel{\deg^\ast}{\longrightarrow}\cA_B^S(\Delta_H,M)\longrightarrow\cA_B^S(\Delta_H^0,M)\longrightarrow 0.
\end{equation}
%We denote by $\Z(-1)$ the group $\Z$ equipped with the action of $G(F)$ given by the character $g\mapsto{\rm sign}(\det g)$. The natural $G(F)$-morphism
%$\deg^-:\Delta_H\rightarrow\Z(-1)$, $\deg^-(g_\infty i)={\rm sign}(\det g_\infty )$ provides a similar exact sequence
%\begin{equation}\label{exseqDelta2}
%0\longrightarrow\cA_f^S(\Z(-1),M)\stackrel{(\deg^-)^\ast}{\longrightarrow}\cA_f^S(\Delta_H,M)\longrightarrow\cA_f^S(\Delta_H^{0-},M)\longrightarrow 0.
%\end{equation}
%where $\Delta^{0-}_H\subset \Delta_H$ is the kernel of $\deg^-$.

\begin{lemma}\label{pairings}
The well defined $G_B(F)^+$-equivariant morphism
\[
 ev^\pm:\cA^\sigma(I^\pm,\C)\longrightarrow \cA_B(\Delta_H,\C);\qquad ev^\pm(\phi)(g)(\tau)=\frac{1}{\pi}\phi(f_0)(\tau,g),
\]
makes the following diagram commutative
\begin{equation}\label{comdiagev}
\xymatrix{
0\ar[r]& \cA_B(\C(\pm 1),\C)\ar[r]^{\iota^\pm}\ar[d]_{Id}& \cA^\sigma(I^\pm,\C)\ar[r]^{pr^\pm}\ar[d]_{ev^\pm}&\cA^\sigma(\cD,\C)\ar[r]\ar[d]_{ev^\pm_0}& 0\\
0\ar[r]& \cA_B(\C)\ar[r]^{\deg^\ast}& \cA_B(\Delta_H,\C)\ar[r]&\cA_B(\Delta_H^{0},\C)\ar[r]& 0,
}
\end{equation}
where the map $Id:\cA_B(\C(\pm 1),\C)\rightarrow \cA_B(\C)$ is the natural identification as $G_B(F)^+$-modules.
\end{lemma}
\begin{proof}
We compute that
\begin{equation*}
ev^\pm(\iota^\pm(f))(g)(\tau)=\frac{1}{\pi}\iota^\pm(f)(f_0)(\tau,g)=\frac{1}{\pi} f(g)\left(\int_{0}^\pi d\theta \right)=\deg^\ast(f)(g)(\tau).
\end{equation*}
Hence, $\deg^\ast=ev^\pm\circ\iota^\pm$. The existence of $ev^\pm_0$ follows from a diagram chasing.
\end{proof}

The corresponding morphism in cohomology
\[
ev^\pm_0:H^0(G_B(F),\cA^\sigma(\cD,\C)^U)\longrightarrow H^0(G_B(F)^+,\cA_B(\Delta_H^0,\C)^U),
\]
has also a geometric interpretation. Indeed, for any $m=\sum_in_i\tau_i\in\Delta_H^0$ and $g\in G_B(\hat F)$, we have the divisor $(m,gU)=\sum_in_i(\tau_i,gU)\in \Div^0(Y_U)$. By \eqref{primitive}, we have that 
\begin{equation}\label{IntEv0}
ev^\pm_0(\phi)(g)(m)=\int_{(m,gU)}(\omega_\phi\pm\bar\omega_\phi).
\end{equation} 
Hence, $ev^-_0$ and $ev^+_0$ define the real and imaginary part, respectively, of the image of $(m,gU)\in\Div^0(Y_U)$ under the Abel-Jacobi map.

\subsection{Multiplicity one}\label{MultOne}
Let $\Pi$ be an automorphic parallel weight 2 representation of $G_B(\A_F)$, let $\Pi^S$ be its restriction to $G_B(\hat F^S)$, and let $L_\Pi$ be its field of definition. Thus, there is a smooth irreducible $L_\Pi$-representation $V$ of $G_B(\hat F^S)$, such that $\Pi^S\simeq V\otimes_L\C$. For an extension $L/L_\Pi$ and a smooth semi-simple $L$-representation $W$ of $G_B(\hat F^S)$, we write $W_\Pi$ for $\Hom_{G_B(\hat F^S)}(\Pi^S,W):=\Hom_{G_B(\hat F^S)}(V\otimes_{L_\Pi}L,W)$. 

\begin{definition}
A $G_B(\hat F^S)$-representation $W$ over $L$ is of \emph{automorphic type} if $W$ is smooth and semi-simple and the only irreducible subrepresentations of $W$ are either the one-dimensional representations or the $G_B(\hat F^S)$-representations over $L$, attached to automorphic representations of $G_B(\A_F)$ with parallel weight 2 .
\end{definition}

Let $W$ be a $G_B(\hat F^S)$-representation over $L$ of automorphic type.
By strong multiplicity one, $W_\Pi$ is independent of the set $S$ in the following sense: if $S'\supset S$ and $U_{S'}=\prod_{v\in S'\setminus S}U_v$ is the open subgroup of $G_B(F_{S'})$ such that $\dim(\Pi_{S'}^{U_{S'}})=1$, then we have
\[
\Hom_{G_B(\hat F^{S'})}(\Pi^{S'},W^{U_{S'}})=\Hom_{G_B(\hat F^S)}(\Pi^S,W).
\]

\begin{proposition}\label{1-dim}
Assume that the quaternion algebra $B/F$ is either totally definite or splits at a single archimedean place. Then, for any extension $L/L_\Pi$, the $G_B(\hat F)$-representation
$H^q(G_B(F),\cA_B(L))$ is of automorphic type for all $q\in \Z$. Moreover,
\begin{eqnarray*}
H^0(G_B(F),\cA_B(L))_{\Pi} \simeq L\;\;\mbox{and}&H^k(G_B(F),\cA_B(L))_{\Pi}=0\;(k\neq 0);&\mbox{if $B$ totally definite}\\
H^1(G_B(F),\cA_B(L))_{\Pi} \simeq L\;\;\mbox{and}&H^k(G_B(F),\cA_B(L))_{\Pi}=0\;(k\neq 1);&\mbox{otherwise}.
\end{eqnarray*}
\end{proposition}
\begin{proof}
%We divide this proof into two cases:
%{\bf Definite case}: Assume that $G$ is the multiplicative group of a definite quaternion algebra.

Let $U\subset G_B(\hat F)$ be an open compact subgroup.

In the totally definite case, the double coset space $X_U=G_B(F)\backslash G_B(\hat F)/U$ is a finite set of points. The group $H^k(G_B(F),\cA_B(L))^U$ can be identified with the singular cohomology of $X_U$ with coefficients in $L$. Thus, we deduce that $H^k(G_B(F),\cA_B(L))=0$ for $k>0$.
Moreover, $H^0(G_B(F),\cA_B(L))$ is in correspondence with the set of modular forms $\phi:G_B(F)\backslash G_B(\A_F)\slash UG_B(F_\infty)\rightarrow{\bf C}$, for some open compact subgroup $U\subset G_B(\hat F)$. This proves that $H^k(G_B(F),\cA_B(L))$ is of automorphic type and, by multiplicity one, $H^0(G_B(F),\cA_B(L))_{\Pi}=L$.

In case $B$ splits at a single archimedean place $\sigma$, %the above description of the Shimura curve $X_U$ shows that 
the groups $H^k(G_B(F)^+,\cA_B(L)^U)$ can be identified with the singular cohomology of the Shimura curve $X_U$ with coefficients in $L$. 
Thus, $H^k(G_B(F)^+,\cA_B(L))=0$, for $k>2$, $H^0(G_B(F)^+,\cA_B(L))$ and $H^2(G_B(F)^+,\cA_B(L))$ contain only one-dimensional irreducible subrepresentations and, by Eichler-Shimura, $H^1(G_B(F)^+,\cA_B(L))_\Pi=L\oplus L$. 
Moreover, the groups $H^k(G_B(F)^+,\cA_B(L))$ come equipped with an action of $G_B(F)/G_B(F)^+\simeq\Z/2\Z$ given by conjugation. We can identify $H^k(G_B(F),\cA_B(L))$ as the elements of $H^k(G_B(F)^+,\cA_B(L))$ fixed by $G_B(F)/G_B(F)^+$. This implies that $H^k(G_B(F),\cA_B(L))$ is of automorphic type and $H^1(G_B(F),\cA_B(L))_\Pi\simeq L$.
\end{proof}

\section{$p$-adic global distributions and measures}\label{globdistmeas}

%As in the previous section,
%$F$ will be a totally real number field. Fix a prime $\cP$ with uniformizer $\varpi$ and cardinality of the residue field $q$. We will denote by $\cO_\cP$ the integer ring of $F_\cP$.

\subsection{Definite anticyclotomic distributions}\label{padicdist}

%Throughout this section, $E/F$ will be either a totally imaginary quadratic extension or the trivial extension $E=F\times F$. We will denote by $T$ the torus $E^\times/F^\times$, likewise, $\hat T:=\hat E^\times/\hat F^\times$, $\A_T:=\A_E^\times/\A_F^\times$, $T_v=E_v^\times/F_v^\times$,  and $T^v=\prod'_{w\neq v}E_w^\times/F_w^\times$, where $\prod'$ denotes the restricted product with respect to $\cO_{E,w}^\times/\cO_{F,w}^\times$.
Let $\cP$ be a prime ideal of $F$ above $p$ write $\cO_\cP:=\cO_{F_\cP}$, and denote by $\Sigma_\infty$ the set of infinite places of $F$. Fixing an open subgroup $\Gamma\subseteq T(F^\cP)$, we shall consider the $p$-adic abelian extension of $K$ associated with $\Gamma$, namely, the maximal abelian $p$-adic anticyclotomic Galois group $\cG_{K,\cP}^\Gamma$ such that the Artin map $\rho_A:T(\A_F)/T(F)\rightarrow\cG_{K,\cP}^\Gamma$ factors through $\Gamma$.

\begin{remark}
If, instead of anticyclotomic extensions of $K$, we want to consider cyclotomic extensions of $F$, we only have to consider the trivial extension $K=F\times F$ (Notice that in this case $T(F)=F^\times$). The formalism on the construction of the cyclotomic $p$-adic $L$-function is completely analogous (see \cite{Spiess}).
\end{remark}

Let $B/F$ be a quaternion algebra that splits at $\cP$.
Let $(\Pi,V_\Pi)$ be an automorphic representation of $G_B(\A_F)$, and let us identify $V_\Pi$ with a subrepresentation of $H^0(G_B(F),\cA(\C))$. Let also assume that $(\Pi_\cP,V_{\Pi_\cP})\simeq(\pi_\alpha^\C,V_\alpha^\C)$, where $\alpha=\pm 1$ or $|\alpha|=q$. 
By the \emph{Tensor Product Theorem} \cite[Theorem 3.3.3]{Bump}, $V_\Pi\simeq \bigotimes'_vV_{\Pi_v}$.
Let $U\subset G_B(\hat F)$ be an open compact subgroup satisfying $\dim_\C (\bigotimes_{v\nmid\infty}'V_{\Pi_v})^U=1$ and $U_\cP=\GL_2(\cO_\cP)$ or $K_0(\varpi)$. Write $U=U_\cP\times U^\cP$. Again by the \emph{Tensor Product Theorem}, we have a $G(F_\cP)$-equivariant morphism  
\begin{equation}\label{eqpiindpi}
V_\alpha^\C\simeq V_{\Pi_\cP}\longrightarrow V_\Pi^{U^\cP}.
\end{equation}
\begin{remark}\label{unicitymap}
Note that the above morphism is not unique in general. Nevertheless, since $\dim_\C (\bigotimes_{v\nmid\infty}'V_{\Pi_v})^U=1$, this morphism is unique up to constant if $V_{\Pi_v}$ is trivial for all $v\mid\infty$. This will be the case when the quaternion algebra $B$ is totally definite and the automorphic representation $\Pi$ is of parallel weight 2.
\end{remark}

Assume that for every prime dividing the discriminant of $B$ the extension $K/F$ is non-split. This implies that there exists a (fixed) embedding $\imath:K\hookrightarrow B$. Hence, we can consider $T$ inside $G_B$ by means of $\imath$. Let us also assume that $T(\hat F^\cP)\cap U^\cP=\Gamma$.
We choose $P_\cP\subset G_B(F_\cP)$, some conjugate of the group of upper triangular matrices, in such a way that $T(F_\cP)\not\subset P_\cP$. Let $R\subset \C$ be any ring (endowed with the discrete topology) such that $\alpha\in R^\times$. The map $\delta_T$ of \eqref{deltaT} provides a $T(F_\cP)$-equivariant $R$-module morphism 
%$\delta_{T}:C_c(T(F_\cP),R)\rightarrow V_\alpha^R$,
%and it induces the $T(F_\cP)$-equivariant morphism
\begin{equation}\label{deltaarep}
\delta:C_c(T(F_\cP),R)\stackrel{\delta_{T}}{\longrightarrow}V_\alpha^R\subseteq V_{\Pi_\cP} \stackrel{\eqref{eqpiindpi}}{\longrightarrow} V_\Pi^{U^\cP}.
\end{equation}
Finally, we define the distribution $\mu_{K,\cP}^\Pi$ on $\cG_{K,\cP}^\Gamma$ as follows: for $g\in C(\cG_{K,\cP}^\Gamma,\C)=C_c(\cG_{K,\cP}^\Gamma,\C)_0$,
\[
\int_{\cG_{K,\cP}^\Gamma}g(\gamma)d\mu_{K,\cP}^\Pi(\gamma):=[H_\cP:H]\int_{T(\A_F)/T(F)}g(\rho_A(x))\delta(1_H)(x)d^\times x,
\] 
where $H\subset T(F_\cP)$ is an open subgroup small enough that $g\circ\rho_A$ is $H$-invariant, $1_H\in C_c(T(F_\cP),\Z)$ is the characteristic function of $H$, and $H_\cP$ is the maximal open subgroup of $T(F_\cP)$. A simple computation shows that the above definition does not depend on the choice of $H$. %Indeed, for any other $H'\subseteq H$, we have $1_H=\sum_{h\in H/H'}1_{hH'}=\sum_{h\in H/H'}h\ast 1_{H'}$, hence (if $g\circ\rho_A$ is $H$-invariant)
%\begin{eqnarray*}
%\int_{\cG_{E,\cP}^\Gamma}g(\gamma)d\bar\mu_{E,\cP}^\Pi(\gamma)&=&[H_\cP:H]\int_{\A_T/T}g(\rho_A(x))\delta\left(\sum_{h\in H/H'}h\ast 1_{H'}\right)(x)d^\times x\\
%&=&[H_\cP:H]\sum_{h\in H/H'}\int_{\A_T/T}g(\rho_A(x))\delta(1_{H'})(xh)d^\times x\\
%&=&[H_\cP:H]\sum_{h\in H/H'}\int_{\A_T/T}g(\rho_A(xh^{-1}))\delta(1_{H'})(x)d^\times x\\
%&=&[H_\cP:H][H:H']\int_{\A_T/T}g(\rho_A(x))\delta(1_{H'})(x)d^\times x,
%\end{eqnarray*}
%where the second equality is obtained from the $T_\cP$-equivariance of $\delta$.

In the cyclotomic setting, Leopoldt's conjecture predicts that the maximal abelian extension of $F$ unramified outside $\infty$ and $p$ is isomorphic to $\Z_p$ up to a finite group.
The following lemma describes the free part of the Galois group $\cG_{K,\cP}^\Gamma$, and establishes a big difference between the cyclotomic and the anticyclotomic setting:
\begin{lemma}\label{lemstructGp}
Let $G^{\Gamma}_{K,\cP}$ be the torsion subgroup of $\cG^{\Gamma}_{K,\cP}$. Then $\cG^{\Gamma}_{K,\cP}=G^{\Gamma}_{K,\cP}\times\cG_{K,\cP}$, where
$\cG_{K,\cP}=\Z_p^{[F_\cP:\Q_p]}$. 
\end{lemma}
\begin{remark}
Observe that $\cG_{K,\cP}$ does not depend on $\Gamma$.
\end{remark}
\begin{proof}
First note that the Artin map factors through $T(\hat F)/\Gamma T(F)$. Since $\Gamma\subseteq T(\hat F^\cP)$ is open, $T(\hat F)/\Gamma=T(F_\cP)\times T(\hat F^\cP)/\Gamma$ where $T(\hat F^\cP)/\Gamma$ is discrete.
Moreover,
\begin{equation}\label{unitsEunitsF}
\rank_\Z(\cO_K^\times)=\frac{1}{2}[K:\Q]-1=[F:\Q]-1=\rank_\Z(\cO_F^\times).
\end{equation}
Hence $T(F)$ is a discrete subgroup of $T(\hat F)$. We deduce that the $\Z_p$-rank of $\cG^{\Gamma}_{K,\cP}$ coincides with the $\Z_p$-rank of $\cO_{K,\cP}^\times/\cO_{F,\cP}^\times$, which is clearly $[F_\cP:\Q_p]$.
\end{proof}

From now on, we shall consider the distributions $\mu_{K,\cP}^\Pi$ restricted to functions supported on $\cG_{K,\cP}$. 

\subsubsection{Waldspurger formula and interpolation properties}\label{WaldFormula}

%Let $E$ be a quadratic extension of $F$, as above. Again, we denote by $T$ the torus $E^\times/F^\times$.
% and we also denote by $\A_{T}$, $\hat T$ and $T_v$ its ring of adeles, finite adeles and its localization at a place $v$ of $F$, respectively.
Let us consider $(\pi,V_\pi)$, an irreducible cuspidal automorphic representation in $L^2(\GL_2(F)\backslash \GL_2(\A_F))$ with trivial central character and parallel weight $2$, and let $\chi$ be a finite character of $T(\A_F)/T(F)$. Write $L(s,\pi_K,\chi)$ for the Rankin-Selberg $L$-series  associated with $\pi$, $\chi$ and $K$. We also consider the finite sets of places of $F$
\begin{eqnarray}
\Sigma_\pi^\chi:&=&\{v: \;\dim(\Hom_{T(F_v)}(\pi_v\otimes\chi_v,\C))=0\},\\
\hat\Sigma_\pi^\chi:&=&\{v\nmid \infty: \;\dim(\Hom_{T(F_v)}(\pi_v\otimes\chi_v,\C))=0\}.
\end{eqnarray}
Let $B/F$ be a quaternion algebra with ramification set $\Sigma_B$. For any place $v$ of $F$, denote by $(\pi_v^{B},V_{\pi_v^{B}})$ the Jacquet-Langlands lift of the local representation $\pi_v$ on $G_B(F_v)$ (if $v\not\in\Sigma_B$, then $\pi_v^{B}=\pi_v$). 
We also consider the global Jacquet-Langlands lift $(\pi^{B},V_{\pi^{B}})$ of $\pi$ on $G_B(\A_F)$. Invoking again the \emph{Tensor Product Theorem}, $(\pi^{B},V_{\pi^{B}})\simeq(\bigotimes'_v\pi_v^{B},\bigotimes'_vV_{\pi_v^{B}})$. We define the following pairing on $\pi^{B}$
\[
\beta_{\pi^{B},\chi}=\frac{\xi(2)}{L(1,\eta)L(1,\pi,ad)}\prod_v\alpha_{\pi_v^{B},\chi_v},
\]
where 
$\eta:\A_F^\times/F^\times\Norm_{K/F}(\A_K^\times)\rightarrow\{\pm 1\} $
is the quadratic character associated with the extension $K/F$ and $\alpha_{\pi_v^{B}}$ are the local pairings defined in \S \ref{TorInn}.
The pairing $\beta_{\pi^{B},\chi}$ is well-defined by Proposition \ref{propWaldloc}. The following result is due to Waldspurger \cite{Walds}:
\begin{theorem}[Waldspurger]\label{WaldThm}
Assume that $\phi\in V_{\pi^{B}}$ corresponds to $\otimes_v f_v\in\bigotimes'_vV_{\pi_v^{B}}$ under the isomorphism $\pi^{B}\simeq\bigotimes'_v\pi_v^{B}$, then
\[
\left |\int_{T(\A_F)/T(F)}\chi(t)\phi(t)d^\times t\right |^2=\frac{1}{2}L(1/2,\pi_K,\chi)L(1,\eta)\beta_{\pi^{B},\chi}(\otimes_v f_v,\otimes_v f_v).
\]
%where $\langle\cdot,\cdot\rangle$ is the invariant Peterson inner product.
Moreover, the above expression is $0$ unless $\Sigma_B=\Sigma_\pi^\chi$.
\end{theorem}

By Saito-Tunnel (Proposition \ref{Saito-Tunnel}), if $K$ does not split at some place $\sigma\mid\infty$, then $\Sigma_\infty\subseteq \Sigma_\pi^\chi$ for every finite character $\chi$. Since the $\C$-vector space in $C(\cG_{K,\cP},\C)$ is generated by the set of finite characters with trivial component outside $\cP$, we deduce the following direct consequence of the Waldspurger formula:
\begin{corollary}\label{consWald}
The condition $\Sigma_\pi^1\setminus\{\cP\}=\Sigma_B$ is necessary for  the Jacquet-Langlands lift  $\pi^{B}$ and the distribution $\mu_{K,\cP}^{\pi^{B}}$ to be non-zero. In particular, $B$ is totally definite and $\#(\hat\Sigma_\pi^1\setminus\{\cP\})+[F:\Q]$ is even.
\end{corollary}

\begin{remark}
The above result implies that whenever $\#(\hat\Sigma_\pi^1\setminus\{\cP\})+[F:\Q]$ is odd we will not be able to construct non-zero anticyclotomic $p$-adic distributions of $\cG_{K,\cP}$ using the previous procedure. In \S \ref{DerAntDist}, we will explain a different procedure to deal with this situation.
\end{remark}

%From now on, we assume that $E/F$ is a totally imaginary quadratic extension, thus, we deal with the anticyclotomic situation. For a complete description using similar techniques to the cyclotomic case (interpolation property and exceptional zero phenomenon) see \cite{Spiess}.

\begin{definition}\label{defdef}
Let $(\pi,V_\pi)$ be an irreducible cuspidal automorphic representation in $L^2(\GL_2(F)\backslash \GL_2(\A_F))$ with trivial central character and parallel weight $2$, and assume that $(\pi_\cP,V_{\pi_\cP})\simeq(\pi_\alpha^\C,V_\alpha^\C)$, where $\alpha=\pm 1$ or $|\alpha|^2=q$. Under the assumption that $\#(\hat\Sigma_\pi^1\setminus\{\cP\})+[F:\Q]$ is even, we define $\mu_{K,\cP}^{\rm def}$ to be the distribution $\mu_{K,\cP}^{\pi^{D}}$ of $\cG_{K,\cP}$, where $\pi^{D}$ is the corresponding Jacquet-Langlands automorphic representation attached to the totally definite quaternion algebra $D/F$ with ramification set $(\hat\Sigma_\pi^1\setminus\{\cP\})\cup\Sigma_\infty$.
\end{definition}

\begin{theorem}[Interpolation Property]\label{intprop1}
There is a non-zero constant $C_K$ depending on $K/F$ such that, 
for any continuous character $\chi:\cG_{K,\cP}\rightarrow\C^\times$,
\[
\left|\int_{\cG_{K,\cP}}\chi(\gamma)d\mu_{K,\cP}^{\rm def}(\gamma)\right|^2=C_K C(\pi_\cP,\chi_\cP) \frac{L(1/2,\pi_K,\chi)}{L(1,\pi,ad)},
\]
where
\[
C(\pi_\cP,\chi_\cP)=\left\{\begin{array}{ll}
					  \frac{L(1,\pi_\cP,ad)}{L(1/2,\pi_\cP,\chi_\cP)}, & |\alpha|^2=q,\\
					 \frac{1}{L(-1/2,\pi_\cP,\chi_\cP)}, & \alpha=\pm 1,\;\chi_\cP\mid_{\cO_{K_\cP}^\times}=1,\\
					 \frac{q^{n_\chi}}{L(1/2,\pi_\cP,\chi_\cP)}, & \alpha=\pm 1,\;\chi_\cP\mid_{\cO_{K_\cP}^\times}\neq1,
					\end{array}\right.
\]
and $n_\chi$ is the conductor of $\chi_\cP$.
\end{theorem}
\begin{proof}
%Let $\pi^{JL}$ be the Jacquet-Langlands automorphic representation attached to the totally definite quaternion algebra with ramification set $\Sigma_\pi^1\setminus\{\cP\}$. Let $G$ be the multiplicative group of such quaternion algebra. 
In order to define $\mu_{K,\cP}^{\pi^{D}}$, we have to choose a $G_D(F_\cP)$-equivariant morphism $V_{\alpha}^\C\rightarrow V_{\pi^{D}}^{U^\cP}$, which is unique up to constant by Remark \ref{unicitymap}. 
Assume that $\delta(1_H)$ corresponds to $\bigotimes_v f_v\in\bigotimes'_v V_{\pi^{D}_v}$ under the isomorphism of the \emph{Tensor Product Theorem}. Thus, $f_v\in V_{\pi^{D}_v}^{U_v}$ is a generator of $V_{\pi^{D}_v}$ as $\C[G_D(F_v)]$-module for all $v\nmid \infty\cP$, %since $\dim_\C(\bigotimes_{v\nmid\infty}V_{\pi^{JL}_v})^{U^\cP}=1$, 
and $f_\cP=\delta_{T}(1_H)$. 
We compute that
\begin{eqnarray*}
\left|\int_{\cG_{K,\cP}}\chi(\gamma)d\mu_{K,\cP}^{\rm def}(\gamma)\right|^2&=&[H_\cP:H]^2\left|\int_{T(\A_F)/T(F)}\chi(\rho_A(t))\delta(1_H)(t)d^\times t\right|^2\\
&=&C_\phi[H_\cP:H]^2\frac{L(1/2,\pi_K,\chi)}{L(1,\pi,ad)}\alpha_{\pi_\cP,\chi_\cP}(\delta_{T}(1_H),\delta_{T}(1_H)),
\end{eqnarray*}
where $C_\phi=\frac{1}{2}\xi(2)\prod_{v\neq \cP}\alpha_{\pi_v^{D},\chi_v}(f_v,f_v)$. Since $\chi$ is a character of $\cG_{K,\cP}$, the components $\chi_v$ at $v\neq \cP$ are trivial. By Proposition \ref{Saito-Tunnel}, the choice of the quaternion algebra ensures that $\alpha_{\pi_v^{D},\chi_v}(f_v,f_v)\neq 0$ for $v\neq \cP$, thus $C_\phi\neq 0$. By Corollary \ref{Corcalclocal},
\[
\alpha_{\pi_\cP,\chi_\cP}(\delta_{T}(1_H),\delta_{T}(1_H))=k_{T}C(\pi_\cP,\chi_\cP){\rm vol}(H)^2,
\]
for some non-zero constant $k_{T}$ depending on $T(F_\cP)$. Finally, the result follows from the fact that ${\rm vol}(H)[H_\cP:H]={\rm vol}(H_\cP)$ is a fixed constant.
\end{proof}

Sometimes it is more convenient to compute the square of the integral rather than the square of its absolute value. Although it is not complicated to compute such square, we have preferred to compute the absolute value in order to have more analogy with the indefinite setting.
Recent results by P-C. Hung compute explicitly such square when $U=\Gamma_0(\mathfrak{n})$ and we have a fixed morphism \eqref{eqpiindpi} provided by a given newform (see \cite{Hung1} for more details).

\subsection{Cohomological interpretation of $\mu_{K,\cP}^{\rm def}$}\label{CoInt}

%The $\Z$-rank of the units of $E$ is
%\[
%\rank_\Z(\cO_E^\times)=\frac{1}{2}[E:\Q]-1=[F:\Q]-1=\rank_\Z(\cO_F^\times).
%\]
We showed in the proof of Lemma \ref{lemstructGp} that $T(F)$ is a discrete subgroup in $T(\hat F)$.
The Artin map provides an isomorphism $\cG_{K,\cP}^\Gamma\simeq T(\hat F)/T(F)\Gamma$, for some open compact subgroup $\Gamma\subset T(\hat F^\cP)$. Hence, for any Hausdorff topological ring $R$, we have an injection 
\begin{equation}\label{defpart}
\partial:C(\cG_{K,\cP},R)\hookrightarrow H_0(T(F),C_c(T(\hat F)/\Gamma,R)).
\end{equation}
Moreover, the natural map
\begin{equation}\label{decfunct}
C_c(T(F_\cP),R)\otimes_R C_c(T(\hat F^\cP)/\Gamma,R)_0\stackrel{\simeq}{\longrightarrow}C_c(T(\hat F)/\Gamma,R);\quad (f_\cP\otimes f^\cP)\longmapsto f_\cP\cdot f^\cP,
\end{equation}
is an isomorphism
and, since $T(\hat F^\cP)/\Gamma$ is discrete, any $f^\cP\in C_c(T(\hat F^\cP)/\Gamma,R)_0$ can be written as a finite combination of characteristic functions of $T(\hat F^\cP)/\Gamma$.

Since the representation $(\pi^{D},V_{\pi^{D}})$ satisfies $(\pi^{D}_\cP,V_{\pi_\cP^{D}})\simeq (\pi_\alpha^\C,V_\alpha^\C)$, equation \eqref{eqpiindpi} provides an element $\phi\in\Hom_{G_D(F_\cP)}(V_\alpha^\C,\cA_D(\C)^{U^\cP})=\cA_D^\cP(V_\alpha^\C,\C)^{U^\cP}$, which is unique up to constant by Remark \ref{unicitymap}. Since $\pi^{D}$ is an automorphic representation, $\phi\in H^0(G_D(F),\cA_D^\cP(V_\alpha^\C,\C)^{U^\cP})$. 
\begin{lemma}\label{lemtenpro} 
We have a natural isomorphism 
\[
H^k\left(G_D(F),\cA_D^\cP\left(\Ind_{U_p}^{G_D(F_\cP)}1_{R}/\cI,R\right)\right)\simeq H^k(G_D(F),\cA_D^\cP\left(\Ind_{U_p}^{G_D(F_\cP)}1_{R'}/\cI,R'\right)\otimes_{R'}R),
\]
for any ideal $\cI\subseteq \Ind_{U_p}^{G_D(F_\cP)}1_{R'}$, any $k\in \N$ and any flat ring homomorphism $R'\rightarrow R$.
\end{lemma}
\begin{proof}
Let $\phi\in H^k\left(G_D(F),\cA_D^\cP\left(\Ind_{U_p}^{G_D(F_\cP)}1_{R}/\cI,R\right)\right)$. By Remark \ref{AvsHom},
\[
\cA_D^\cP\left(\Ind_{U_p}^{G_D(F_\cP)}1_{R}/\cI,R\right)\simeq\Hom_{G_D(F_\cP)}\left(\Ind_{U_p}^{G_D(F_\cP)}1_{R}/\cI,\cA_D(R)\right)
\]
Moreover, $\phi(1_{U_p})\in H^k(G_D(F),\cA_D(R'))\otimes_{R'}R$ by Lemma \ref{lemmaext}, hence the result follows.
\end{proof}

Since $V_\alpha^{L}=\Ind_{U_p}^{G_D(F_\cP)}1_{L}/(\cT-a)$ ($L=\C$ or $\bar\Q$) by Lemma \ref{lemlattices},
the above result implies that we can assume that $\phi\in H^0(G_D(F),\cA_D^\cP(V_\alpha^{\bar\Q},\bar\Q)^{U^\cP})$.
The $T(F_\cP)$-equivariant $\bar\Q$-module morphism 
\begin{equation*}
\delta_{T}:C_c(T(F_\cP),\bar\Q)\longrightarrow V_\alpha^{\bar\Q},
\end{equation*}
provides a $T(F)$-equivariant morphism
\begin{equation}\label{kappa}
\kappa: \cA_D^\cP(V_\alpha^{\bar\Q},\C_p)^{U^\cP}\longrightarrow C_c(T(\hat F)/\Gamma,\C_p)_0^\vee;\quad  \langle\kappa(\phi),f_\cP\otimes f^\cP\rangle:=\sum_{x\in T(\hat F^\cP)/\Gamma}f^\cP(x)\phi(x)(\delta_\cP(f_\cP)).
\end{equation}
We deduce that 
\[
\int_{\cG_{K,\cP}}g(\gamma)d\mu_{K,\cP}^{\rm def}(\gamma)=\kappa({\rm res}\phi)\cap\partial g,%\quad i=I,II,
\]
where ${\rm res}: H^0(G_D(F),\cA_D^\cP(V_\alpha^{\bar\Q},\bar\Q))\rightarrow H^0(T(F),\cA_D^\cP(V_\alpha^{\bar\Q},\bar\Q))$ is the restriction map.

\subsection{Heegner points and Gross-Zagier-Zhang formula}\label{ShimCM}

Let $B/F$ now be a quaternion algebra that splits at a single archimedean place $\sigma$, and admitting an embedding $K\hookrightarrow B$.

Assume that the automorphic representation $\pi$ admits a Jacquet-Langlands lift $\pi^{B}$ to $G_B$. Since $\pi$ has parallel weight 2, $\dim_\C H^{0}(G_B(F),\cA^\sigma(\cD,\C))_{\pi^{B}}=1$. Hence, for some open compact subgroup $U\subset G_B(\hat F)$, there exists $\phi\in H^{0}(G_B(F),\cA^\sigma(\cD,\C))^U$ that generates $\pi^{B}\mid_{G_B(\hat F)}$. Moreover, $\phi$ is unique up to constant.

We have explained how $\phi$ defines an holomorphic differential form $\omega_\phi\in\Omega^1_{Y_U/\C}$, and since $\pi$ is cuspidal, in fact $\omega_\phi\in\Omega^1_{X_U/\C}$. 
Let $\Omega_{X_U}^1$ be the cotangent space of $X_U/F$, and let $\Omega^1_X$ be the $\Q$-vector space
\[
\Omega_X^1=\varinjlim_U \Omega_{X_U}^1\subset \varinjlim_U \Omega_{Y_U/\C}^1\simeq H^0(G_B(F),\cA^\sigma(\cD,\C)).
\]
This is a $\Q$-representation of $G_B(\hat F)$ that decomposes as $\Omega^1_X=\oplus \rho$, where $\rho$ are irreducible $\Q$-representations. %Let $L$ be the field of definition of $\pi^{JL}$, then 
There exists an irreducible $L_\pi$-representation $\rho_\pi\subset\Omega^1_X$ such that
\[
\rho_\pi\otimes_{L_\pi}\C\simeq\pi^{B}\mid_{G_B(\hat F)},\qquad\rho_\pi\otimes_\Q\C=\oplus_{\tau\in\Gal(L_\pi/\Q)} {^{\tau}}\pi^{B}\mid_{G_B(\hat F)},
\] 
where ${^{\tau}}\pi^{B}=\rho_\pi\otimes_{\tau(L_\pi)}\C\subset H^0(G_B(F),\cA^\sigma(\cD,\C))$.
After scaling conveniently, we can assume that $\omega_\phi\in\rho_\pi$. The above equation allows to embed $\rho_\pi$ in different $\C$-representations $ {^{\tau}}\pi^{B}\mid_{G_B(\hat F)}$. For all $\tau\in\Gal(L_\pi/\Q)$, we write $^\tau\omega_\phi$ for the image of $\omega_\phi$.
The $\Z$-module
\[
\Lambda_\pi:=\left\{\left(\int_c {^\tau}\omega_\phi\right)_{\tau\in\Gal(L/\Q)}, \;c\in H_1(X_U,\Z)\right\}\subset L\otimes_\Q \C\simeq\C^{[L:\Q]},
\]
defines a lattice. Moreover, the complex torus $(L\otimes_\Q\C)/\Lambda_\pi$ defines an abelian variety of $\GL_2$-type $A$ defined over $F$ such that $\End^0(A)=L_\pi$. The Abel-Jacobi map 
\begin{equation}\label{AJmap}
AJ:\Div^0(X_U)\longrightarrow A(\C);\quad m\longmapsto \left(\int_m {^\tau}\omega_\phi\right)_{\tau\in\Gal(L_\pi/\Q)}\mbox{ mod }\Lambda_\pi,
\end{equation}
provides, in fact, a morphism $\Jac(X_U)\rightarrow A$ defined over $F$. Since we are interested in points up to torsion, write $A^0(M)=A(M)\otimes_{\End(A)}L_\pi$, for any field $M$. Hence $A^0(\C)\simeq(L_\pi\otimes_\Q\C)/\Lambda_\pi^0$, where $\Lambda_\pi^0=\Lambda_\pi\otimes_\Z\Q$.

We write $\Delta=\Z[\mathfrak{H}]$ and $\Delta^0=\ker(\deg:\Delta\rightarrow\Z)$.  Let us consider the morphism $\lambda_\pi:\cA_B(\Delta^0,\C)\longrightarrow\cA_B(\Delta^0,A^0(\C))$ given by
\[
\lambda_\pi\psi(m,g)=1\otimes\psi(m,g)\mod\Lambda_\pi^0\in (L\otimes_\Q\C)/\Lambda^0_\pi\simeq A^0(\C).
\]
Recall the morphisms $ev^\pm_0:\cA^\sigma(\cD,\C)\rightarrow\cA_B(\Delta^0,\C)$ of Lemma \ref{pairings}. Thus, the composition of $ev_0:=ev^+_0+ev^-_0$ and $\lambda_\pi$ provides a morphism of $G_B(\hat F)$-representations over $L_\pi$:
\[
\varphi:\rho_\pi\longrightarrow H^0(G_B(F)^+,\cA_B(\Delta^0,A^0(\C))),\quad\mbox{hence}\quad\varphi\in H^0(G_B(F)^+,\cA_B(\Delta^0,A^0(\C)))_{\pi^{B}}.
\]

Let us consider the $L_\pi$-module $A^0(\C)$, and the exact sequence \eqref{exseqDelta}
\[
0\longrightarrow\cA_B(A^0(\C))\longrightarrow\cA_B(\Delta,A^0(\C))\longrightarrow\cA_B(\Delta^0,A^0(\C))\longrightarrow 0.
\]
This provides an exact sequence in cohomology
\begin{eqnarray*}
H^0(G_B(F)^+,\cA_B(A^0(\C)))_{\pi^{B}}&\longrightarrow& H^0(G_B(F)^+,\cA_B(\Delta,A^0(\C)))_{\pi^{B}}\longrightarrow\\
\longrightarrow H^0(G_B(F)^+,\cA_B(\Delta^0,A^0(\C)))_{\pi^{B}}&\stackrel{\partial}{\longrightarrow}& H^1(G_B(F)^+,\cA_B(A^0(\C)))_{\pi^{B}}.
\end{eqnarray*}

\begin{lemma}\label{lempartial0}
We have that $H^0(G_B(F)^+,\cA_B(A^0(\C)))_{\pi^{B}}=0$. Moreover, $\partial\varphi=0$.
\end{lemma}
\begin{proof}
Note that $A^0(\C)$ is a free $L_\pi$-module of infinite rank, hence the first claim follows from the fact that $H^0(G_B(F)^+,\cA_B(L_\pi))_{\pi^{B}}=0$ (proof of Proposition \ref{1-dim}). 

On the other side, we know that $\rho_\pi$ is generated by $\omega_\phi\in\rho_\pi^U$. Thus, if we prove that $\partial(\varphi(\omega_\phi))=0$ the assertion will follow. 

By Remark \ref{CCinHOMOL}, we have that 
\[
\partial(\varphi(\omega_\phi))=\lambda_\pi(\partial^+\phi+\partial^-\phi)=\lambda_\pi\left(C^+(\omega_\phi)+C^-(\omega_\phi)\right),
\]
under the identifications given in \S \ref{CohShi}. For all $c\in H_1(X_U,\Z)$,
\[
C^+(\omega_\phi)(c)+C^-(\omega_\phi)(c)=2\left(\int_c{^\tau}\omega_\phi\right)_\tau\in\Lambda_\pi.
\]
Hence, the result follows.
\end{proof}

The above Lemma implies that there exists a unique $\bar\varphi\in H^0(G_B(F)^+,\cA_B(\Delta,A^0(\C)))_{\pi^{B}}$ extending $\varphi$. Let us describe $\bar\varphi(\omega_\phi)$ geometrically in two different ways: On the one hand, 
\[
\bar\varphi(\omega_\phi):\mathfrak{H}\times G_B(\hat F)\longrightarrow X_U(\C)\longrightarrow A^0(\C) 
\]
can be described as the extension of the composition of the $F$-morphism $\iota:X_U\rightarrow\Jac(X_U)$, given by a suitable multiple of the Hodge class (resp., of the cusp at infinity if $B = \M_2(\Q)$), with the modular parametrization $\Jac(X_U)\rightarrow A$. Indeed, on the degree zero divisors,
\[
\bar\varphi(\omega_\phi)(z_1-z_2,g)=\lambda_\pi \circ (ev^-_0+ev^+_0)(\omega_\phi)(z_1-z_2,g)=\left(\int_{(z_1,gU)}^{(z_2,gU)}{^\tau}\omega_\phi\right)_\tau\mod\Lambda_\pi,
\]
by \eqref{IntEv0}. On the other hand, let $\cH_U=(\Ind_{U}^{G_B(\hat F)}1_\Z)^U$ be the Hecke algebra of compactly supported and $U$-bi-invariant functions. By Frobenius reciprocity, $\cH_U$ acts on the 1-dimensional space $\rho_\pi^U$, providing a morphism of algebras
\[
\lambda:\cH_U\longrightarrow L_\pi,\qquad \lambda(\cT)\omega_\phi=\cT\ast\omega_\phi,\mbox{ for all }\cT\in\cH_U.
\] 
let $\dL$ be a prime of $F$ such that $U_\dL\simeq\GL_2(\cO_{F_\dL})/\cO_{F_\dL}^\times$, and the class of $\dL$ is trivial in $\hat F^\times/N(U)F^+$, where $N : (B\otimes_F\hat F)^\times \rightarrow \hat F^\times$ is the norm map. Write 
\[
\cV:=T_\dL-(\ell+1)1_U\in\cH_U,\qquad T_\dL=1_{Ug_\ell U},\quad g_\ell=\left(\begin{array}{cc}\varpi_\ell&\\&1\end{array}\right),
\]
where $\omega_\ell\in\cO_{F_\dL}$ is an uniformizer and $\ell=\#(\cO_{F_\dL}/\omega_\ell)$. %It is easy to check that $\cT\in\cH^0_U$. Moreover, 
Notice that $\lambda(\cV)=a_\dL-\ell-1$, where $|a_\dL|\leq \ell^{1/2}$, hence $\lambda(\cV)\neq 0$.

%We write 
%\[
%\cH_U^0=\left\{\cT\in\cH_U; \;\int_{N^{-1}(F^+ a N(U))}\cT(g)dg=0, \quad\mbox{for all }a\in \hat F^\times\right\},
%\]
%where $N : G(\hat F) \rightarrow \hat F^\times$ is the norm map.
%Notice that both $\cH_U$ and $\cH_U^0$ act on $\Div(X_U)$.%$\hat\varphi(\omega_\phi)\in H^0(G(F),\cA(\Delta,A_\pi^0(\C))^U)$.
\begin{lemma}
For any $(z,gU)\in\Div(X_U)$,
\[
\cV\ast(z,gU)={\rm vol}(U)^{-1}\int_{G_B(\hat F)}\cV(h)(z,ghU)dh\in\Div^0(X_U).
\]
\end{lemma}
\begin{proof}
Since $B$ is a quaternion algebra that splits at a single archimedean place, the norm map induces an isomorphism $G_B(F)^+\backslash G_B(\hat F)/U\simeq F^+\backslash \hat F^\times/N(U)$. %Write $G_B(F)^+g_a U$ for the class corresponding to $F^+a N(U)$. 
We compute,
\begin{eqnarray*}
\cV\ast(z,gU)&=&{\rm vol}(U)^{-1}\sum_{[a]\in F^+\backslash \hat F^\times/N(U)}\int_{N^{-1}([a])}\cV(h)(z,ghU)dh\\
&=&\left({\rm vol}(U)^{-1}\int_{U}\cV(h)(\gamma_{gh} z)dh,gg_\ell U\right),
\end{eqnarray*}
where $\gamma_{gh}\in G_B(F)^+$ satisfies $\gamma_{gh}^{-1}gg_\ell U=ghU$. By definition, the divisor
\[
{\rm vol}(U)^{-1}\int_{U}\cV(h)(\gamma_{gh} z)dh \in\Z[\mathfrak{H}]
\]
has degree zero, hence $\cV\ast(z,gU)\in\Div^0(X_U)$.
\end{proof}
%Let $\cT\in\cH_U^0$ such that $\lambda(\cT)\neq 0$. 
Then, 
\begin{equation}\label{2nd}
\bar\varphi(\omega_\phi)(z,g)=\lambda(\cV)^{-1}AJ(\cV\ast(z,gU)), 
\end{equation}
where $AJ$ is the Abel-Jacobi map of \eqref{AJmap}. Moreover, the above description does not depend on $\cV$.
%\begin{remark}
%Such a $\cT\in\cH^0_U$ with $\lambda(\cT)\neq 0$ exists. Indeed, let $\cQ$ be a prime of $F$ such that $U_\cQ\simeq\GL_2(\cO_{F_\cQ})$, where $\cO_{F_\cQ}$ is the integer ring of $F_\cQ$, and the class of $\cQ$ is trivial in $\hat F^\times/N(U)F^+$. Write 
%\[
%\cT=T_\cQ-(q+1)1_U\in\cH_U,\qquad T_\cQ=\mbox{characteristic function of }U\left(\begin{array}{cc}\varpi_q&\\&1\end{array}\right)U,
%\]
%where $\omega_q\in\cO_{F_\cQ}$ is an uniformizer and $q=\#(\cO_{F_\cQ}/\omega_q)$. It is easy to check that $\cT\in\cH^0_U$. Moreover, $\lambda(\cT)=a_\cQ-q-1$, where $|a_\cQ|\leq q^{1/2}$, hence $\lambda(\cT)\neq 0$. 
%\end{remark}

Let $\tau_K\in\mathfrak{H}$ be the unique point such that $t\tau_K=\tau_K$, for all $t\in K^\times$. Such $\tau_K$ defines a $G_B(F)^+$-equivariant monomorphism $\Ind_{K^\times}^{G_B(F)^+}1_\Z\hookrightarrow \Delta$. Restricting $\bar\varphi$ we obtain,
\[
\Phi_T:\rho_\pi\longrightarrow H^0(G_B(F)^+,\Coind_{K^\times}^{G_B(F)^+}\cA_B(A^0(\C)))=H^0(K^\times,\cA_B(A^0(\C))).
\]
\emph{Shimura's reciprocity law} asserts that the image of $\Phi_T(\omega_\phi)$ lies in $A^0(K^{ab})$, where $K^{ab}$ is the maximal abelian extension of $K$, and the Galois action of $\Gal(K^{ab}/K)$ is given by
\[
{^{\rho_A(t)}}\Phi_T(\omega_\phi)(g)=\Phi_T(\omega_\phi)(tg),\qquad t\in \hat K^\times/K^\times,\quad g\in G_B(\hat F),
\] 
where $\rho_A:\hat K^\times/K^\times\rightarrow\Gal(K^{ab}/K)$ is the Artin map. Thus,
\[
\Phi_T\in H^0(K^\times,\cA_B(A^0(K^{ab})))_{\pi^{B}}.
\]

In analogy with \S\ref{WaldFormula}, we consider the pairing on $\pi^{B}\mid_{G_B(\hat F)}\simeq\bigotimes'_{v\nmid\infty}V_{\pi_v^{B}}$
\[
\hat\beta_{\pi^{B},\chi}=\frac{\xi(2)}{L(1,\eta)L(1,\pi,ad)}\prod_{v\nmid\infty}\alpha_{\pi_v^{B},\chi_v},
\]
and let us also consider the Neron-Tate pairing on $A$:
$$
\langle\;,\;\rangle:A(\bar\Q)\times A(\bar\Q)\longrightarrow\R.
$$
By Galois and Hecke equivariance, the Neron-Tate pairing provides a $\hat K^\times/K^\times$-invariant Hermitian pairing $\langle \;,\;\rangle$ on $A^0(K^{ab})\otimes_{L_\pi}\C$. Write $|Q |^2:=\langle Q,Q\rangle$ and $\hat\Sigma_B=\{v\in\Sigma_B,\;v\nmid \infty\}$. 

\begin{theorem}[Gross-Zagier-Zhang]\label{GrZaZh}
For any finite character $\chi$ and any $\otimes_v f_v\in \bigotimes'_{v\nmid\infty}V_{\pi^{B}_v}$ we have
\[
\left |\int_{T(\hat F)/T(F)} \chi(t)\Phi_T(\otimes_v f_v)(t) d^\times t\right |^2=L'(1/2,\pi_K,\chi)L(1,\eta)\hat\beta_{\pi^{B},\chi}(\otimes_v f_v,\otimes_v f_v).
\]
Moreover, the above expression is $0$ whether $\hat\Sigma_B\neq\hat\Sigma_\pi^\chi$.
\end{theorem}
\begin{proof}
The expression is 0 whether $\hat\Sigma_B\neq\hat\Sigma_\pi^\chi$ by Proposition \ref{propWaldloc}. The expression is a reformulation of the Gross-Zagier-Zhang formula that can be found in \cite{YZZ}. We leave the details to the reader.
\end{proof}

\subsection{Indefinite anticyclotomic distributions}\label{DerAntDist}

As above, let $(\pi,V_\pi)$ be an irreducible cuspidal automorphic representation in $L^2(\GL_2(F)\backslash \GL_2(\A_F))$ with trivial central character and parallel weight $2$, and assume that $(\pi_\cP,V_{\pi_\cP})\simeq(\pi_\alpha^\C,V_\alpha^\C)$, where $\alpha=\pm 1$ or $|\alpha|=q$. %Let $E/F$ be a totally imaginary quadratic extension. 

In contrast with Definition \ref{defdef}, we assume that $\#(\hat\Sigma_\pi^1\setminus\{\cP\})+[F:\Q]$ is odd. Let $B$ be the quaternion algebra with ramification set $(\hat\Sigma_\pi^1\setminus\{\cP\})\cup(\Sigma_\infty\setminus\sigma)$, for a fixed archimedean place $\sigma$.
Let $\pi^{B}$ be the corresponding Jacquet-Langlands lift to $G_B$. %, the algebraic group associated with the multiplicative group of $A$.

The previous conditions imply that $B$ admits an embedding $\imath:K\hookrightarrow B$, that we fix once for all. %Let us denote by $T$ the torus associated with $E^\times/F^\times$.
As above, let $U=U_\cP\times U^\cP\subset G_B(\hat F)$ be an open compact subgroup such that $\dim_\C\left(\bigotimes'_{v\nmid\infty}\pi_v^{B}\right)^U=1$ and $U_\cP=\GL_2(\cO_\cP)$ or $K_0(\varpi)$. We have constructed a morphism of $G_B(\hat F)$-representations
\[
\Phi_T:\bigotimes'_{v\nmid\infty}\pi_v^{B}\simeq \pi^{B}\mid_{G_B(\hat F)}\longrightarrow H^0(T(F),\cA_B(A^0(K^{ab})\otimes_{L_\pi}\C)).
\]
Let $R\subset\C$ be any ring (endowed with the discrete topology) such that $\alpha\in R^\times$. In analogy with the definite situation, we consider
the $T(F_\cP)$-equivariant morphism 
\begin{equation}\label{hatdelta}
\hat\delta:C_c(T(F_\cP),R)\stackrel{\delta_{T}}{\longrightarrow}V_\alpha^R\subseteq V_{\pi^{B}_\cP}\longrightarrow\left(\bigotimes'_{v\nmid\infty}\pi^{B}_v\right)^{U^\cP}.
\end{equation}
Write $\Gamma=T(\hat F^\cP)\cap U^\cP$. Similarly as above, we decompose $\cG_{K,\cP}^\Gamma=G_{K,\cP}^\Gamma\times\cG_{K,\cP}$, where $G_{K,\cP}^\Gamma$ is its torsion subgroup. By Lemma \ref{lemstructGp}, $\cG_{K,\cP}\simeq\Z_p^{[F_\cP:\Q_p]}$ does not depend on $\Gamma$. We will consider the anticyclotomic distributions of functions supported on $\cG_{K,\cP}$.
%We know that the Artin map $\rho_A:\hat T/T\rightarrow \cG_{E,\cP}^\Gamma$ factors through $\Gamma$.
For $g\in C(\cG_{K,\cP},\C)$, we define
\[
\int_{\cG_{K,\cP}^\Gamma}g(\gamma)d\mu_{K,\cP}^{\rm ind}(\gamma):=[H_\cP:H]\int_{T(\hat F)/T(F)}g(\rho_A(t))\Phi_T(\hat\delta(1_H))(t)d^\times t\in A^0(K^{ab})\otimes_{L_\pi}\C,
\]
where $H\subset T(F_\cP)$ is an open compact subgroup small enough that $g\circ\rho_A$ is $H$-invariant. %, $1_H\in C_c(T_\cP,\Z)\subset C_c(T_\cP,\C)$ is the characteristic function of $H$, and $H_\cP$ is the maximal open compact subgroup of $T_\cP$.
 Since $\Phi_T\circ\hat\delta$ is $T(F_\cP)$-equivariant, one proves that this definition does not depend on $H$, exactly the same way as in the definite setting. %Indeed, for any $H'\subseteq H$, we have $1_H=\sum_{h\in H/H'}1_{hH'}=\sum_{h\in H/H'}h\ast 1_{H'}$, hence (if $g\circ\rho_A$ is $H$-invariant)
%\begin{eqnarray*}
%\int_{\cG_{E,\cP}^\Gamma}g(\gamma)d\mu_{E,\cP}^{\pi,II}(\gamma)&=&[H_\cP:H]\int_{\hat T/T}g(\rho_A(x))\Phi_T\circ\hat\delta\left(\sum_{h\in H/H'}h\ast 1_{H'}\right)(x)d^\times x\\
%&=&[H_\cP:H]\sum_{h\in H/H'}\int_{\hat T/T}g(\rho_A(x))\Phi_T\circ\hat\delta(1_{H'})(xh)d^\times x\\
%&=&[H_\cP:H]\sum_{h\in H/H'}\int_{\hat T/T}g(\rho_A(xh^{-1}))\Phi_T\circ\hat\delta(1_{H'})(x)d^\times x\\
%&=&[H_\cP:H][H:H']\int_{\hat T/T}g(\rho_A(x))\Phi_T\circ\hat\delta(1_{H'})(x)d^\times x,
%\end{eqnarray*}
%where the second equality is obtained from the $T_\cP$-equivariance of $\Phi_T\circ\hat\delta$.

\begin{remark}
By Theorem \ref{GrZaZh}, a necessary condition for  the Jacquet-Langlands lift  $\pi^{B}$ and the above expression to be non-zero is precisely that the ramification set of $B$ is $(\hat\Sigma_\pi^1\setminus\{\cP\})\cup(\Sigma_\infty\setminus\sigma)$.
\end{remark}

\begin{theorem}[Interpolation Property]\label{intprop2}
There is a non-zero constant $C_K$ depending on $K/F$ such that, 
for any continuous character $\chi:\cG_{K,\cP}\rightarrow\C^\times$,
\[
\left|\int_{\cG_{K,\cP}}\chi(\gamma)d\mu_{K,\cP}^{\rm ind}(\gamma)\right|^2=C_K C(\pi_\cP,\chi_\cP) \frac{L'(1/2,\pi_K,\chi)}{L(1,\pi,ad)},
\]
where
\[
C(\pi_\cP,\chi_\cP)=\left\{\begin{array}{ll}
					  \frac{L(1,\pi_\cP,ad)}{L(1/2,\pi_\cP,\chi_\cP)}, & |\alpha|^2=q,\\
					 \frac{1}{L(-1/2,\pi_\cP,\chi_\cP)}, & \alpha=\pm 1,\;\chi_\cP\mid_{\cO_{K_\cP}^\times}=1,\\
					 \frac{q^{n_\chi}}{L(1/2,\pi_\cP,\chi_\cP)}, & \alpha=\pm 1,\;\chi_\cP\mid_{\cO_{K_\cP}^\times}\neq1,
					\end{array}\right.
\]
and $n_\chi$ is the conductor of $\chi_\cP$.
\end{theorem}
\begin{proof}
The proof is completely analogous to the proof of Theorem \ref{intprop1}, if the Waldspurger formula (Theorem \ref{WaldThm}) is replaced by the Gross-Zagier-Zhang formula (Theorem \ref{GrZaZh}).
\end{proof}

\subsection{Cohomological interpretation of $\mu_{K,\cP}^{\rm ind}$}

Recall that
$\rho_\pi\otimes_L\C\simeq\pi^{B}\mid_{G_B(\hat F)}\simeq\bigotimes'_{v\nmid\infty}\pi_v^{B}$.
Notice that the composition
\[
\Ind_{U_\cP}^{G_B(F_\cP)}1_{\bar\Q}/(\cT_\cP-a_\cP)\Ind_{U_\cP}^{G_B(F_\cP)}1_{\bar\Q}\stackrel{\simeq}{\longrightarrow}V_\alpha^{\bar\Q}\longrightarrow \left(\bigotimes'_{v\nmid\infty}\pi_v^{B}\right)^{U^\cP}\simeq\left(\pi^{B}\mid_{G_B(\hat F)}\right)^{U^\cP}
\]
sends $f$ to $\sum_{g\in G_B(F_\cP)/U_\cP}f(g) g\ast\omega_\phi$.
Since $\Phi_T(\omega_\phi)\in H^0(T(F),\cA_B(A^0(K^{ab})))$, 
we deduce that the image of 
\begin{equation}\label{eqconPhi}
V_\alpha^{\bar\Q}\longrightarrow\left(\pi^{JL}\mid_{G_B(\hat F)}\right)^{U^\cP}\stackrel{\Phi_T}{\longrightarrow}H^0(T(F),\cA_B(A^0(K^{ab})\otimes_{L_\pi}\C)^{U^\cP})
\end{equation}
lies in $H^0(T(F),\cA_B(A^0(K^{ab})\otimes_{L_\pi}\bar\Q)^{U^\cP})$.
Given the fixed embeddings $L_\pi\hookrightarrow\bar\Q\hookrightarrow\C_p$, let
\[
\log_p=\log_{\omega_\phi}:A^0(K^{ab})\otimes_{L_\pi}\bar\Q\longrightarrow\C_p.
\] 
be the formal group logarithm attached to the differential $\omega_\phi\in\Omega^1_{X_U}$. Notice that the formal group logarithm extends to the whole $A(K^{ab})$ because, for any $P\in A(K^{ab})$, there is $n\in\N$ such that $nP$ lies in the subgroup of points reducing to the identity section.
Composing with such a formal logarithm, we obtain a morphism
\begin{equation}\label{eqindfAV}
V_\alpha^{\bar\Q}\longrightarrow H^0(T(F),\cA_B(A^0(K^{ab})\otimes_{L_\pi}\bar\Q)^{U^\cP})\longrightarrow H^0(T(F),\cA_B(\C_p)^{U^\cP}).
\end{equation}

From now on, we also denote by $\mu_{K,\cP}^{\rm ind}\in\Dist(\cG_{K,\cP},\C_p)$ (by abuse of notation) the $p$-adic distribution
\[
\int_{\cG_{K,\cP}}g(\gamma)d\mu_{K,\cP}^{\rm ind}(\gamma):=[H_\cP:H]\int_{T(\hat F)/T(F)}g(\rho_A(t))\log_p\left(\Phi_T\circ\hat\delta(1_H)(t)\right)d^\times t.
\]
Thus $\mu_{K,\cP}^{\rm ind}\in\Dist(\cG_{K,\cP},\C_p)$ can be interpreted as the logarithm of the corresponding $(A^0(K^{ab})\otimes_{L_\pi}\C)$-valued distribution.

Analogously to \S \ref{CoInt}, equation \eqref{eqindfAV} provides an element 
$\log\phi\in H^0(T(F),\cA_B^\cP(V_\alpha^{\bar\Q},\C_p)^{U^\cP})$. %\qquad \cA_f^\cP(\hat V_\alpha^{\bar\Q}\otimes_{\bar\Q}\C_p)=\Hom_{G(F_\cP)}(V_\alpha^{\bar\Q},\cA_f(\C_p)). 
If we recall the $T(F_\cP)$-equivariant morphism 
$\kappa: \cA_B^\cP(V_\alpha^{\bar\Q},\C_p)^{U^\cP}\rightarrow C_c(T(\hat F)/\Gamma,\C_p)_0^\vee$ of \eqref{kappa},
we deduce that
\[
\int_{\cG_{K,\cP}}g(\gamma)d\mu_{K,\cP}^{\rm ind}(\gamma)=\kappa(\log\phi)\cap\partial g.%\quad i=I,II,
\]

\subsection{$p$-adic measures and $p$-adic $L$-functions}\label{padicmeas}

%Let $E/F$ be a totally imaginary quadratic extension.
Let $\pi$ be a cuspidal automorphic representation of $\GL_2$ with parallel weight $2$ and a trivial central character. %Let $\Sigma_{\pi}^1$ be the finite set of places defined in previous sections. 

In the definite case ($\#(\Sigma_\pi^1\setminus\{\cP\})+[F:\Q]$ even), we have constructed the distribution $\mu_{K,\cP}^{\rm def}$ defined by
\[
g\in C(\cG_{K,\cP},\C_p)_0\longmapsto\kappa({\rm res}\phi)\cap\partial(g),
\]
where $\phi\in H^0(G_D(F),\cA_D^\cP(V_\alpha^{\bar\Q},\bar\Q)^{U^\cP})$ is some generator of the Jacquet-Langlands lift to the totally definite quaternion algebra $D$.

In the indefinite case $(\#(\Sigma_\pi^1\setminus\{\cP\})+[F:\Q]$ odd), we have constructed the distribution $\mu_{K,\cP}^{\rm ind}$ defined by
\[
g\in C(\cG_{K,\cP},\C_p)_0\longmapsto\kappa(\log\phi)\cap\partial(g),
\]
where $\log\phi\in H^0(T(F),\cA_B^\cP(V_\alpha^{\bar\Q},\C_p)^{U^\cP})$ has been obtained from the $p$-adic formal logarithm of Heegner points associated with $K$ on the Shimura curve attached to a quaternion algebra $B$ that splits at a single place.

\begin{definition}\label{defordinary}
We say that $\pi$ is $\cP$-ordinary if $\pi_\cP$ is of the form $V^\C_\alpha$ where
$\alpha\in \cO_{\C_p}^\times\cap\bar\Q$.
\end{definition}

\begin{proposition}\label{dist-meas}
If $\pi$ is $\cP$-ordinary then both $\mu_{K,\cP}^{\rm def}$ and $\mu_{K,\cP}^{\rm ind}$ are $p$-adic measures.
\end{proposition}
\begin{proof}
Write $G=G_D$ or $G_B$ depending if we are in the definite or indefinite setting.
Let $\bar\cO:=\bar\Q\cap\cO_{\C_p}$, $V^{\bar\cO}:=\Ind_{U_\cP}^{G(F_\cP)}1_{\bar\cO}/(\cT_\cP-a_\cP)\Ind_{U_\cP}^{G(F_\cP)}1_{\bar\cO}\subset V_\alpha^\C$ and write 
\[
\cW:=\{\psi\in\cA_\ast^\cP(V_\alpha^{\bar\Q},\C_p)^{U^\cP}:\psi(T(\hat F^\cP))(v)\subseteq\cO_{\C_p},\mbox{ for all }v\in V^{\bar\cO}\},\qquad\ast=B,D. 
\]
%Notice that there is a natural morphism of $\cO_{\C_p}$-modules $\cV\rightarrow \cA(\hat V_\alpha^{\bar\Q}\otimes_{\bar\Q}\C_p)^{U^\cP}$, since $V^{\bar\cO}\otimes_{\bar\cO}\bar\Q=V_\alpha^{\bar\Q}$ by Lemma \ref{lemlattices}. 
We claim that, if $\psi$ is in the image of the natural monomorphism
\[
H^0(T(F),\cW)\otimes_{\cO_{\C_p}}\C_p\longrightarrow H^0(T(F),\cA_\ast^\cP(V_\alpha^{\bar\Q},\C_p)^{U^\cP}),
\]
then the distribution $g\mapsto \kappa(\psi)\cap\partial(g)$ is a $p$-adic measure.
Indeed, by \eqref{defpart}, the restriction of $\partial$ provides an isomorphism
\[
\partial:C(\cG_{K,\cP}^\Gamma,\cO_{\C_p})_0\stackrel{\simeq}{\longrightarrow}H_0(T(F),C_c(T(\hat F)/\Gamma,\cO_{\C_p})_0).
\]
Moreover, by \eqref{decfunct}, $C_c(T(\hat F)/\Gamma,\cO_{\C_p})_0\simeq C_c(T(F_\cP), \cO_{\C_p})_0\otimes_{\cO_{\C_p}}C_c(T(\hat F^\cP)/\Gamma,\cO_{\C_p})_0$. Since $\alpha\in\bar\cO^{\times}$ (ordinary), by Lemma \ref{bounddelta} there exists a fixed $\lambda\in\bar\cO$ such that the restriction of the morphism $\kappa$ of \eqref{kappa} factors through
\[
\kappa: \cW\longrightarrow \lambda^{-1}C_c(T(\hat F)/\Gamma,\cO_{\C_p})_0^\vee,
\]
where $\lambda^{-1}C_c( T(\hat F)/\Gamma,\cO_{\C_p})_0^\vee:=\Hom_{\cO_{\C_p}}(C_c(T(\hat F)/\Gamma,\cO_{\C_p})_0,\lambda^{-1}\cO_{\C_p})$. We conclude that, for any $\psi\in H^0(T,\cW)$, the map $g\mapsto \kappa(\psi)\cap\partial(g)$ is a distribution in $\Dist(\cG_{K,\cP},\lambda^{-1}\cO_{\C_p})$. The claim follows from \eqref{defMeas}.

In the definite case, $\psi={\rm res}\phi$, where $\phi\in H^0(G_D(F),\cA_D^\cP(V_\alpha^{\bar\Q},\bar\Q)^{U_\cP})$. Let $v_0=[1_{U_\cP}]\in V^{\bar\cO}$. Note that $\phi(v_0)$ is an automorphic form on a totally definite quaternion algebra, hence its image is a finite set of values. This implies that, after scaling, we can assume that $\phi(v_0)$ has values in $\bar\cO$. Since $V^{\bar\cO}=\bar\cO[G(F_\cP)]v_0$, we deduce that $\phi(f)$ has values in $\bar\cO$, for any $f\in V^{\bar\cO}$. This implies that ${\rm res}\phi\in H^0(T(F),\cW)$, hence $\mu_{K,\cP}^{\rm def}\in\Meas(\cG_{K,\cP},\C_p)$.

In the indefinite case, $\psi=\log\phi$. By \eqref{2nd}, 
\[
(\log\phi)(g)(g_\cP v_0)=\log_p\circ(\Phi_T(\omega_\phi)(g_\cP,g))=\lambda(\cV)^{-1}(\log_p(AJ(\cV\ast(\tau_E,(g_\cP,g)U)))),
\]
where $g\in G_B(F^\cP)$, $g_\cP\in G_B(F_\cP)$ ($(g_\cP,g)\in G_B(\hat F)$), $\cV$ is some auxiliary Hecke operator, and $AJ$ is the Abel-Jacobi map \eqref{AJmap}. By Shimura's reciprocity law, 
\[
^{\rho_A(\Gamma)}\Phi_T(\omega_\phi)(G_B(F_\cP)\times T(\hat F^\cP))=%\Phi_T(\omega_\phi)(G_B(F_\cP)\times T(\hat F^\cP)\Gamma)=
\Phi_T(\omega_\phi)(G_B(F_\cP)\times T(\hat F^\cP)).
\]
Thus, the Galois action on $\lambda(\cV)\Phi_T(\omega_\phi)(G_B(F_\cP)\times T(\hat F^\cP))\subset A(K^{ab})$ factors through $\cG_{K,\cP}^\Gamma$. This implies that 
\[
\lambda(\cV)\Phi_T(\omega_\phi)(G_B(F_\cP)\times T(\hat F^\cP))\subseteq A(K^\Gamma),\quad\mbox{ where }\quad\Gal(K^\Gamma/K)=\cG_{K,\cP}^\Gamma.
\]
Let $K^\Gamma_\cP\subset\C_p$ be the local field extension generated by $K^\Gamma$, let $\cP^\Gamma$ above $\cP$ be its maximal ideal and let $k$ be its residue field. Since $\cG_{K,\cP}^\Gamma=\cG_{K,\cP}\times G_{K,\cP}^\Gamma$, where $G_{K,\cP}^\Gamma$ is finite and $\cG_{K,\cP}$ is the Galois group of an extension totally ramified at $\cP$, $k$ is finite.
Note that the formal logarithm (after normalization) is given by a formal series in $\bar\cO[[t]]$. We have the exact sequence 
\[
0\longrightarrow \cG_{A}(\cP^\Gamma)\longrightarrow A(K^\Gamma_\cP)\stackrel{red}{\longrightarrow}A(k)\longrightarrow 0,
\]
where $\cG_{A}$ is the formal group of $A$ 
and $A(k)$ is finite. This implies that the set $\log_p(A(K^\Gamma_\cP))$ has bounded denominators. We conclude that there exists $\lambda'\in\C_p$ such that $\log_p(\lambda(\cT)\Phi_T(\omega_\phi)(G_B(F_\cP)\times T(\hat F^\cP)))\subseteq (\lambda')^{-1}\cO_{\C_p}$. Again, since $V^{\bar\cO}=\bar\cO[G_B(F_\cP)]v_0$,
\begin{eqnarray*}
(\log\phi)(T(\hat F^\cP))(V^{\bar\cO})&=&(\log\phi)(T(\hat F^\cP))(\bar\cO[G(F_\cP)]v_0)\\
&=&\lambda(\cV)^{-1}\bar\cO(\log_p(\lambda(\cV)\Phi_T(\omega_\phi)(G_B(F_\cP)\times T(\hat F^\cP))))\\
&\subseteq&\lambda(\cV)^{-1}(\lambda')^{-1}\cO_{\C_p}.
\end{eqnarray*}
Thus, $\log\phi\in H^0(T(F),\cW)\otimes_{\cO_{\C_p}}\C_p$ and $\mu_{K,\cP}^{\rm ind}\in\Meas(\cG_{K,\cP},\C_p)$.
\end{proof}

Hence, in the $\cP$-ordinary case, $\mu_{K,\cP}^{\rm def},\mu_{K,\cP}^{\rm ind}\in\Meas(\cG_{K,\cP},\C_p)$ define elements $L^{\rm def}_\cP(\pi_K)$ and $L^{\rm ind}_\cP(\pi_K)$ in the Iwasawa algebra $\Lambda_{\C_p}=\cO_{\C_p}[[\cG_{E,\cP}]]\otimes_{\cO_{\C_p}}\C_p$, called (anticyclotomic) $p$-adic $L$-functions.
Interpolation properties %of the $p$-adic $L$-functions 
(Theorem \ref{intprop1} and Theorem \ref{intprop2}) and the vanishing of the local factor $L(0,\pi_\cP,1)^{-1}$ when $\pi_\cP$ is the Steinberg representation ($\alpha=1$), show the appearance of what are known as \emph{Exceptional Zeros}. In terms of Iwasawa algebras, this means that both $L_\cP^{\rm def}(\pi_K)$ and $L^{\rm ind}_\cP(\pi_K)$ lie in the augmentation ideal $\cI:=\ker(\deg)$, where
\[
\deg:\Lambda_{\C_p}\simeq\Meas(\cG_{K,\cP},\C_p)\longrightarrow \C_p,\qquad \mu\longmapsto\int_{\cG_{K,\cP}}d\mu.
\]

\begin{corollary}[Exceptional Zero]
Assume that $\pi_\cP$ is Steinberg.
Then the $p$-adic $L$-functions satisfy
\[
L^\bullet_\cP(\pi_K)\in\cI=\ker(\deg),\qquad \bullet={\rm def},{\rm ind},
\]
in both definite and indefinite cases.
\end{corollary}

\subsection{Explicit interpolation properties}\label{explIP}

The aim of this section is to compute the constants appearing in the interpolation formulas of $\mu_{K,\cP}^{\rm def}$ and $\mu_{K,\cP}^{\rm ind}$, under certain conditions, and for a fixed %embedding embedding of $K$ in the corresponding quaternion algebra and for fixed 
morphism \eqref{eqpiindpi} provided by a given newform.  

Let $\pi$ be a Hilbert automophic representation of parallel weight 2, level $K_0(\mathfrak{n})$, and trivial central character. Let $d_{K/F}$ be the discriminant of $K$ over $F$.
\begin{assumption}\label{hypdis}
We assume that $\gcd(\mathfrak{n},d_{K/F})$ is square free. 
\end{assumption}
Then we have that $\mathfrak{n}=\cP^n\mathfrak{n}'$ for $n\in\{0,1\}$, where $\mathfrak{n}'$ is prime-to-$\cP$. Write $\mathfrak{n}_0=\mathfrak{n}'$, $K$ is inert at $\cP$, or $\mathfrak{n}_0=\mathfrak{n}$, otherwise. Depending we are in the definite or indefinite setting, we consider the Jacquet-Langlands lift $\pi^D$ or $\pi^B$. I claim that there exists an order $\cO_{\mathfrak{n}_0}$ in the corresponding quaternion algebra $D$ or $B$ containing $\cO_K$ and with reduced discriminant $\mathfrak{n}_0$. Indeed, we can fix a maximal order $\cO_1$ containing $\cO_K$ and, by local conditions, an ideal $\cN_0\subset\cO_K$ such that $\Norm_{K/F}(\cN_0)\cdot\disc(\cO_1)=\mathfrak{n}_0$. We define $\cO_{\mathfrak{n}_0}=\cO_K+\cN_0\cO_1$. Thus, there exists a unique up to constant newform $\phi\in(\pi^D)^{U}$ (respectively $(\pi^B)^{U}$), where $U=\hat\cO_{\mathfrak{n}}^\times\subseteq\hat\cO_{\mathfrak{n}_0}^\times$ and we have an equality if $n=0$ or $K$ is not inert at $\cP$. Such newform $\phi$ fixes the morphism
\[
\delta:C_c(T(F_\cP),R)\longrightarrow V_{\pi^D}^{U^\cP},%\qquad \hat\delta:C_c(T(F_\cP),R)\longrightarrow V_{\pi^B}^{\infty,U^\cP}
\]
of \eqref{deltaarep} in the definite case.%and \eqref{hatdelta}.

Observe that $\delta(g)_v\in(\pi^D)^{U_v}$ and $\hat\delta(g)_v\in(\pi^B)^{U_v}$, for any $g\in C_c(T(F_\cP),R)$. Hence
we proceed to compute 
\[
C_v:=\frac{\beta_{\pi_v^B,\chi_v}(f_v,f_v)}{\langle f_v,f_v\rangle},\qquad f_v\in (\pi_v^B)^{U_v},
\]
for $v\neq \cP$ and $\chi_v$ unramified, which is independent of the choice of $f_v$.

%\subsubsection*{Assume that $B_v$ is a division algebra} In this case $\pi^B_v$ and $\chi_v$ are trivial, hence 
%\[
%C_v=\frac{\int_{T(F_v)}\chi_v(t)\langle \pi_v(t)f_v,f_v\rangle d^\times t}{\langle f_v,f_v\rangle}=\int_{T(F_v)}d^\times t={\rm vol}(T(F_v)).
%\]
\subsubsection*{Assume that $K$ is inert at $v$} In this case $T(F_v)=\cO_{K_v}^\times/\cO_{F_v}^\times$. Since $\cO_{K_v}^\times\subset U_v$, we have that $T(F_v)$ acts trivially on $f_v$. This implies
\[
C_v=\frac{\int_{T(F_v)}\chi_v(t)\langle \pi_v(t)f_v,f_v\rangle d^\times t}{\langle f_v,f_v\rangle}=\int_{T(F_v)}d^\times t={\rm vol}(T(F_v)),
\]
since $\chi_v$ is also trivial.

\subsubsection*{Assume $K$ is ramified at $v$} By the assumption \ref{hypdis}, the representation $\pi_v$ is either spherical or special. Hence if $B_v=\M_2(F_v)$, we can apply the results obtained in \S \ref{locthe}. Notice that $T(F_v)=(\varpi_K\cO_{K_v}^\times/\cO_{F_v}^\times)\times (\cO_{K_v}^\times/\cO_{F_v}^\times)$ and $f_v$ is $\cO_{K_v}^\times$-invariant. By Proposition \ref{innerprod} and Proposition \ref{prop-sc-prod}, we have that
\[
C_v=\frac{\left(\int_{T(F_v)}f_v(t)\chi_v(t)d^\times t\right)\overline{\left(\int_{T(F_v)}f^*_v(t)\chi_v(t)d^\times t\right)}}{\int_{T(F_v)}f_v(t)\overline{f^*_v(t)}d^\times t},
\]
where $f_v^*=f_v$, if $\pi_v$ is spherical, and $f_v^*=\Lambda(f_v)$, if $\pi_v$ is special. 

We can choose $\varpi_K$ such that $\varpi_K^2=\varpi$, hence
\[
\varpi_K=\left(\begin{array}{cc}a&b\\c&-a\end{array}\right)=\left(\begin{array}{cc}-\frac{\varpi}{c} &a\\&c\end{array}\right)\left(\begin{array}{cc}&-1\\1&-\frac{a}{c}\end{array}\right),\qquad a^2+bc=\varpi.
\]
Note that there exists a constant $C$ such that 
\[
f_v\left(\left(\begin{array}{cc}t_1&x\\&t_2\end{array}\right)g\right)=\left\{\begin{array}{l}C\alpha_v^{\nu_v(t_2/t_1)},\\ C\alpha_v^{\nu_v(t_2/t_1)}1_{K_0(\varpi)}(g),\end{array}\right.\qquad g\in K_0(1). 
\]
Since $a/c\in\cO_{F_v}$ because $a^2+bc=\varpi$, we deduce that $f_v(\varpi_{K})=C\alpha_v^{-1}$, if $\pi_v$ is spherical, and $f_v(\varpi_K)=0$, if $\pi_v$ is special. Using the same techniques as in \S \ref{locthe}, we deduce
\[
C_v=\left\{\begin{array}{lc}\frac{(1+\alpha_v^{-1}\chi_v(\varpi_K))(1+\alpha_v q_v^{-1}\chi_v(\varpi_K))}{1+q_v^{-1}}{\rm vol}(T(F_v))=\frac{L(1/2,\pi_v,\chi_v)\xi_v(2)}{L(1,\pi_v,ad)}{\rm vol}(T(F_v)),&\pi_v\mbox{ spherical}\\ 
%\frac{(1-\alpha_v\chi_v(\varpi_K))(1+\alpha_v\chi_v(\varpi_K)q^{-1})q}{q-1}{\rm vol}(T(F_v)),&\pi_v\mbox{ special}\end{array}\right.
\frac{1-\alpha_v\chi_v(\varpi_K)}{2}{\rm vol}(T(F_v)),&\pi_v\mbox{ special}\end{array}\right.
\]
In the special case, we are under the assumption that $\Hom_{T(F_v)}(\pi_v\otimes\chi_v,\C)\neq\emptyset$, hence $\alpha_v\chi_v(\varpi_K)=-1$ and $C_v={\rm vol}(T(F_v))$.

If $B_v$ is now a quaternion algebra, we have the other situation $\alpha_v\chi_v(\varpi_K)=1$. Notice that the corresponding Jacquet-Langlands representation satisfies $\pi^{JL}(\varpi_K)f_v=\alpha_v f_v$. Thus,
\[
C_v=\frac{\int_{T(F_v)}\chi_v(t)\langle \pi_v(t)f_v,f_v\rangle d^\times t}{\langle f_v,f_v\rangle}=(1+\chi_v(\varpi_v)\alpha_v){\rm vol}(\cO_{K_v}^\times/\cO_{F_v}^\times)={\rm vol}(T(F_v)).
\]

\subsubsection*{Assume that $K$ splits at $v$}
Note that in this case $B_v=\M_2(F_v)$ and $T(F_v)\simeq F_v^\times$ a conjugation of the diagonal torus, hence it is of the form
\[
T(F_v)=\left\{\left(\begin{array}{cc}t+\alpha(t-1)&\frac{\alpha(\alpha+1)}{c}(1-t)\\c(t-1)&1+\alpha(1-t)\end{array}\right)\;\; t\in F_v^\times\right\},\\
\]
for some $\alpha,c\in F_v$, with $c\neq 0$ because $T(F_v)$ intersects trivially with the Borel subgroup of upper triangular matrices $T_0$.
Notice that
\begin{eqnarray*}
&\left(\begin{array}{cc}t+\alpha(t-1)&\frac{\alpha(\alpha+1)}{c}(1-t)\\c(t-1)&1+\alpha(1-t)\end{array}\right)=A^{-1}\left(\begin{array}{cc}t&\\&1\end{array}\right)A=B^{-1}\left(\begin{array}{cc}1&\\&t\end{array}\right)B,\\
&A=\left(\begin{array}{cc}-1&\frac{\alpha}{c}\\c&-\alpha-1\end{array}\right),\qquad B=\left(\begin{array}{cc}-1&\frac{\alpha+1}{c}\\-c&\alpha\end{array}\right),
\end{eqnarray*}
Since $\cO_{K_v}\subseteq\GL_2(\cO_{F_v})$, we deduce that $\alpha,\frac{\alpha(\alpha+1)}{c}\in\cO_{F_v}$ and $c\in \mathfrak{n}_0$. Hence, we can assume $\frac{\alpha}{c}\in\cO_{F_v}$ and thus $A\in K_0(\mathfrak{n}_0)$ (otherwise we use the other expression and the matrix $B$).

Let us consider the $K_0(\mathfrak{n}_0)$-invariant Whittaker function $W_v:\GL_2(F_v)\rightarrow\C$ of $\pi_v$ normalized so that $W_v(d_{F_v})=1$, where $d_{F_v}$ is a generator of the discriminant $\cD_{F_v}$ (see \cite[\S 4.4]{Bump}). Thus,
\[
C_v\langle W_v,W_v\rangle=\int_{T(F_v)}\chi_v(t)\langle\pi_v(t)W_v,W_v\rangle d^\times t=\int_{T_0}\chi_v(t)\langle\pi_v(tA)W_v,\pi_v(A)W_v\rangle d^\times t.
\]
Since $\pi_v(A)W_v=W_v$, we compute
\begin{eqnarray*}
C_v\langle W_v,W_v\rangle=\int_{T_0}\chi_v(t)\int_{T_0}W_v(\tau t)\overline{W_v(\tau)}d^\times\tau d^\times t=\left|\int_{T_0}\chi_v(t)W_v(t)d^\times t\right|^2.
\end{eqnarray*}
Moreover, by \cite[Proposition 4.7.5]{Bump} and \cite[\S 2.4]{Hung2}
\[
\left|\int_{T_0}\chi_v(t)W_v(t)d^\times t\right|^2=L\left(\frac{1}{2},\pi_{K,v},\chi_v\right) \cdot{\rm N}(\cD_{F_v}).
\]
%Thus,
%\[
%C_v\langle W_v,W_v\rangle={\rm N}(\cD_{F_v})\cdot L\left(\frac{1}{2},\pi_v\otimes\chi_v\right)\cdot L\left(\frac{1}{2},\pi_v\otimes\chi_v^{-1}\right)={\rm N}(\cD_{F_v})\cdot L\left(\frac{1}{2},\pi_{K,v},\chi_v\right).
%\]
\subsubsection*{At the prime $\cP$} In this part we will compute the inner product $\langle f_\cP,f_\cP\rangle$,
where $f_\cP\in (V_\alpha)^{U_\cP}$ is such that $f_\cP(1)=1$. 
In order to simplify things, we choose the Haar measure of $T(F_\cP)$ so that the image of $\cO_{K_\cP}^\times$ has volume 1, as in \S \ref{locpair}.
Recall that 
\[
\langle f_\cP,f_\cP\rangle=c_T\int_{T(F_\cP)}f_\cP(t)\overline{f_\cP^\ast(t)}d^\times t,
\]
where $f_\cP^\ast=f_\cP$, if $|\alpha|^2=q$, or $f_\cP^\ast=\Lambda(f_\cP)$, if $\alpha=\pm 1$. 
\begin{proposition}
If $|\alpha|^2=q$, we have that
\[
\langle f_\cP,f_\cP\rangle=c_T\cdot q^{n_s}\cdot(1+q^{-1})\cdot L(1,\eta_\cP)=\left\{\begin{array}{lc}
c_T,&T(F_\cP)\mbox{ is inert};\\
c_T \cdot q^{n_s}\cdot\frac{1+q^{-1}}{1-q^{-1}},&T(F_\cP)\mbox{ is split};\\
c_T\cdot(1+q^{-1}),&T(F_\cP)\mbox{ is ramified},
\end{array}\right.
\]
where $n_s=\max\{n\in\N: \cO_K\subset \cO_{\mathfrak{n}_0\cP^n}\}$.

If $\alpha=\pm 1$, we have that
\[
\langle f_\cP,f_\cP\rangle=-c_T\bar C_T\cdot q^{2n_T-1}\cdot L(1,\eta_\cP)^2=\left\{\begin{array}{lc}
-c_T\bar C_T\cdot q^{2n_T-1}\frac{1}{(1+q^{-1})^2},&T(F_\cP)\mbox{ is inert};\\
-c_T\bar C_T\cdot q^{2n_T-1}\frac{1}{(1-q^{-1})^2},&T(F_\cP)\mbox{ is split};\\
-c_T\bar C_T\cdot q^{2n_T-1},&T(F_\cP)\mbox{ is ramified}.
\end{array}\right.
\]
where $n_T$ is the constant introduced in Theorem \ref{teocompIT}.
\end{proposition}
\begin{proof}
Note that, if $|\alpha|^2=q$ and $T(F_\cP)$ inert, then $f_\cP=\delta(1_{T(F_\cP)})$. Hence $\langle f_\cP,f_\cP\rangle=c_T$.

If $|\alpha|^2=q$ and $T(F_\cP)$ splits, we have seen above that we can assume that 
\[
T(F_\cP)=\left\{A\left(\begin{array}{cc}t&\\&1\end{array}\right)A^{-1}\right\},\quad \mbox{for some}\quad A=\left(\begin{array}{cc}a&b\\c&d\end{array}\right)\in K_0(1). 
\]
Hence, %by (left and right) Invariance (Corollary \ref{coras}),
\[
\langle f_\cP,f_\cP\rangle=c_T\int_{T(F_\cP)}\left|f_\cP\left(A\left(\begin{array}{cc}t&\\&1\end{array}\right)A^{-1}\right)\right|^2d^\times t=c_T\int_{T(F_\cP)}\left|f_\cP\left(\begin{array}{cc}at&b\\ct&d\end{array}\right)\right|^2d^\times t.
\]
Writing $D=ad-bc\in\cO_{F_\cP}^\times$, we have
\[
\left(\begin{array}{cc}at&b\\ct&d\end{array}\right)=\left(\begin{array}{cc}\frac{D}{c}&at\\&ct\end{array}\right)\left(\begin{array}{cc}&-1\\1&\frac{d}{ct}\end{array}\right)=\left(\begin{array}{cc}\frac{Dt}{d}&b\\&d\end{array}\right)\left(\begin{array}{cc}1&\\\frac{ct}{d}&1\end{array}\right).
\]
Hence, if $n_0:=\nu(d/c)$ and writing $n=\nu(t)$,
\[
\langle f_\cP,f_\cP\rangle=c_T\left(\sum_{n<n_0}q^{2\nu(c)+n}+\sum_{n\geq n_0}q^{2\nu(d)-n}\right)=c_T\cdot q^{\nu(c\cdot d)}\left(\frac{1+q^{-1}}{1-q^{-1}}\right).
\]

If $|\alpha|^2=q$ and $T(F_\cP)$ ramifies, we have seen above that 
\[
T(F_\cP)=(\varpi_K\cO_{K_\cP}^\times/\cO_{F_\cP}^\times)\times(\cO_{K_\cP}^\times/\cO_{F_\cP}^\times),\quad f_\cP(\cO_{K_\cP}^\times/\cO_{F_\cP}^\times)=1,\;\mbox{ and }\;f_\cP(\varpi_K\cO_{K_\cP}^\times/\cO_{F_\cP}^\times)=\alpha^{-1}.
\]
Thus,
$\langle f_\cP,f_\cP\rangle=c_T\cdot(1+q^{-1})$.

If $\alpha=\pm 1$ and $T(F_\cP)$ is inert, we have that $K_0(1)\cap K_{\cP}^\times=\cO_{K_\cP}^\times=(\cO_{F_\cP}+\beta\cO_{F_\cP})^\times$, but $K_0(1)\cap K_{\cP}^\times=(\cO_{F_\cP}+\varpi\beta\cO_{F_\cP})^\times$. Hence $f_\cP\mid_{T(F_\cP)}=1_{H_1}$, where $H_n=(\cO_{F_\cP}+\varpi^n\beta\cO_{F_\cP})^\times)/\cO_{F_\cP}^\times$. We compute
\begin{eqnarray*}
\langle f_\cP,f_\cP\rangle&=&c_T\int_{H_1}\overline{\Lambda(f_\cP)(t)}d^\times t=\left.c_T\bar C_T\int_{H_1}\int_{T(F_\cP)}\overline{\theta_T(s)(y)\cdot f_\cP(y^{-1}t)}d^\times yd^\times t\right|_{s=0}\\
&=&\left.c_T\bar C_T\int_{H_1}\int_{H_1}|c(t)|^{2s-2}d^\times y d^\times t\right|_{s=0}=\left.\frac{c_T\bar C_T}{q+1}\sum_{n\geq 1}	q^{(n_T+n)(2-2s)}{\rm vol}(H_n\setminus H_{n+1})\right|_{s=0}\\
&=&c_T\bar C_Tq^{2n_T}\frac{q(q-1)}{(1+q)^2(1-q)}.
\end{eqnarray*}
since ${\rm vol}(H_1)=(1+q)^{-1}$ and ${\rm vol}(H_n\setminus H_{n+1})=(q+1)^{-1}q^{-n}(q-1)$.

If $\alpha=\pm 1$ and $T(F_\cP)$ splits, we have seen above that we can assume that 
\[
T(F_\cP)=\left\{A\left(\begin{array}{cc}t&\\&1\end{array}\right)A^{-1}\right\},\quad \mbox{for some}\quad A=\left(\begin{array}{cc}a&b\\c&d\end{array}\right)\in K_0(\cP). 
\]
Hence, since $f_\cP\mid_{K_0(1)}=1_{K_0(\cP)}$ and $\nu(d)=0$, we have
\[
f_\cP(t)=f_\cP\left(\begin{array}{cc}at&b\\ct&d\end{array}\right)=\left\{\begin{array}{ll}\alpha^{n},&n=\nu(t)>-\nu(c)=n_T\\
0,&\mbox{otherwise.}\end{array}\right.
\]
Thus,
\begin{eqnarray*}
\langle f_\cP,f_\cP\rangle&=&c_T\sum_{n>n_T}\alpha^n\int_{\varpi^n\cO_{K_\cP}^\times}\overline{\Lambda(f_\cP)(t)}d^\times t\\
&=&\left.c_T\bar C_T\sum_{n>n_T}\alpha^n\int_{\cO_{K_\cP}^\times}\int_{T(F_\cP)}\overline{\theta_T(s)(y)\cdot f_\cP(\varpi^n y^{-1}t)}d^\times yd^\times t\right|_{s=0}\\
&=&\left.c_T\bar C_T\sum_{n>n_T}	\sum_{m<n-n_T}\int_{\varpi^m\cO_{F_\cP^\times}}\left|\frac{y}{c^2(y-1)^2}\right|^{1-s}d^\times y\right|_{s=0}\\
&=&\left.c_T\bar C_Tq^{n_T(2-2s)}\sum_{n>n_T}	\sum_{N>n_T-n}q^{N(s-1)}\right|_{s=0}=c_T\bar C_Tq^{2n_T}\frac{1}{(1-q^{-1})(1-q)}.
\end{eqnarray*}

If $\alpha=\pm 1$ and $T(F_\cP)$ ramifies, we have seen above that $f_\cP(\cO_{K_\cP}^\times/\cO_{F_\cP}^\times)=1$ and $f_\cP(\varpi_K)=0$. Hence, with the notations of \S \ref{locthe},
\begin{eqnarray*}
\langle f_\cP,f_\cP\rangle&=&c_T\int_{\cO_{K_\cP}^\times}\overline{\Lambda(f_\cP)(t)}d^\times t=\left.c_T\bar C_T\int_{\cO_{K_\cP}^\times}\int_{T(F_\cP)}\overline{\theta_T(s)(y)\cdot f_\cP(y^{-1}t)}d^\times yd^\times t\right|_{s=0}\\
&=&\left.c_T\bar C_T\cdot\int_{\cO_{K_\cP}^\times}\left|\frac{\det(t)}{c(t)^2}\right|^{1-s}d^\times y\right|_{s=0}\\
&=&\left.c_T\bar C_T\cdot\sum_{n\geq 0}q^{(2-2s)(n_T+n)}{\rm vol}(H_n\setminus H_{n+1})\right|_{s=0}=-c_T\bar C_T\cdot q^{2n_T-1}.
\end{eqnarray*}
\end{proof}

\subsubsection*{Interpolation formulas}
Applying the above result and Corollary \ref{Corcalclocal}, Theorem \ref{WaldThm}, we obtain that
\[
\left|\int_{\cG_{K,\cP}}\chi(\gamma)d\mu_{K,\cP}^{\rm def}(\gamma)\right|^2=\frac{C}{2}\cdot\xi(2)\cdot C_{\rm inert}\cdot C_{\rm ram}\cdot e(\pi_\cP,\chi_\cP)\cdot\frac{L(1/2,\pi_K,\chi)}{L(1,\pi,ad)}\cdot\prod_v\langle f_v,f_v\rangle,
\]
where 
\begin{eqnarray*}
&&C=\langle f_\cP,f_\cP\rangle\cdot{\rm vol}(\cO_{K_\cP}^\times/\cO_{F_\cP}^\times)\cdot\prod_{v\;{\rm split}}\frac{\xi_v(1)L(1,\pi_v,ad)}{\langle W_v,W_v\rangle\xi_v(2)}\cdot{\rm N}(\cD_{F_v})\prod_{v\;{\rm nonsplit}}{\rm vol}(T(F_v)),\\
&&C_{\rm inert}=\left(\prod_{v\mid\disc(B_v),\;{\rm inert}}\frac{1}{\xi_v(1)}\right),\qquad
C_{\rm ram}=\left(\prod_{v\;{\rm ram},\;v\mid\mathfrak{n}'}\frac{1}{L(1/2,\pi_v,\chi_v)}\right),\\
&&e(\pi_\cP,\chi_\cP)=\left\{\begin{array}{ll}
\frac{L(1,\pi_\cP,ad)}{q^{n_s}\xi_\cP(1)L(1/2,\pi_\cP,\chi_\cP)},&\mbox{if }|\alpha|^2=q,\\
\frac{\xi_\cP(1)q^{-1}}{\xi_\cP(2)L(-1/2,\pi_\cP,\chi_\cP)},&\mbox{if }\alpha=\pm 1,\chi_\cP|_{\cO_{K_\cP}^\times}=1,\\
\frac{\xi_\cP(1)q^{n_\chi-1}}{\xi_\cP(2)L(1/2,\pi_\cP,\chi_\cP)},&\mbox{if }\alpha=\pm 1,\chi_\cP|_{\cO_{K_\cP}^\times}\neq1.\\
\end{array}\right.
\end{eqnarray*}
Similarly, applying the above result, Corollary \ref{Corcalclocal}, and Theorem \ref{GrZaZh}, we obtain
\[
\left|\int_{\cG_{K,\cP}}\chi(\gamma)d\mu_{K,\cP}^{\rm ind}(\gamma)\right|^2=C\cdot\xi(2)\cdot C_{\rm inert}\cdot C_{\rm ram}\cdot e(\pi_\cP,\chi_\cP)\cdot\frac{L'(1/2,\pi_K,\chi)}{L(1,\pi,ad)}\cdot\prod_{v}\langle f_v,f_v\rangle.
\]
Moreover, these expressions do not depend on the choice of the newform or the the Haar measure. But we can apply the explicit Waldspurger and Gross-Zagier-Zhang formulas given in \cite[Theorem 6.1, Theorem 7.1]{Zhang03} for a suitable special case in order to compute the terms $\langle f_v,f_v\rangle$, $\langle W_v,W_v\rangle$ and ${\rm vol}(T(F_v))$. We obtain the following result: 
\begin{theorem}\label{explIPthm}
Let $\cN$ be the level of the Eichler order $\cO_{\mathfrak{n}'}$, and write 
\[
K_{ram}:=\prod_{v{\,\rm ram},\,v\mid\cN}\frac{\xi_v(1)}{\xi_v(2)}=\prod_{v{\,\rm ram},\,v\mid\cN}\frac{1}{L(1/2,\pi_v,\chi_v)}.
\]
In the indefinite case we have that
\[
\left|\int_{\cG_{K,\cP}}\chi(\gamma)d\mu_{K,\cP}^{\rm ind}(\gamma)\right|^2=\frac{\sqrt{\Norm(d_{K/F})}}{2^{d+1}}\cdot K_{ram}\cdot e(\pi_\cP,\chi_\cP)\cdot \frac{L'(1/2,\pi_K,\chi)}{\parallel f\parallel^2},
\]
where $f$ is the corresponding Hilbert newform, and $||f||^2$ is computed using the invariant measure on the Hilbert variety induced by $dxdy/y^2$ on $\mathfrak{H}$.
In the definite case, 
if we choose the newform $\phi$ to be of norm 1, then   
\[
\left|\int_{\cG_{K,\cP}}\chi(\gamma)d\mu_{K,\cP}^{\rm def}(\gamma)\right|^2=\frac{\sqrt{\Norm(d_{K/F})}}{2^{d}}\cdot K_{ram}\cdot e(\pi_\cP,\chi_\cP)\cdot \frac{L(1/2,\pi_K,\chi)}{\parallel f\parallel^2}.
\]
\end{theorem}

\section{Automorphic $\cL$-invariants}\label{Ap2}

%Let $G$ be the algebraic group associated with the multiplicative group of a quaternion algebra over $F$ that can be either definite or split at a single archimedean place. In addition, assume that $G(F_\cP)=\GL_2(F_\cP)$, where $\cP$ is a prime ideal of $F$ dividing $p$. 

Let $(\pi, V_\pi)$ be an automorphic representation of $\GL_2(\A_F)$ of parallel weight 2 with a trivial central character, and assume that $(\pi_\cP,V_{\pi_\cP})$ is the Steinberg representation $(\pi_1^\C,V_1^\C)$. Let us assume that $\cP$ splits in $K$. %Let $U\subset G(\hat F)$ be an open compact subgroup such that $\dim_\C(\Pi\mid_{G(\hat F)})^U=1$. In this section, we will assume that $T^2$ is a torus in $G(F)$ associated with a totally imaginary extension that splits at $\cP$. Write $T=T^2/Z$ as usual.

\subsection{Extensions of the Steinberg representation}\label{ExtStRep}

The main references in this section are \cite[\S 2.7]{Spiess} and \cite{B-G}.
%Let $G:=G(F_\cP)$ and $B$ is its Borel subgroup. 
Recall that the Steinberg representation $V_1^\C$ is defined by means of the exact sequence
\[
0\longrightarrow \C\phi_0\longrightarrow \bar V_1^\C\longrightarrow V_1^\C\longrightarrow 0,
\]
where $\phi_0(g)=1$ for all $g\in \GL_2(F_\cP)$. For any topological ring $R$, we defined in \S\ref{Loc-Re} its $R$-valued analogue $(\pi^R,V^R):=(\pi_1^R,V_1^R)$.
%We consider its dual, namely, $D^R:=\Hom_R(V^R,R)$, with the usual $G$-action $(\pi^R)^\vee$
%\[
%(\pi^R)^\vee(g) f(v)=f(\pi^R(g^{-1})v),\quad g\in G,\;f\in D^R, \;v\in V^R.
%\] 

Fix en embedding $K_\cP^\times\hookrightarrow\GL_2(F_\cP)$.
Since  $K_\cP^\times$ splits, there are two eigenvectors $v_1,v_2\in F_\cP^2$ and two eigenvalue morphisms 
\[
\lambda_1,\lambda_2:K_\cP^\times\longrightarrow F_\cP^\times,
\]
such that $tv_i=\lambda_i(t)v_i$ and $\det(t)=\lambda_1(t)\lambda_2(t)$, for all $t\in K_\cP^\times$. We fix the isomorphism $\psi:T(F_\cP)\stackrel{\simeq}{\rightarrow}F_\cP^\times$ provided by $t\mapsto \lambda_1(t)/\lambda_2(t)$. Notice that we have the map
\begin{equation}\label{varphiPP1}
\varphi:\GL_2(F_\cP)\longrightarrow \PP^1(F_\cP);\quad g=\left(\begin{array}{cc}a&b\\c&d\end{array}\right)\longmapsto \frac{-d}{c},
\end{equation}
that identifies $P_\cP\backslash \GL_2(F_\cP)$ with $\PP^1(F_\cP)$. If $K_\cP^\times\not\subset P_\cP$, the restriction of $\varphi$ provides an injection 
\[
\varphi:T(F_\cP)\hookrightarrow\PP^1(F_\cP),\quad\mbox{ such that }\quad\PP^1(F_\cP)\setminus\varphi(T(F_\cP))=\{x_1,x_2\},
\]
where $x_1$ and $x_2\in\PP^1(F_\cP)$ correspond to the spaces generated by $v_1$ and $v_2$, respectively. Since $\varphi(1)=\infty\in\varphi(T(F_\cP))$, the points $x_i\neq\infty$ can be seen as values in $F_\cP$. If $P_{x_i}:=\varphi^{-1}(x_i)\subseteq \GL_2(F_\cP)$, we have the maps
\[
\Lambda_i:\GL_2(F_\cP)\setminus P_{x_i}\longrightarrow F_\cP^\times; \qquad \left(\begin{array}{cc}a&b\\c&d\end{array}\right)\longmapsto d+x_i c.
\]
It is easy to check that %$\Lambda_i\mid_{T_\cP^2}=\lambda_i$ and 
$\Lambda_i(gt)=\Lambda_i(g)\lambda_i(t)$, for all $t\in K_\cP^\times$. Let $H\subset T(F_\cP)$ be the maximal open subgroup. Thus $T(F_\cP)/H\simeq\varpi^\Z$ for some $\varpi\in T(F_\cP)$. Assume that $x_1=\lim_{x_n\in\varpi^n H}x_n$ in $\PP^1(F_\cP)$ and let $U=x_1\cup\bigcup_{n\in\N} \varpi^n H$.

Let $\ell:F_\cP^\times\rightarrow R$ be any continuous group homomorphism. Then, we can define the cocycle $c_\ell\in H^1(\GL_2(F_\cP),V^R)$ associated with the extension of $R[\GL_2(F_\cP)]$-modules
\begin{equation}\label{exseqEell}\xymatrix{
0\ar[r] &V^R\ar[rr]_{\phi\mapsto (\phi,0)}^\iota& &\mathscr{E}(\ell)\ar[rr]^\pi_{(\phi,y)\mapsto y}&& R\ar[r]& 0,
}\end{equation}
where  
\[
\mathscr{E}(\ell):=\left\{(\phi,y)\in C(\GL_2(F_\cP),R)\times R:\; \phi\left(\left(\begin{array}{cc}t_1&x\\&t_2\end{array}\right)g\right)=\phi(g)+\ell(t_2)y\right\}/R\iota\phi_0,
\]
with $G$-action $g\ast(\phi,y)=(g\ast\phi,y)$ and $g\ast\phi(h)=\phi(hg)$.
Note that the above sequence is exact, since $\pi$ is surjective, indeed, we can define $(\phi_1,1)\in \mathscr{E}(\ell)$, where
\[
\phi_1(g)=\left\{\begin{array}{ll}\ell\left(\Lambda_2(g)\right),&\varphi(g)\in U,\\\ell\left(\Lambda_1(g)\right),&\varphi(g)\not\in U. \end{array}\right.
\] 

\begin{remark}\label{2intEell}
The above definition of $\mathscr{E}(\ell)$ differs from the one given in \cite{Spiess}, for example, where the extensions introduced there are of the form
\[
\mathscr{E}(\ell)^*:=\left\{(\phi,y)\in C(\GL_2(F_\cP),R)\times R:\; \phi\left(\left(\begin{array}{cc}t_1&x\\&t_2\end{array}\right)g\right)=\phi(g)+\ell\left(\frac{t_2}{t_1}\right)y\right\}/R\iota\phi_0.
\]
It is easier to show that these extensions have trivial action of the center. But notice that there is an isomorphism
\[
\mathscr{E}(2\ell)\longrightarrow \mathscr{E}(\ell)^*;\qquad (\phi,y)\longmapsto(\phi-yf_0,y),\quad f_0(g):=\ell(\det(g)),
\] 
that is $\GL_2(F_\cP)$-equivariant since we are taking quotients by $R\iota\phi_0$. Thus, both approaches are almost equivalent.
\end{remark}

\begin{remark}\label{rempreprop}
The identification $P_\cP\backslash \GL_2(F_\cP)\simeq\PP^1(F_\cP)$ provides an $R$-module isomorphism $\bar V_1^R\simeq C(\PP^1(F_\cP),R)$. Thus, we can consider $V_1^R=V^R$ as a quotient of $C(\PP^1(F_\cP),R)$.
\end{remark}

\begin{proposition}\label{eqcoc}
Write $\tilde\ell:=\ell\circ\psi:T(F_\cP)\rightarrow R$.
The restriction ${\rm res}(c_\ell)\in H^1(T(F_\cP),V^R)$ coincides with the cocycle
\[
{\rm res}(c_\ell)=z_{\tilde\ell};\qquad z_{\tilde\ell}(t):=(1-t)\tilde\ell1_U\in C(\PP^1(F_\cP),R).
\]
\end{proposition}
\begin{proof}
We compute ${\rm res}(c_\ell)(x)$, for $x\in\PP^1(F_\cP)$, 
\[
{\rm res}(c_\ell)(t)(x)=t\ast\phi_1(g_x)-\phi_1(g_x)=\phi_1(g_xt)-\phi_1(g_x),
\]
for any choice $g_x\in \varphi^{-1}(x)$. % and the expression does not depend on this choice.
Since $\phi_1(g)=\ell\left(\Lambda_1(g)\right)+\ell\left(\frac{\Lambda_2(g)}{\Lambda_1(g)}\right)1_U(\varphi(g))$, we have
\begin{eqnarray*}
{\rm res}(c_\ell)(t)(x)&=&\ell\left(\Lambda_1(g_xt)\right)-\ell\left(\Lambda_1(g_x)\right)+(1-t)\tilde\ell1_U(x)\\
&=&\ell\left(\lambda_1(t)\right)+(1-t)\tilde\ell1_U(x)=\ell\left(\lambda_1(t)\right)1_{\PP^1}(x)+(1-t)\tilde\ell1_U(x),
\end{eqnarray*}
since $\ell\left(\frac{\Lambda_1(g_x)}{\Lambda_2(g_x)}\right)$ coincides with $\tilde\ell(t)$ whenever $x=\varphi(t)$ with $t\in T(F_\cP)$. The result follows because the constant function $\ell\left(\lambda_1(t)\right)1_{\PP^1}(x)$ corresponds to $\ell\left(\lambda_1(t)\right)\phi_0$ under the identification of Remark \ref{rempreprop}.
\end{proof}

\subsection{$\cL$-invariants}

As in \S \ref{MultOne}, let $L/L_\pi$ be any field field extension endowed with the discrete topology. Let $B/F$ be a quaternion algebra that splits at $\cP$ and admitting a Jacquet-Langlands lift $\pi^B$.
Remark \ref{AvsHom} implies that
\begin{equation}\label{1-dimeq}
H^k(G_B(F),\cA_B^\cP(V^{{L}},L))_{\pi^B}\simeq H^k(G_B(F),\cA_B({\bf C}))_{\pi^B}.
\end{equation}

Let $c_\ord\in H^1(G_B(F),V^{\Z})$ be the restriction of the class associated with the continuous morphism $\ord:F_\cP^\times\rightarrow\Z$ (here the ring $\Z$ is considered with the discrete topology).
\begin{proposition}\label{propord}
The cup product by $c_\ord$ provides an isomorphism of 1-dimensional $L$-vector spaces
\[
H^k(G_B(F),\cA_B^\cP(V^{L},L))_{\pi^B}\stackrel{\cup c_\ord}{\longrightarrow}H^{k+1}(G_B(F),\cA_B^\cP({L}))_{\pi^B} ,
\]
where $k=0$, if $B$ is definite, and $k=1$, if $B$ splits at a single archimedean place.
\end{proposition}
\begin{proof}
The proof is completely analogous to \cite[Lemma 5.2 (b)]{Spiess}.
\end{proof}

\begin{remark}\label{remVCs}
%Let $R_0$ denote the topological ring $R$ endowed with the discrete topology. 
Since $P_\cP\backslash \GL_2(F_\cP)\simeq\PP^1(F_\cP)$, we can identify $V_0^R:=V_{1,0}^{R}$ with $C_c(F_\cP,R)_0$ and $V^{R}$ with $C_\diamond(F_\cP,R)$, both inside $C(\PP^1(F_\cP),R)$. 
\end{remark}

%Let $\cO=\cO_{\C_p}$ be the valuation ring of $\C_p$.%, and write $\hat V^{\cO_0}:=\Hom_{\cO[G(F_\cP)]}(V^{\cO_0},\cO)$. 
The above remark implies that there is a canonical pairing
\begin{equation}\label{pairingASTA}
\langle\;,\;\rangle:\cA_B^\cP(V^{\cO_{\C_p}}_0,\cO_{\C_p})\otimes_{\cO_{\C_p}}\C_p\times V^{\Z_p}\longrightarrow \cA_B^\cP(\C_p),
\end{equation}
since $\Hom_{\cO_{\C_p}}(V_0^{\cO_{\C_p}},\cO_{\C_p})$ is identified with a space of bounded distributions. Moreover, if we equip $V^{\Z_p}$ with the action of $G_B(F)$ provided by the inclusion $G_B(F)\hookrightarrow \GL_2(F_\cP)/F_\cP^\times$, the above pairing is $G_B(F)$-equivariant. 

Recall that, by Lemma \ref{lemlattices} and Lemma \ref{lemtenpro}, we can identify
\[
H^k(G_B(F),\cA_B^\cP(V^{\bar\Q},\C_p))\simeq H^k(G_B(F),\cA_B^\cP(V_0^{\cO_{\C_p}},\cO_{\C_p})\otimes_{\cO_{\C_p}}\C_p).
\]
Given a continuous homomorphism $\ell:F^\times_\cP\rightarrow\Z_p$, the cup product by $c_\ell\in H^1(G_B(F),V^{\Z_p})$ on the $\pi^B$-isotypic component induces a morphism 
\[
(\cdot\cup c_\ell)_{\pi^B}: H^k(G_B(F),\cA_B^\cP(V^{\bar\Q},\C_p))_{\pi^B}\longrightarrow H^{k+1}(G_B(F),\cA_B^\cP(\C_p))_{\pi^B},
\]
where $k=0$ if $B$ is definite, and $k=1$ if $B$ splits at a single archimedean place.
By Proposition \ref{propord}, both $H^k(G_B(F),\cA_B^\cP(V^{\bar\Q},\C_p))_{\pi^B}$ and $H^{k+1}(G_B(F),\cA_B^\cP(\C_p))_{\pi^B}$ are one-dimensional $\C_p$-vector spaces with a fixed isomorphism $(\cdot\cup c_\ord)_{\pi^B}$.

\begin{definition}\label{defLinv}
Given a continuous group homomorphism $\ell:F^\times_\cP\rightarrow \Z_p$, the \emph{automorphic $\cL$-invariant} associated with $\ell$ and $\pi^B$ is the unique $\cL_\cP(\pi^B,\ell)\in \C_p$ such that
\[
(\cdot\cup c_\ell)_{\pi^B}=\cL_\cP(\pi^B,\ell)(\cdot\cup c_\ord)_{\pi^B}.
\]
\end{definition}

We expect this automorphic $L$-invariant to depend only on the representation $\pi$ and not on the  Jacquet-Langlands lift $\pi^B$.%, that is to be independent of the quaternion algebra whose multiplicative group is $G$.  

\begin{conjecture}\label{conjLinv}
Let $B/F$ and $D/F$ be two quaternion algebras over $F$ that are either definite or split at a single archimedean place, but both split at $\cP$.
Then we have that
\[ 
\cL_\cP(\pi^B,\ell)=\cL_\cP(\pi^D,\ell),
\]
for any continuous morphism $\ell: F_\cP^\times\rightarrow \Z_p$.
\end{conjecture}
\begin{remark}
We can define $\cL_\cP(\pi^B,\ell)$ for any quaternion algebra $B/F$ that splits at $\cP$. In this case, the automorphic representation $\pi^B$ lies in $H^k(G_B(F),\cA_B^\cP(V^{\bar\Q},\C_p))$, where $k$ is the number of archimedean places where $B$ splits. Hence, we can reformulate the previous conjecture for any pair of quaternion algebra that split at $\cP$ (see \cite[Definition 2.5]{lennart-L-invariants}). 

Conjecture \ref{conjLinv} has been recently proven in \cite{lennart-L-invariants} under certain technical assumptions.
\end{remark}

For the rest of the section let us assume that $B$ is a quaternion algebra that splits at a single place. The restriction map identifies $H^1(G_B(F),\cA_B^\cP(V^{\bar\Q},\C_p))$ with the subspace of $H^1(G_B(F)^+,\cA_B^\cP(V^{\bar\Q},\C_p))$ fixed by $G_B(F)/G_B(F)^+$. Write $H^1(G_B(F)^+,\cA_B^\cP(V^{\bar\Q},\C_p))^-$ for the subspace such that the non-trivial element of $G(F)/G(F)^-$ acts as $-1$. The proof of Proposition \ref{1-dim} shows that 
\[
H^1(G_B(F)^+,\cA_B^\cP(V^{\bar\Q},\C_p))^-_{\pi^B}=H^1(G_B(F)^+,\cA_B(\C_p))^-_{\pi^B}\simeq \C_p.
\] 
Moreover, the analogous result to Proposition \ref{propord} shows that 
\[
\cup c_{\ord}:H^1(G_B(F)^+,\cA_B^\cP(V^{\bar\Q},\C_p))^-_{\pi^B}\longrightarrow H^2(G_B(F)^+,\cA_B^\cP(\C_p))^-_{\pi^B},
\]
is an isomorphism of 1-dimensional $\C_p$-vector spaces. The following result shows that in fact the same $L$-invariant is obtained when using these new cohomology groups.
\begin{proposition}\label{L-inv-}
Given a group homomorphism $\ell:F_\cP^\times\rightarrow\Z_p$, we have that 
\[
(\cdot\cup c_\ell)_{\pi^B}=\cL_\cP({\pi^B},\ell)(\cdot\cup c_\ord)_{\pi^B}
\]
in $\Hom\left(H^1(G_B(F)^+,\cA_B^\cP(V^{\bar\Q},\C_p))^-_{\pi^B}, H^2(G_B(F)^+,\cA_B^\cP(\C_p))^-_{\pi^B}\right)$.
\end{proposition}
\begin{proof}
The result is clear for $\ell=\lambda\cdot\ord$, for some $\lambda\in\C_p$. Thus, it remains to prove it for $\ell=\log_\sigma=\log\circ\sigma$, where $\log:\C_p^\times\rightarrow\C_p$ is the $p$-adic logarithm sending $p$ to 0, and $\sigma:F_\cP\hookrightarrow\C_p$. For $F=\Q$, it is proved in \cite[Corollaire 5.1.3]{Breu} using completed cohomology.

Recent results on the uniqueness of Breuil's $\cL$-invariant allow us to mimic the arguments given in \cite{Breu} in order to prove the result for arbitrary totally real field $F$. Let $\cL:=-\cL_\cP({\pi^B},\log_\sigma)$, and write $\log_{\sigma,\cL}$ for the $p$-adic logarithm such that $\log_{\sigma,\cL}(p)=\cL$ and $\log_{\sigma,\cL}\mid_{\cO_{F,\cP}}=\log_\sigma$. In \cite{Ding}, the set of $\Q_p$-analytic vectors in $\mathscr{E}(2\log_{\sigma,\cL})$ are denoted by $\Sigma(\cL)=\Sigma(\underline{2},0;1,\cL)$ (see Remark \ref{2intEell}). Note that $\log_{\sigma,\cL}=\log_{\sigma}-\cL\ord$, hence $c_{2\log_{\sigma,\cL}}=2c_{\log_{\sigma}}-2\cL c_{\ord}$. Thus, by definition,
\[
(\cdot\cup c_{2\log_{\sigma,\cL}})_{\pi^B}=0.
\]
This implies that the exact sequence \eqref{exseqEell} provides a surjective morphism
\[
\Hom_{G_B(F_\cP)}(\Sigma(\cL),\hat H^1(G_B(F)^+,\cA_B(L)))_{\pi^B}\longrightarrow H^1(G_B(F)^+,\cA_B^\cP(V^{L},L))_{\pi^B},
\]
where $L$ is the localization of the field of definition $L_\pi$ and $\hat H^1(G_B(F)^+,\cA_B(L))=\varprojlim_n H^1(G_B(F)^+,\cA_B(\cO_{L}/\varpi^n))\otimes_{\cO_L}{L}$ is the completed cohomology. 
By \cite[Corollary 4.7]{Ding} the element $\cL$ coincides with Breuil's $\cL$-invariant, and its definition does not depend on the on the sign of the action of $G_B(F)/G_B(F)^+$, thus the result follows.
\end{proof}

\subsection{Geometric $\cL$-invariants}\label{GeoLinv}

In this section we show that, in the definite setting, automorphic $\cL$-invariants coincide with the geometric $\cL$-invariants defined below. During the paper review process F. Bergunde and L. Gehrmann informed us that in \cite{B-G} they prove this result using similar techniques.  

Recall that $A$ is the abelian variety of $\GL_2$-type associated with $\pi$, and let $A^\vee$ be its dual abelian variety. %Let $L$ be the field of definition of $\Pi$. %, and let $\cO_L\subset L$ be the endomorphism ring of $A_\Pi$. 
Since $\pi$ is Steinberg at $\cP$, the abelian varieties $A$ and $A^\vee$ have purely multiplicative reduction at $\cP$, hence they admit the following analytic description: There is a pairing (determined up to canonical isomorphism)
\[
X\times Y\stackrel{j}{\longrightarrow} F_\cP^\times,
\]
where $X$ and $Y$ are free abelian groups of rank $d=[L:\Q]$, and $j$ is a bi-multiplicative mapping such that the composition $\ord_\cP\circ j$ tensored with $\Q$ gives a perfect duality of $\Q$-vector spaces
\begin{equation}\label{exseqordlat}
X\otimes\Q\times Y\otimes\Q\stackrel{\ord_\cP\circ j}{\longrightarrow}\Q,
\end{equation}
moreover, such that there is a pair of exact sequences of $\Gal(\bar\Q_p/F_\cP)$-modules ($X$ and $Y$ endowed with trivial Galois action)
\begin{eqnarray}
0\longrightarrow X\longrightarrow&\Hom(Y,\bar\Q_p^\times)&\stackrel{t_A}{\longrightarrow} A(\bar\Q_p)\longrightarrow0\\
0\longrightarrow Y\longrightarrow&\Hom(X,\bar\Q_p^\times)&\stackrel{t_{A^\vee}}{\longrightarrow} A^\vee(\bar\Q_p)\longrightarrow0
\end{eqnarray}
where the morphisms on the left are induced by $j$. 

Let $\cO_{L_\pi}\subset L_\pi$ be the endomorphism ring of $A$ (and $A^\vee$). Thus, $X$ and $Y$ are $\cO_{L_\pi}$-modules and $j(\alpha x,y)=j(x,\alpha y)$ for all $x\in X$, $y\in Y$ and $\alpha\in\cO_{L_\pi}$. The non-degenerate pairing \eqref{exseqordlat} provides an isomorphism of 1-dimensional $L_\pi$-vector spaces
\[
\alpha: X\otimes\Q\longrightarrow\Hom(Y\otimes\Q,\Q)
\]
Given a continuous morphism $\ell:F_\cP^\times\rightarrow\Z_p$, we consider the corresponding bilinear pairing 
\[
X\otimes\Q_p\times Y\otimes\Q_p\stackrel{\ell\circ j}{\longrightarrow}\Q_p,
\]
and the corresponding homomorphism of $L_\pi\otimes\Q_p$ modules (free of rank 1)
\[
\beta_\ell: X\otimes\Q_p\longrightarrow\Hom_{\Q_p}(Y\otimes\Q_p,\Q_p).
\]
Hence there exist $\cL_\cP(A,\ell)'\in L\otimes \Q_p$ such that
\[
\beta_\ell=\cL_\cP(A,\ell)'\alpha_p,
\]
where $\alpha_p=\alpha\otimes 1:Y\otimes\Q_p\rightarrow\Hom_{\Q_p}(X\otimes\Q_p,\Q_p)$. We define $\cL_\cP(A,\ell)\in\C_p$ to be the image of $\cL_\cP(A,\ell)'$ under the homomorphism $L_\pi\otimes\Q_p\rightarrow\C_p$ given by the fixed embedding $\bar\Q\hookrightarrow\C_p$.

Assume we are in the definite case ($D/F$ is a definite quaternion algebra that splits at $\cP$ and admitting a Jacquet-Langlands lift $\pi^D$), hence $\pi^D$ is generated by an automorphic form $\phi\in H^0(G_D(F),\cA_D(\cO_{L_\pi})^U)$. Since $\cO_{L_\pi}\simeq\Z^{d}$, $\phi$ can be seen as $\phi=(\phi_i)_{i=1,\cdots, d}$, where $\phi_i\in H^0(G_D(F),\cA_D(\Z)^U)$. Since $\pi_\cP$ is Steinberg, each $\phi_i$ define an element $f_i\in H^0(G_D(F),\cA_D^\cP(V^\Z,\Z)^{U^\cP})$.

Fix $g\in G_D(F^\cP)$. Since $V^R$ is the quotient of the space $C(\PP^1(F_\cP),R)$ modulo the subspace generated by $1_{\PP^1(F_\cP)}$, we can interpret $f_i(g)\in\Hom(V^\Z,\Z)$ as a distribution of $\PP^1(F_\cP)$ with integral values and such that $\int_{\PP^1(F_\cP)}df_i(g)=0$. Hence, it makes sense to consider the corresponding multiplicative integral. Let $\dH_p(\bar\Q_p):=\PP^1(\bar\Q_p)\setminus\PP^1(F_\cP)$ be the $p$-adic Poincar\'e hyperplane, write $\Delta_\cP=\Z[\dH_\cP]$, equipped with the natural degree morphism $\deg:\Delta_\cP\rightarrow\Z$, and let $\Delta_\cP^0$ be the kernel of $\deg$. We can identify $Y$ with the $\Z$-module generated by the $\{f_i\}_{i=1,\cdots,d}$.
Similarly as in \S \ref{CohShi}, we define ${\rm ev}_p(\phi)\in H^0(G_D(F),\cA_D^\cP(\Delta_\cP^0,\Hom(Y,\bar\Q_p^\times))^{U^\cP})$ by
\[
{\rm ev}_p(\phi)(g)(z_1-z_2)(f_i):=\mint_{\PP^1(F_\cP)}\left(\frac{z_2-t}{z_1-t}\right)df_i(g)(t),\qquad g\in G_D(\hat F^\cP), z_1,z_2\in\dH_\cP.
\] 
Note that, in this situation, we also have the exact sequence given by the degree map $\deg$:
\begin{equation}\label{exseqDeltaP}
0\longrightarrow\cA_D^\cP(\Hom(Y,\bar\Q_p^\times))\stackrel{\deg^\ast}{\longrightarrow}\cA_D^\cP(\Delta_\cP,\Hom(Y,\bar\Q_p^\times))\longrightarrow\cA_D^\cP(\Delta_\cP^0,\Hom(Y,\bar\Q_p^\times))\longrightarrow 0.
\end{equation}
We consider the image of ${\rm ev}_p(\phi)$ under the connection morphism
\[
H^0(G_D(F),\cA_D^\cP(\Delta_\cP^0,\Hom(Y,\bar\Q_p^\times))^{U^\cP})\stackrel{\partial}{\longrightarrow} H^1(G_D(F),\cA_D^\cP(\Hom(Y,\bar\Q_p^\times))^{U^\cP}).
\]
Using the $p$-adic uniformization of the Shimura curve $X_U$ and Manin-Drinfeld Theorem on the Jacobian of Mumford curves (see \cite{Dasg}), one can show that, in fact, $\partial({\rm ev}_p(\phi))\in H^1(G_D(F),\cA_D^\cP(X)^{U^\cP})$.
\begin{proposition}
Let $\ell:F_\cP^\times\rightarrow \Z_p$ be a continuous homomorphism. We have that
\[
f_i\cup c_\ell=\beta_\ell(\partial({\rm ev}_p(\phi)))(f_i)\in H^1(G_D(F),\cA_D^{\cP}(\C_p)^{U^\cP}),
\] 
for any $i=1,\cdots, d$.
\end{proposition}
\begin{proof}
Let $\bar\ell:\bar\Q_p^\times\rightarrow\bar\Q_p$ be any extension of $\ell$, namely a continuous additive morphism such that $\bar\ell\mid_{F_\cP^\times}=\ell$ (one can always find such an extension composing $\ell$ with the norm map on each finite extension of $F_\cP$ and dividing by the corresponding degree).
For any $z\in\dH_p$, we have the class of $(c_{\bar\ell}(z),1)\in\mathscr{E}(\ell)$, where 
\[
c_{\bar\ell}(z)\left(\begin{array}{cc}a&b\\c&d\end{array}\right)=\bar\ell\left(cz+d\right),
\]
It is easy to check that
\[
c_{\bar\ell}(z)(hg)=c_{\bar\ell}(gz)(h)+c_{\bar\ell}(z)(g),\qquad g,h\in G_D(F_\cP).
\]
Since the function $h\mapsto c_{\bar\ell}(z)(g)$ is obviously constant, we deduce that $g(c_{\bar\ell}(z),1)=(c_{\bar\ell}(gz),1)$, for all $g\in G_D(F_\cP)$. 

Let $\tilde f_i\in \cA_D^\cP(\mathscr{E}(\ell),\bar\Q_p)$ be any pre-image of $f_i\in H^0(G_D(F),\cA_D^\cP(V^\Z,\Z))$. Hence, for all $\gamma\in G_D(F)$ and all $g\in G_D(\hat F^\cP)$,
\begin{eqnarray*}
(f_i\cap c_\ell)(\gamma)(g)&=&\gamma\tilde f_i(g)(c_{\bar\ell}(z),1)-\tilde f_i(g)(c_{\bar\ell}(z),1)\\
%&=&\tilde f_i(\gamma^{-1}g)(\gamma^{-1}c_{\bar\ell}(z),1)-\tilde f_i(g)(c_{\bar\ell}(z),1)\\
&=&f_i(g)(c_{\bar\ell}(\gamma^{-1}z)-c_{\bar\ell}(z))+b_i(\gamma)(g)\\
&\equiv& \int_{\PP^1(F_\cP)}\bar\ell\left(\frac{\gamma^{-1} z-t}{z-t}\right)df_i(g)(t)=\bar\ell\left({\rm ev}_p(\phi)(g)(\gamma^{-1}z-z)(f_i)\right)\\
&=&\beta_\ell(\partial({\rm ev}_p(\phi)))(f_i)(\gamma)(g),
\end{eqnarray*}
where $b_i$ corresponds to the 1-coboundary $b_i(\gamma)=(1-\gamma)\tilde f_i(c_{\bar\ell}(z),1)$. Hence the result follows.
\end{proof}
This proposition implies that the automorphic $\cL$-invariant $\cL_\cP(\pi^D,\ell)$ coincides with the geometric $\cL$-invariant $\cL_\cP(A,\ell)$ in the definite setting. Such a claim in the indefinite setting is equivalent to Conjecture \ref{conjLinv}.

\subsection{$\cL$-invariants and Heegner points}

Throughout this section, assume that $B$ is a quaternion algebra that splits at a single archimedean place $\sigma$.
%As in previous sections, let ${\bf C}$ be a field endowed with discrete topology and containing the field of definition $L$ of $\Pi$. 
Let $\phi\in H^0(G_B(F),\cA^\sigma(\cD,\C))^U$ be a generator of $\pi^B\mid_{G(\hat F)}$, corresponding to a differential form $\omega_\phi\in\Omega^1_{X_U}$ of the cotangent space of $X_U/F$.
Recall that the abelian variety of $\GL_2$-type $A$ associated with $\pi$ is provided by the complex torus
\[
A(\C)\simeq (\C\otimes_\Q L)/\Lambda_\pi,\qquad \Lambda_\pi=\left\{\left(\int_c{^\tau}\omega_\phi\right)_{\tau\in\Gal(L_\pi/\Q)},\;c\in H_1(X_U,\Z)\right\}.
\]
%Let $T$ be the torus associated with $E^\times/F^\times$, where $E/F$ is an imaginary quadratic extension admitting 
Fix an embedding $K\hookrightarrow B$, and let $\tau_K\in\dH$ as in \S \ref{ShimCM}.
Let $\Delta_T=\Z[G_B(F)^+\tau_B]\simeq \Ind^{G_B(F)^+}_{T(F)}1_\Z$ be the set of divisors supported on $G_B(F)^+\tau_K$, and $\Delta_T^0=\ker(\deg:\Delta_T\rightarrow\Z)$ the set of degree zero divisors. 

%In \S \ref{ShimCM}, we constructed $\varphi\in H^0(T,\cA_f(CM_K^L))_\Pi$, where $CM_K^{\bf C}=CM_K\otimes_{\cO_L}{\bf C}$ and $CM_K\subseteq A_\Pi(E^{ab})$ is the set of Heegner points. We claim that such morphism is unique up to constant. Indeed, since $\Hom_{T(\hat F)}(\Pi\mid_{G(\hat F)},\C)$ is at least 1-dimensional by Saito-Tunnel (Proposition \ref{Saito-Tunnel}), and we have the injective morphism
%\[
%h:H^0(T,\cA_f(CM_K^L))_\Pi\longrightarrow \Hom_{T(\hat F)}(\Pi\mid_{G(\hat F)},\C);\qquad h(\varphi)(v)=NT(\varphi(v)(1)),
%\]
%for all $v\in\Pi\mid_{G(\hat F)}$, where $NT$ is the Neron-Tate height, we deduce that the cohomology group $H^0(T,\cA_f(CM_K^{\bf C}))_\Pi$ is at least 1-dimensional. In fact, $\varphi$ lies in $H^0(T,\cA_f(CM_K)\otimes_{\cO_L}L)_\Pi$. Hence, if $\varphi\neq 0$,
%\[
%H^0(T,\cA_f(CM_K^{\bf C}))_\Pi=H^0(T,\cA_f(CM_K)\otimes_{\cO_L}{\bf C})_\Pi={\bf C}\varphi.
%\]

Let 
\[
H_1(X_U,\Z)^\pm:=\{c\in H_1(X_U,\Z);\quad \bar c=\pm c\},
\]
where $\bar c$ is the complex conjugated path. Then, we have that 
\[
\cI_\pi^\pm:=\left\{C^\pm(\omega_\phi)(c)=\left(\int_c{^\tau}\omega_\phi\right)_{\tau\in\Gal(L/\Q)},\;c\in H_1^\pm(X_U,\Z)\right\}
\]
is a $\cO_{L_\pi}$-module locally free of rank 1.
%Indeed, by means of $\partial^\pm$ of \eqref{CompMor}, we can identify $H^1(G(F),\cA_f(\C)^U)$ with the dual of both $H_1^\pm(X_U,\Q)$, hence the claim follows from the fact that $H^1(G(F),\cA_f(\C))_\Pi=H^1(G(F),\cA_f(L))_\Pi\otimes_L\C$. 
It is clear that $\cI_\pi^+ +\cI_\pi^-\subseteq \Lambda_\pi$ has finite index. Moreover, we have the following surjective morphisms with finite cokernel,
\[
W^+/\cI_\pi^+\longrightarrow A(\R):=A(\C)^+,\qquad W^-/\cI_\pi^-\longrightarrow A(\C)^-,
\]
where $W^\pm=\cI_\pi^\pm\otimes\R\subseteq\C\otimes L_\pi$ is the subspace where complex conjugation acts by $\pm 1$, and analogously for $A(\C)^\pm$.

Since $\cI_\pi^\pm\otimes\Q=L_\pi\Omega_\pi^\pm$, for some $\Omega_\pi^\pm\in\C\otimes_\Q L_\pi$, the exact sequences 
\[
0\longrightarrow L_\pi\longrightarrow W^+\longrightarrow A^0(\R)\longrightarrow 0,\qquad 0\longrightarrow \Delta_T^0\longrightarrow \Ind_{T(F)}^{G_B(F)^+}1_\Z\longrightarrow\Z\longrightarrow 0,
\]
provide the commutative diagram:
\begin{equation}\label{bigdiag}
\small{\xymatrix{
&H^0(G_B(F)^+,\cA_B(A^0(\R)))\ar[d]\ar[r]&H^1(G_B(F)^+,\cA_B(L_\pi))\ar[d]^{\rm res}\\
H^0(T(F),\cA_B(W^+))\ar[r]\ar[d]&H^0(T(F),\cA_B(A^0(\R)))\ar[d]\ar[r]_{d_2}&H^1(T(F),\cA_B(L_\pi))\ar[d]^{r}\\
H^0(G_B(F)^+,\cA_B(\Delta_T^0,W^+))\ar[r]\ar[d]^{d_1}&H^0(G_B(F)^+,\cA_B(\Delta_T^0,A^0(\R)))\ar[r]\ar[d]&H^1(G_B(F)^+,\cA_B(\Delta_T^0,L_\pi))\\
H^1(G_B(F)^+,\cA_B(W^+))\ar[r]_{p_1}&H^1(G_B(F)^+,\cA_B(A^0(\R)))&
}}
\end{equation}

Moreover, we have an analogous commutative diagram for $A(\C)^-$.
The evaluation $ev_0^+$ \eqref{comdiagev} of ${^\tau}\omega_\phi$, $\tau\in\Gal(L_\pi/\Q)$, gives rise to an element $\varphi\in H^0(G_B(F)^+,\cA_B(\Delta_T^0,W^+))$ such that $d_1(\varphi)\neq 0$ (see Lemma \ref{pairings}) and $p_1\circ d_1(\varphi)=0$ (see Lemma \ref{lempartial0}). Hence, we obtain $\bar\varphi\in H^0(T(F),\cA_B(A^0(\R)))_{\pi^B}$, which is unique because $H^0(G_B(F)^+,\cA_f(A^0(\R)))_{\pi^B}=0$. We claim that $d_2(\bar\varphi)\neq 0$. Indeed, if $d_2(\bar\varphi)= 0$ then there should exists some pre-image of $\varphi$ in $H^0(T(F),\cA_B(W^+))$, and this contradicts the fact that $d_1(\varphi)\neq 0$. A similar diagram chasing argument shows that $r\circ d_2(\bar\varphi)=0$. Thus, there exists $\phi_1^+\in H^1(G_B(F)^+,\cA_B(L_\pi))_{\pi^B}$ such that ${\rm res}(\phi_1^+)=d_2(\bar\varphi)$.

The analogous commutative diagram for $A(\C)^-$ produces an element $\bar\varphi^-\in H^0(T(F),\cA_B(A^0(\C)^-))_{\pi^B}$, such that ${\rm res}(\phi_1^-)=d_2^-(\bar\varphi^-)$, for some $\phi_1^-\in H^1(G_B(F)^+,\cA_B(L_\pi))_{\pi^B}$, where 
\[
d_2^-:H^0(T(F),\cA_B(A^0(\C)^-))\longrightarrow H^1(T(F),\cA_B(L_\pi)),
\]
is analogous to $d_2$. By Remark \ref{CCinHOMOL}, it is clear that $\phi_1^\pm\in H^1(G_B(F)^+,\cA_B(L_\pi))^{\pm}_{\pi^B}$.

By equation \eqref{2nd} and Shimura's reciprocity law, $\bar \varphi+\bar\varphi^-$ defines, in fact, an element of 
$\Phi_T\in H^0(T(F),\cA_B(A(\bar\Q))\otimes_{\cO_{L_\pi}} {L_\pi})_{\pi^B}$. %=H^0(T,\cA_f(V^L,A_\Pi(\bar\Q))\otimes_{\cO_L} L)_\Pi.
Rescaling the generator $\phi\in\pi^B\mid_{G_B(\hat F)}$ if necessary, we can always assume that $\Phi_T(\phi)\in H^0(T(F),\cA_B(A(\bar\Q)))$. Since $\pi^B_\cP$ is Steinberg, we have that $\phi\in G_B(F_\cP)\phi\simeq%\Ind_{K_0(\cP)F_\cP^\times}^{G(F_\cP)}1_{\cO_L}/(\cU-1)\simeq 
V^{\cO_{L_\pi}}$. Hence $\Phi_T(\phi)$ defines an element 
$\Phi_T\mid_{V^{\cO_{L_\pi}}}\in H^0(T(F),\cA_B^\cP(V^{\cO_{L_\pi}},A(\C_p)))$. By means of such identifications, we consider 
\[
\Phi_T\in H^0(T(F),\cA_B^\cP(V^{\cO_{L_\pi}},A(\C_p))\otimes {L_\pi})_{\pi^B}. 
\]
%Since $\Phi_T(\phi)=n\Phi_T(\phi/n)$, for any $n\in\N$, we have $\Phi_T(\phi)\in H^0(T,\cA_f(A_\Pi(\bar\Q)_{tf}))$, where $A_\Pi(\bar\Q)_{tf}$ is the torsion free submodule of $A_\Pi(\bar\Q)$.
%Recall that the abelian variety $A_\Pi(\C_p)\simeq \Hom(Y,\C_p^\times)/X$, for some locally free $\cO_L$-modules of rank one $X$ and $Y$. Let $\cO=\cO_{\C_p}$ be the integer ring of $\C_p$. Since $\alpha:X\otimes\Q\rightarrow \Hom(Y\otimes\Q,\Q)$ is an isomorphism, we can identify the torsion free submodule $A_\Pi(\C_p)_{tf}$ of $A_\Pi(\C_p)$ with the image of
%\[
%\exp_p^\ast:\Hom(Y,\cO)\hookrightarrow  \Hom(Y,\C_p^\times)/X\simeq A_\Pi(\C_p),
%\]
%where $\exp_p:\cO\rightarrow\C_p^\times$ is the $p$-adic exponential morphism. 

Notice that, multiplication by $p^n$ provides the following exact sequence of $T(F)$-modules
\begin{equation}\label{exseqpn}
0\longrightarrow\cA^\cP_B(V^{\cO_{L_\pi}},A[p^n](\bar\Q))\longrightarrow \cA_B^\cP(V^{\cO_{L_\pi}},A(\C_p))\stackrel{p^n}{\longrightarrow}\cA_B^\cP(V^{\cO_{L_\pi}},A(\C_p))\longrightarrow 0.
\end{equation}
%where $A_\Pi'\simeq \Hom(\omega^{-n}Y,\bar\Q_p^\times)/\omega^nX$, we have the isogeny given by restriction 
%\[
%\varphi_{\varpi^n}:A_\Pi'\simeq \Hom(\omega^{-n}Y,\bar\Q_p^\times)/\omega^nX\longrightarrow  \Hom(Y,\bar\Q_p^\times)/X\simeq A_\Pi,  
%\]
%$(p^n \phi)(g)=p^n\phi(g)$, for all $\phi\in\cA_f^\cP(A_\Pi(\C_p))$ and $g\in G(\hat F)$. 
Since $\Lambda_\pi/p^n\Lambda_\pi\simeq A[p^n](\bar\Q)$, %applying $\Hom_{G(F_\cP)}(V^{\cO_L},\cdot)$ to the exact sequence \eqref{exseqpn} 
we obtain the connection morphism 
\[
d^n:H^0(T(F),\cA_B^\cP(V^{\cO_{L_\pi}},A(\C_p)))\longrightarrow H^1(T(F),\cA_B^\cP(V^{\cO_{L_\pi}},\Lambda_\pi/p^n\Lambda_\pi)),
\]
for all $n\in\N$. It is easy to check that 
\begin{equation}\label{BetiEtale}
d^n(\Phi_T\mid_{V^{\cO_{L_\pi}}})=\Omega_\pi^+d_2(\bar\varphi\mid_{V^{\cO_{L_\pi}}})+\Omega_\pi^-d_2^-(\bar\varphi^-\mid_{V^{\cO_{L_\pi}}})\mod p^n.
\end{equation}
%under the identification $\phi\in G(F_\cP)\phi\simeq V^{\cO_L}$.

Recall that the abelian variety $A(\C_p)\simeq \Hom(Y,\C_p^\times)/X$, for some locally free $\cO_L$-modules of rank one $X$ and $Y$. %Let $\cO=\cO_{\C_p}$ be the integer ring of $\C_p$. 
Since $\alpha:X\otimes\Q\rightarrow \Hom(Y\otimes\Q,\Q)$ is an isomorphism, we can identify the image of $A(\C_p)$ in $A^0(\C_p)=A(\C_p)\otimes {L_\pi}$ (killing all the torsion) with the image of the composition
\[
\Hom(Y,\cO_{\C_p})\stackrel{(\exp_p)_\ast}{\hookrightarrow}  \Hom(Y,\C_p^\times)/X\simeq A(\C_p)\longrightarrow A^0(\C_p),
\]
where $\exp_p:\cO_{\C_p}\rightarrow\C_p^\times$ is the $p$-adic exponential morphism. Hence, we can assume that $\Phi_T\mid_{V^{\cO_{L_\pi}}}\in H^0(T,\cA^\cP_f(V^{\cO_{L_\pi}},\Hom(Y,\cO_{\C_p})))$.

Given a continuous homomorphism $\ell:F^\times_\cP\rightarrow\Z_p$, we consider $c_\ell\in H^1(G_B(F_\cP),V^{\Z_p})$ as above. Since $\Hom(V^{\cO_{L_\pi}},\Hom(Y,\cO_{\C_p}))$ can be interpreted as a space of bounded distributions,  the cup product by $c_\ell$ provides following commutative diagram
\[
\xymatrix{
H^0(T(F),\cA_B^\cP(V^{\cO_{L_\pi}},\Hom(Y,\cO_{\C_p})))\ar[r]^{d^n}\ar[d]^{\cup c_\ell}&H^1(T(F),\cA_B^\cP(V^{\cO_{L_\pi}},\Lambda_\pi/p^n\Lambda_\pi))\ar[d]^{\cup c_\ell}\\
H^1(T(F),\cA_B^\cP(\Hom(Y,\cO_{\C_p})))\ar[r]^{d_1^n}&H^2(T(F),\cA_B^\cP(\Lambda_\pi/p^n\Lambda_\pi))
}
\]  
Let $\cL_\cP(\pi^B,\ell)'\in L_\pi\otimes \Q_p$ be the element whose components are $\cL_\cP((\pi^B)^\tau,\ell)$, $\tau\in\Gal(L_\pi/\Q)$. The above commutative diagram together with \eqref{BetiEtale} and Proposition \ref{L-inv-} imply that
\begin{equation*}
d^n_1(\Phi_T\mid_{V^{\cO_{L_\pi}}}\cup c_\ell)=\cL_\cP(\pi^B,\ell)'d_1^n(\Phi_T\mid_{V^{\cO_{L_\pi}}}\cup c_\ord).
\end{equation*}
Since $\ker(d_1^n)=p^n H^1(T(F),\cA_B^\cP(A(\C_p)))$ by  \eqref{exseqpn}, the image of $\Phi_T\mid_{V^{\cO_{L_\pi}}}\cup c_\ell-\cL_\cP(\pi^B,\ell)'\left(\Phi_T\mid_{V^{\cO_{L_\pi}}}\cup c_\ord\right)$ in $H^1(T(F),\cA_B^\cP(\Hom(Y,\cO_{\C_p}/p^n\cO_{\C_p})))$ is 0 for all $n$. We conclude that
\[
\Phi_T\cup c_\ell=\cL_\cP(\pi^B,\ell)'\left(\Phi_T\cup c_\ord\right).
\]
When we apply $\log_{\omega_\phi}$ the formal group logarithm attached to the differential $\omega_\phi$, we recover the component corresponding to the fixed embedding $\bar\Q\hookrightarrow\bar\Q_p$. Hence
\begin{equation}\label{perfi}
\log_{\omega_\phi}(\Phi_T\cup c_\ell)=\cL_\cP(\pi^B,\ell)\left(\log_{\omega_\phi}(\Phi_T\cup c_\ord)\right).
\end{equation}

\section{Exceptional zero phenomenon in the split case}\label{CaseSplit}

The aim of the rest of the paper is to compute the class of $L^{\rm def}_\cP(\pi_K)$ and $L^{\rm ind}_\cP(\pi_K)$ in $\cI/\cI^2$ in the presence of the Exceptional Zero phenomenon and in case $K$ splits at $\cP$. We will invoke Corollary \ref{charI/I2} and so such class will be characterized by the integrals
\[
\int_{\cG_{K,\cP}}\ell d\mu_{K,\cP}^{\bullet}=\kappa(\psi_\bullet)\cap\partial(\ell),\qquad \bullet={\rm def},{\rm ind},\quad\psi_{\rm def}:={\rm res}\phi,\;\psi_{\rm ind}:=\log\phi,
\]  
for all $\ell\in\Hom_{\Z_p}(\cG_{K,\cP},\Z_p)=\cG_{K,\cP}^\vee$.

Let $H$ be the maximal open compact subgroup of $T(F_\cP)\simeq F_\cP^\times$. We have the exact sequence
\[
0\longrightarrow H\longrightarrow T(F_\cP) \stackrel{\ord}{\longrightarrow}\Z\longrightarrow 0.
\]
Then we can consider the real manifold $M=\R\times T(\hat F^\cP)/\Gamma$ with the following natural action of $T(F)$:
\[
T(F)\times M\longrightarrow M;\quad (t,(x,t^\cP))\longmapsto (x+\ord(\iota_\cP(t)),\iota^\cP(t)t^\cP).
\]
Since we can identify $H_0(M,\Z)$ with $C_c(T(\hat F^\cP)/\Gamma,\Z)$, we can consider the fundamental class $\vartheta$ of the oriented compact manifold $M/T(F)$ as an element of $H_1(T(F),C_c(T(\hat F^\cP)/\Gamma,\Z))$ by means of the identifications
\[
\vartheta\in H_1(M/T(F),\Z)=H_1(T(F),H_0(M,\Z))=H_1(T(F),C_c(T(\hat F^\cP)/\Gamma,\Z)).
\]

Let $\ell:\cG_{K,\cP}\rightarrow\Z_p$ be a continuous group homomorphism ($\ell\in C(\cG_{K,\cP},\Z_p)$ not necessarily locally constant). It corresponds to a continuous homomorphism $\hat\ell:T(\hat F) \rightarrow\Z_p$ that factors through $\Gamma T(F)$. Let $\ell_\cP:T(F_\cP)\rightarrow\Z_p$ and $\ell^\cP:T(\hat F^\cP)\rightarrow\Z_p$ be its corresponding restriction to $T(F_\cP)$ and $T(\hat F^\cP)$. By topological reasons, $\ell^\cP=0$. Let $\cF$ be a fundamental compact domain for $T(\hat F)/\Gamma$ under the action of $T(F)$. Hence $\partial\ell=[\hat\ell 1_{\cF}]$, where $[\hat\ell 1_{\cF}]\in H_0(T(F),C_c(T(\hat F)/\Gamma,\Z_p))$ is the image of $\hat\ell 1_{\cF}$ and $1_\cF\in C_c(T(\hat F)/\Gamma,\Z)$ is the characteristic function of $\cF$. As above, we consider the natural injections $\iota_\cP:T(F)\hookrightarrow T(F_\cP)$ and $\iota^\cP:T(F)\hookrightarrow T(\hat F^\cP)$.

Let us consider again the open compact subset $U=x_1\cup\bigcup_{n\in\N}\varpi^n H$ of $\PP^1(F)$.
Since $\ell^\cP=0$, the morphism $\ell$ depends only on $\ell_\cP$. Hence it makes sense to consider the cocycle $z_{\ell_\cP}\in H^1(T(F),C_c(T(F_\cP),\Z_p))$ defined by
\[
z_{\ell_\cP}(t):=(1-\iota_\cP(t))(\ell_\cP 1_U).
\]
\begin{remark}\label{remaux}
Note that $z_{\ell_\cP}(t)\in C_c(T(F_\cP),\Z_p)\subset C_\diamond(T(F_\cP),\Z_p)$. Indeed,
\begin{eqnarray*}
z_{\ell_\cP}(t)(x)&=&\ell_\cP(x) 1_U(x)-\ell_\cP(\iota_\cP(t)^{-1}x) 1_U(\iota_\cP(t)^{-1}x)\\
&=&\ell_\cP(x) 1_U(x)-(\ell_\cP(x) -\ell_\cP(\iota_\cP(t)))1_U(\iota_\cP(t)^{-1}x)\\
&=&\pm \ell_\cP 1_{U_t}(x)+\ell_\cP(\iota_\cP(t)))\iota_\cP(t)1_U(x),
\end{eqnarray*}
where $U_t=U\setminus \iota_\cP(t)U\subset T(F_\cP)$ is clearly open and compact. Since $\hat\ell(T)=\ell_\cP(\iota_\cP(t))=0$, we deduce that $z_{\ell_\cP}(t)=\pm \ell_\cP 1_{U_t}(x)\in C_c(T(F_\cP),\Z_p)$.
\end{remark}
\begin{proposition}\label{comppart}
We have that 
\[
\partial(\ell)=\vartheta\cap z_{\ell_\cP}\in H_0(T(F), C_c(T(\hat F)/\Gamma,\Z_p)).
\]
\end{proposition}
\begin{proof}
Let $\cF\subset T(\hat F)/\Gamma$ be an open compact fundamental domain for the action of $T(F)$ such that $H\cF=\cF$. By definition, $\partial(\ell)$  is the image of the compactly supported function $\hat\ell1_\cF$ in $H_0(T(F),C_c(T(\hat F)/\Gamma,\Z_p))$. 

Let us consider a finite index subgroup of the form $\cT=t^\Z\times \cT^\cP\subseteq T(F)$, such that $\iota_\cP(\cT^\cP)\subset H$ and $\iota^\cP(t)\in\Gamma$. Let %$G\subseteq T(F_\cP)/H$ be the subgroup generated by $t$ and 
$\cF^\cP$ a fundamental domain for $T(F^\cP)/\Gamma$ under the action of $\cT^\cP$. Hence, $U_t\times\cF^\cP=(U\setminus \iota_\cP(t)U)\times\cF^\cP$ is a fundamental domain for $T(\hat F)/\Gamma$ under the action of $\cT$.
Note that, in this situation, $C_c(T(\hat F^\cP)/\Gamma,\Z)\simeq C(\cF^\cP,\Z)\otimes_{\Z}\Z[\cT^\cP] \simeq C(\cF^\cP,\Z)\otimes_{\Z[t^\Z]}\Z[\cT]$. Thus, by Shapiro's Lemma, $H_1(\cT,C_c(T(\hat F^\cP)/\Gamma,\Z))=H_1(t^\Z,C(\cF^\cP,\Z))$. We describe $\vartheta$ as the co-restriction of the class in $H_1(t^\Z,C(\cF^\cP,\Q))$ defined by the cocycle 
\[
f:t^\Z\longrightarrow C(\cF^\cP,\Q);\qquad f(t^n)=\left\{\begin{array}{cc}\frac{1}{[T(F):\cT]}1_{\cF^\cP},&n=1,\\0,&n\neq 1.\end{array}\right.
\]
Therefore, we compute
\[
\vartheta\cap z_{\ell_\cP}=\frac{1}{[T(F):\cT]}z_{\ell_\cP}(t)\otimes 1_{\cF^\cP}=\frac{1}{[T(F):\cT]}\ell_\cP 1_{U_t}\otimes 1_{\cF^\cP}=\frac{1}{[T(F):\cT]}\hat\ell 1_{U_t\times \cF^\cP},
\]
where the second equality has been obtained from Remark \ref{remaux}. Since $1_{U_t\times\cF^\cP}=\sum_{g\in T(F)/\cT}g1_\cF$ and $\hat\ell$ is trivial on $T(F)$, we deduce that the class of $\hat\ell 1_{U_t\times \cF^\cP}=\sum_{g\in T(F)/\cT}g(\hat\ell1_\cF)$ coincides with the class of $[T(F):\cT]\hat\ell 1_\cF$ in $H_0(T(F), C_c(T(\hat F)/\Gamma,\Z_p))$, hence the result follows.
\end{proof}

Let us consider the subset
\[
T_1=\{x\in T(F):\,\iota_\cP(x)\in H\}\subset T(F)
\]
and let $\cF_1$ be a fundamental domain for $T(\hat F^\cP)/\Gamma$ under the action of $T_1$. We can consider the subgroup 
\[
\cX:=U\times \cF_1\subset T(\hat F)/\Gamma.
\]
As in previous results, let us identify $V^\Z$ as a quotient of $C(\PP^1,\Z)$.
Hence $\ord:T(F_\cP)\rightarrow\Z$ defines a cocycle $z_{\ord}\in H^1(T(F_\cP), V^\Z)$, where $z_{ord}(t)=(1-\iota_\cP(t))(\ord 1_U)$.

\begin{proposition}\label{classord}
The class of the characteristic function $1_\cX$ satisfies
\[
[1_\cX]=\vartheta\cap z_{\ord}\in H_0(T(F),C_c(T(\hat F)/\Gamma,\C_p)_0).
\]
\end{proposition}
\begin{proof}
Since $\iota_\cP(T(F))$ is dense in $T(F_\cP)$, we have the exact sequence
\[
0\longrightarrow T_1\longrightarrow T(F)\stackrel{\ord}{\longrightarrow}\Z\longrightarrow 0,
\] 
where, by abuse of notation, we also denote by $\ord$ the composition $\ord\circ\iota_\cP$. Note that $z_{\ord}$ is the image through the inflation map of an element $\bar z_\ord\in H_1(T(F)/T_1,(V^\Z)^{T_1})$.

On the one hand, let us consider $1_{\cF_1}$ the characteristic function of $\cF_1$. By strong approximation, given $t\in T(F)$ there exists $t_1\in T_1$ such that $\iota^\cP(t^{-1}t_1)\in \Gamma$. This implies that the image $[1_{\cF_1}]\in H_0(T_1,C_c(T(\hat F^\cP)/\Gamma,\Z))$ lies in fact in $[1_{\cF_1}]\in H^0(T(F)/T_1,H_0(T_1,C_c(T(\hat F^\cP)/\Gamma,\Z)))$.

On the other hand, let us consider $\vartheta_1\in H_1(T(F)/T_1,\Z)\simeq H_1(\R/T(F),\Z)$ given by the fundamental class of $\R/T(F)$ ($t\in T(F)$ acts on $\R$ by $tx=x+\ord(\iota_\cP(t))$ as above).
Then it is clear that $[1_{\cF_1}]\otimes\vartheta_1$ is mapped to $\vartheta$ by means of the composition
\[
\xymatrix{
H^0(T(F)/T_1,H_0(T_1,C_c(T(\hat F^\cP)/\Gamma,\Z)))\otimes H_1(T(F)/T_1,\Z)\ar[d]^{\cap}&\\
H_1(T(F)/T_1,H_0(T_1,C_c(T(\hat F^\cP)/\Gamma,\Z)))\ar[r]& H_1(T(F),C_c(T(\hat F^\cP)/\Gamma,\Z)).
}
\]
Thus, the result follows if we show that $\bar z_\ord\cap \vartheta_1=[1_{U}]\in H_0(T(F)/T_1, (V^\Z)^{T_1})$. Indeed, if $t\in T(F)/T_1$ is a generator then
\begin{eqnarray*}
\bar z_\ord\cap \vartheta_1&=&\bar z_\ord(t)=(1-\iota_\cP(t))(\ord 1_U)=\ord 1_U-(\iota_\cP(t)\ord)(\iota_\cP(t) 1_U)\\
&=&\ord(1_U-\iota_\cP(t)1_U)+\ord(\iota_\cP(t))\iota_\cP1_U=\ord\cdot 1_H+\iota_\cP(t) 1_U=\iota_\cP(t) 1_U.
\end{eqnarray*}
\end{proof}

%For this section, we assume that we are in the definite case and $T_\cP$ is split. This implies that $\pi^{JL}$ is aa automorphic representation of a definite quaternion algebra. 

For any continous $\Z_p$-module homomorphism $\ell:\cG_{K,\cP}\rightarrow \Z_p$, let $\hat\ell:T(\hat F)\rightarrow\Z_p$ be the composition $\ell\circ\rho_A$, where $\rho_A:T(\hat F)/\Gamma\rightarrow\cG_{K,\cP}$ is the Artin map, and let $\ell_\cP:T(F_\cP)\rightarrow\Z_p$ be its restriction to $T(F_\cP)$ (recall that $\ell^\cP=0$). We will identify $T(F_\cP)$ with $F_\cP^\times$ by means of the isomorphism $\psi$ of \S \ref{ExtStRep}.
Recall the automorphic $\cL$-invariants $\cL_\cP(\pi^{B},\ell_\cP)$ and $\cL_\cP(\pi^{D},\ell_\cP)$ % associated to $\ell_\cP$, $T$ and $\pi^{JL}$ 
 introduced in Definition \ref{defLinv}.
\begin{definition}
By Corollary \ref{charI/I2}, the morphism in $\Hom_{\Z_p}(\cG_{K,\cP}^\vee,\C_p)$ that maps $\ell\in\cG_{K,\cP}^\vee$ to $\cL_\cP(\pi^{B},\ell_\cP)$ or $\cL_\cP(\pi^{D},\ell_\cP)$ (depending we are in the indefinite or definite case)
 defines a class $\underline{\cL_\cP^\bullet}(\pi)\in\cI/\cI^2$, where $\bullet={\rm def}$ or ${\rm ind}$, and $\cI$ is the augmentation map of $\Lambda_{\C_p}$. Such class $\underline{\cL_\cP^\bullet}(\pi)$ is called \emph{automorphic $\cL$-invariant vector} attached to $\pi$.
\end{definition}

We have shown that $L_\cP^\bullet(\pi_K)\in\cI$ whenever $\pi_\cP$ Steinberg.
The main result of this section computes the class of  $L_\cP^\bullet(\pi_K)$ in $\cI/\cI^2$ in terms of the automorphic $\cL$-invariant vector.
\begin{theorem}\label{main-Res}
Let $C_K$ be the non-zero constant of Theorem \ref{intprop1} or Theorem \ref{intprop2}, depending on whether we are in the definite or indefinite case. If
 $T(F_\cP)$ splits and $\alpha=1$ then $L_\cP^\bullet(\pi_K)\in\cI$, with $\bullet={\rm def},{\rm ind}$, and its class $\nabla L_\cP^\bullet(\pi_K)\in \cI/\cI^2$ is given by the formula 
\begin{eqnarray*}
\nabla L_\cP^{\rm def}(\pi_K)&=&\underline{\cL_\cP^{\rm def}}(\pi)\left( C_K C(\pi_\cP)\frac{L(1/2,\pi_K,1)}{L(1,\pi,ad)}\right)^{1/2},\\
\nabla L_\cP^{\rm ind}(\pi_K)&=&\underline{\cL_\cP^{\rm ind}}(\pi)\log_p(P_T),
\end{eqnarray*}
where $C(\pi_\cP)=\frac{-L(1,\pi_\cP,ad)\zeta_\cP(1)^2}{L(1/2,\pi_\cP,1)}\neq 0$ and $P_T\in A(\bar \Q)\otimes_{\cO_{L_\pi}}\bar\Q$ is a Heegner point with Neron-Tate canonical height
\[
\langle P_T,P_T\rangle=\left|P_T\right|^2=C_KC(\pi_\cP)\frac{L'(1/2,\pi_K,1)}{L(1,\pi,ad)}.
\]
\end{theorem}
\begin{proof}
By Corollary \ref{charI/I2}, in order to obtain the class $\nabla L_\cP^{\bullet}(\pi_K)\in\cI/\cI^2$ we have to compute $\frac{\partial L_\cP^{\bullet}(\pi_K)}{\partial \ell}:=\int_{\cG_{K,\cP}}\ell d\mu_{K,\cP}^{\bullet}$, for all $\ell\in\cG_{K,\cP}^\vee$. By Proposition \ref{comppart},
\[
\frac{\partial L_\cP^{\bullet}(\pi_K)}{\partial \ell}=\kappa(\psi_{\bullet})\cap\partial \ell=\kappa(\psi_{\bullet})\cap(\vartheta\cap z_{\ell_\cP})=(\kappa(\psi_{\bullet})\cup z_{\ell_\cP})\cap\vartheta.
\]
%Let ${\bf C}$ denote $\C_p$ endowed with the discrete topology.
%recall that $\kappa_{\mu^i}\in H^0(T,\Dist(\hat T/\Gamma,\C_p))$, by Proposition \ref{dist-meas}. 

In the definite case (following the notation of \S \ref{Ap2}), $\psi_{\rm def}={\rm res}(\phi)$ is the image of 
\[
\phi\in \Hom_{G(F_\cP)}(V^{\bar\Q},H^0(G(F),\cA_D(\bar\Q))_{\pi^{D}}\simeq H^0(G(F),\cA_D^\cP(V^{\bar\Q},\bar\Q))_{\pi^{D}}.
\] 
through the restriction morphism. %${\rm res}:H^0(G(F),\cA_f^\cP(V^{\bar\Q},\bar\Q))^{U^\cP}\rightarrow H^0(T,\cA_f^\cP(V^{\bar\Q},\bar\Q))^{U^\cP}$.
This implies that, by Proposition \ref{eqcoc},
\begin{eqnarray*}
\kappa(\psi_{\rm def})\cup z_{\ell_\cP}&=&\kappa^\cP({\rm res}(\phi\cup c_{\ell_\cP}))=\cL_\cP(\pi^{D},\ell_\cP)\kappa^\cP({\rm res}(\phi\cup c_{\ord}))\\
&=&\cL_\cP(\pi^{D},\ell_\cP)(\kappa(\psi_{\rm def})\cup z_{\ord}),
\end{eqnarray*}
where $\kappa^\cP:\cA_D^\cP(\C_p)\rightarrow C_c(T(\hat F^\cP)/\Gamma,\C_p)_0^\vee$ is given by $\langle\kappa^\cP(\phi),f\rangle=\sum_{t\in T(\hat F^\cP)/\Gamma}f(t)\phi(t)$. 

In the indefinite case, $\psi_{\rm ind}=\log(\phi)=\log_{\omega_\phi}(\Phi_T)$, where $\Phi_T\in H^0(T(F),\cA^\cP_B(V^{\bar\Q},A^0(\bar\Q)))_{\pi^{B}}$ and $\log_{\omega_\phi}$ is the formal logarithm attached to $\omega_\phi$. By \eqref{perfi}, we have that 
\begin{eqnarray*}
\kappa(\psi_{\rm ind})\cup z_{\ell_\cP}&=&\kappa^\cP(\log_{\omega_\phi}(\Phi_T\cup c_{\ell_\cP}))=\cL_\cP(\pi^{B},\ell_\cP)\kappa^\cP(\log_{\omega_\phi}(\Phi_T\cup c_{\ord}))\\
&=&\cL_\cP(\pi^{B},\ell_\cP)(\kappa(\psi_{\rm ind})\cup z_{\ord}).
\end{eqnarray*}

In any of our settings (definite or indefinite) $\kappa(\psi_{\bullet})\cup z_{\ell_\cP}=\cL^\bullet(\ell)(\kappa(\psi_{\bullet})\cup z_{\ord})$, where $\cL^{\rm def}(\ell):=\cL_\cP(\pi^{D},\ell_\cP)$ or $\cL^{\rm ind}(\ell):=\cL_\cP(\pi^{B},\ell_\cP)$.
Thus
\[
\frac{\partial L_\cP^\bullet(\pi_K)}{\partial\ell}=\cL^\bullet(\ell)(\kappa(\psi_{\bullet})\cup z_{\ord})\cap\vartheta=%\cL_\cP(\pi^{JL},\ell^k_\cP)\kappa_1\cap(\vartheta\cap z_{\ell^k_\cP})=
\cL^\bullet(\ell)\int_{\cG_{K,\cP}}[1_\cX](\gamma)d\mu^{\bullet}_{K,\cP}(\gamma),
\]
by Proposition \ref{classord}. 

We compute that, since $1_\cX=1_{U\times\cF_1}$ is $H$-invariant,
\begin{eqnarray*}
\int_{\cG_{K,\cP}}[1_\cX](\gamma)d\mu^{\rm def}_{K,\cP}(\gamma)&=&\int_{T(\hat F)/T(F)}[1_\cX](x)\delta(1_H)(x)d^\times x,\\
\int_{\cG_{K,\cP}}[1_\cX](\gamma)d\mu^{\rm ind}_{K,\cP}(\gamma)&=&\int_{T(\hat F)/T(F)}[1_\cX](x)\log_p(\hat\delta(1_H)(x))d^\times x.
\end{eqnarray*}
Recall that $T(F_\cP)/H=\varpi^\Z$, $U=\cup_{n\in\N}\varpi^n H$ and $H\times\cF_1$ contains a finite number of fundamental domains of $T(\hat F)$ under the action of $T(F)$. This implies that 
\begin{eqnarray*}
\int_{\cG_{K,\cP}}[1_\cX](\gamma)d\mu^{\rm def}_{K,\cP}(\gamma)&=&\int_{T(\hat F)/T(F)}\sum_{n\in \N}[1_{H\times\cF_1}](\varpi^{-n}x)\delta(1_H)(x)d^\times x\\
&=&\int_{H\times\cF_1}\sum_{n\in \N}\pi^{D}(\varpi^{n})\delta(1_H)(x)d^\times x=\int_{T(\hat F)/T(F)}\delta(1_U)(x)d^\times x.
\end{eqnarray*}
Similarly,
\[
\int_{\cG_{K,\cP}}[1_\cX](\gamma)d\mu^{\rm ind}_{K,\cP}(\gamma)=\int_{T(\hat F)/T(F)}\log_p(\hat\delta(1_U)(x))d^\times x.
\]
Using the Waldspurger and Gross-Zagier formulas, we compute
\begin{eqnarray*}
\left(\int_{T(\hat F)/T(F)}\delta(1_U)(x)d^\times x\right)^2&=&C_\phi \frac{L(1/2,\pi_K,1)}{L(1,\pi,ad)}\alpha_{\pi_\cP,1}(\delta_{T}(1_U),\delta_{T}(1_U)),\\
\left|\int_{T(\hat F)/T(F)}\hat\delta(1_U)(x)d^\times x\right|^2&=&C_\phi \frac{L'(1/2,\pi_K,1)}{L(1,\pi,ad)}\alpha_{\pi_\cP,1}(\delta_{T}(1_U),\delta_{T}(1_U)).
\end{eqnarray*}
Thus, we need to compute the pairing $\beta_{\pi_\cP,1}(\delta_{T}(1_U),\delta_{T}(1_U))$. By Proposition \ref{innerprod},
\begin{eqnarray*}
\beta_{\pi_\cP,1}(\delta_{T}(1_U),\delta_{T}(1_U))&=&c_T\int_{T(F_\cP)}\int_{T(F_\cP)}1_U(t^{-1}x)\overline{\Lambda(\delta_{T}(1_U))(x)}d^\times xd^\times t.
\end{eqnarray*}
Hence by \eqref{eqntT}, 
$\beta_{\pi_\cP,1}(\delta_{T}(1_U),\delta_{T}(1_U))=c_T\bar C_T\int_{T(F_\cP)}F(0)(t) d^\times t$,
where $F(s)(t)$ is the analytic continuation of the expression
\[
F(s)(t)=\int_{T(F_\cP)}1_U(t^{-1}x)\int_{T(F_\cP)}\theta_{T}(s)(y)1_U(x^{-1}y)d^\times yd^\times x
\]
As we showed in the proof of Theorem \ref{teocompIT}, fixing an isomorphism $T_\cP\simeq F_\cP^\times$ (Note that $U$ is mapped to $\cO_{F_\cP}\setminus 0$), there is a non-zero constant $C$, such that $\theta_{T}(s)(y)=\left|\frac{y}{C^2(1-y)^2}\right|^{1-s}$. We deduce that
\begin{eqnarray*}
F(s)(t)&=&\int_{F_\cP^\times}1_{\cO_{F_\cP}}(t^{-1}x)\int_{F_\cP^\times}\left|\frac{y}{C^2(1-y)^2}\right|^{1-s}1_{\cO_{F_\cP}}(x^{-1}y)d^\times yd^\times x\\
&=&q^{(2-2s)n_T}\int_{t\cO_{F_\cP}}\int_{x\cO_{F_\cP}}\left|\frac{y}{(1-y)^2}\right|^{1-s}d^\times yd^\times x\\
&=&q^{(2-2s)n_T}\sum_{m_x=\ord(t)}^\infty\sum_{m_y=m_x}^\infty q^{(s-1)m_y}\int_{\varpi^{m_y}\cO_{F_\cP}^\times}\frac{1}{|1-y|^{2-2s}}d^\times y,
\end{eqnarray*}
where $n_T=\ord(C)$ and $q=\#(\cO_{F_\cP}/\varpi)$.
Moreover,
\[
\int_{\varpi^{m_y}\cO_{F_\cP}^\times}\frac{1}{|1-y|^{2-2s}}d^\times y=\left\{\begin{array}{lr}
1,& m_y>0;\\
q^{(2-2s)m_y},& m_y<0;\\
\frac{q-2+q^{1-2s}}{(q-1)(1-q^{1-2s})},& m_y=0.
\end{array}\right.
\]
By means of a tedious but straightforward computation that will be left to the reader, we obtain that 
\[
F(0)(t)=\left\{\begin{array}{lr}
q^{2n_T}q^{-\ord(t)}(1-q^{-1})^{-2},& \ord(t)>0;\\
q^{2n_T}q^{\ord(t)}(1-q)^{-2},& \ord(t)\leq 0.
\end{array}\right.
\]
This implies that
\begin{eqnarray*}
\alpha_{\pi_\cP,1}(\delta_{T}(1_U),\delta_{T}(1_U))&=&\frac{L(1,\eta_\cP)L(1,\pi_\cP,ad)}{\zeta_\cP(2)L(1/2,\pi_\cP,1)}c_T\bar C_T\int_{T(F_\cP)}F(0)(t) d^\times t\\
&=&\frac{c_T\bar C_Tq^{2n_T}L(1,\eta_\cP)L(1,\pi_\cP,ad)}{\zeta_\cP(2)L(1/2,\pi_\cP,1)(1-q^{-1})^2}\left(\sum_{n>0}q^{-n}+\sum_{n\leq 0}q^{n-2}\right)\\
&=&\frac{c_T\bar C_Tq^{2n_T}L(1,\eta_\cP)L(1,\pi_\cP,ad)}{\zeta_\cP(2)L(1/2,\pi_\cP,1)}\frac{q^{-1}(1+q^{-1})}{(1-q^{-1})^3},
\end{eqnarray*}
and the result follows.
\end{proof}

\begin{remark}
Recall that the constants $C_K$ have been computed in Theorem \ref{explIPthm} under certain mild hypothesis.
\end{remark}

%\begin{remark}
%Recall that the $\C_p$-valued measure $\mu_{II}$ has been obtained from a measure $\mu_{E,\cP}^\pi$ with coefficients in $A_\phi(\bar\Q)\otimes_{\cO_L}\C_p$ by means of the $p$-adic logarithm $\log_\cP$. Note that the above result is still valid for the measure $\mu_{E,\cP}^\pi$ itself.
%\end{remark}

\section{Appendix 1: Local integrals}\label{locint}

\begin{lemma}\cite[Proposition 2.1.5]{Bump}\label{lemmeasp}
Let $H$ be a locally compact group and $N$ a compact subgroup of $H$. Then there is a positive regular Borel measure on the quotient space $H/N$ that is invariant under the action of $H$ by left transaction. This measure is unique up to a constant multiple.
\end{lemma}
%\begin{example}\label{examplePP1}
%We have seen that $\PP^1\simeq G/B$ by means of the isomorphism $g\mapsto g\ast\infty$. Despite, we cannot apply the above lemma since $B$ is not compact. Nevertheless, by Iwasawa decomposition $\PP^1\simeq K/B\cap K$, hence there exist a Borel measure $\mu$ on $\PP^1$ invariant under the action of $K$. Clearly, such measure is given by
%\begin{equation}\label{defintPP1}
%\int_{\PP^1}fd\mu=\int_{K}f(\varphi(k))dk.
%\end{equation}
%\end{example}

Let $F$ be a nonarchimedean local field with absolute value $|\cdot|:F\rightarrow\R$ and integer ring $\cO$. Let us denote by $Z$ the center of $\GL_2(F)$, and let $P$ be its Borel subgroup, namely, the subgroup of upper triangular matrices. % and $K=\GL_2(\cO)$ a maximal compact subgroup.
Let us consider the modular quasicharacter
\[
\kappa:P\longrightarrow\R\qquad\kappa\left(\begin{array}{cc}
					  t_1 & x\\
					  & t_2
					\end{array}\right)=|t_1/t_2|.
\]
It is a group homomorphism that satisfies $d_R(b):=\kappa(b)d_L(b)$, where $d_R$ and $d_L$ are right and left Haar measures on $P$.

\begin{lemma}\label{lemmeas}
Let $M$ be a closed subgroup of $\GL_2(F)$, such that the product $PM$ is open in $\GL_2(F)$ and its complement has zero measure. Let $Z_M\subseteq Z\cap M$ be any subgroup, such that the quotient  $(P\cap M)/Z_M$ is compact. Let $h\in C_c(\GL_2(F),\C)$, such that $h(bg)=\kappa(b)h(g)$, for all $b\in P$ and $g\in \GL_2(F)$. Then, there exists $\phi\in C_c(\GL_2(F),\C)$, such that $h(g)=\int_P\phi(bg)d_Lb$, for all $g\in \GL_2(F)$. Moreover, for any such $\phi$,
\[
\int_{M/Z_M} h(m)d_Rm=C\int_{\GL_2(F)}\phi(g)dg,
\]
where $C=C_{M/Z_M}\in \R$ is a certain constant and $d_R$ is the right Haar measure on $M/Z_M$.
\end{lemma}
\begin{proof}
We choose a right $(P\cap \GL_2(\cO))$-invariant compactly supported function $\phi_0\in C_c(P,\C)$, such that $\int_{P}\phi_0(b)d_Lb=1$. 
Since $h$ is also left $(P\cap \GL_2(\cO))$-invariant, both $\phi_0$ and $h$ provide a function $\phi\in C_c(\GL_2(F),\C)$ defined by 
 $\phi(bk)=\phi_0(b)h(k)$, for all $b\in P$ and $k\in \GL_2(\cO)$. We compute ($g=b'k'$, $b'\in P$, $k'\in \GL_2(\cO)$)
\[
\int_P\phi(bg)d_Lb%=\int_B\phi(bb'g_0m')d_Lb
=h(k')\int_P\phi_0(bb')\kappa^{-1}(b)d_R(b)=h(k')\kappa(b')=h(g).
\]

Let $H:=(P\times M)/Z_M$, and $N:=(P\cap M)/Z_M$. We consider $N$ embedded diagonally in $H$. Observe that there is a homeomorphism $H/N\rightarrow PM$ given by $(b,m)N\mapsto bm^{-1}$. This induces a linear isomorphism $C_c(PM,\C)\simeq C_c(H/N,\C)$. The linear functional on $C_c(PM,\C)$ corresponding to the Haar measure of $\GL_2(F)$ gives rise to a linear functional on $C_c(H/N,\C)$, which is invariant under the left action of $H$. Another left $H$-invariant linear functional  on $C_c(H/N,\C)$ is given by $\int_{M/Z_M}\int_Pf((b,m^{-1})N)d_Lbd_Rm$.
By Lemma \ref{lemmeasp}, both linear functionals must coincide (up to constant). Hence
\[
C\int_{\GL_2(F)} \phi(g)dg=C\int_{PM} \phi(g)dg=\int_{M/Z_M}\int_B\phi(bm)d_Lbd_Rm=\int_{M/Z_M} h(m)d_Rm,
\]
and the result follows.
\end{proof}

\begin{corollary}[Invariance]\label{coras}
Let $h\in C_c(\GL_2(F),\C)$ and $M$ be as above. Then, for any $g\in \GL_2(F)$, 
$$
\int_{M/Z_M} h(mg)d_R m=\int_{M/Z_M} h(m)d_Rm.
$$
\end{corollary}
\begin{proof}
Write $h'\in C_c(\GL_2(F),\C)$ for the translation $h'(g')=h(g'g)$, for all $g'\in \GL_2(F)$. By the above lemma, there exists $\phi\in C_c(\GL_2(F),\C)$, such that $h(g')=\int_P\phi(bg')d_Lb$. Thus,
$$
h'(g')=h(g'g)=\int_P\phi(bg'g)d_Lb=\int_P\phi'(bg')d_Lb,
$$
where $\phi'(g'):=\phi(g'g)$. Applying the second part of the lemma,
\begin{eqnarray*}
\int_{M/Z_M} h(mg)d_Rm&=&\int_{M/Z_M} h'(m)d_R m=\int_{\GL_2(F)}\phi'(g')dg'=\int_{\GL_2(F)}\phi(g'g)dg'\\
&=&\int_{M/Z_M}h(m)d_Rm,
\end{eqnarray*}
and the result follows.
\end{proof}

\begin{corollary}[Comparison]\label{coras2}
Let $M/Z_M=M_1/Z_{M_1}$ or $M_2/Z_{M_2}$ and $h\in C_c(\GL_2(F),\C)$ be as above. Then there is a constant $C=C(M_1/Z_{M_1},M_2/Z_{M_2})$, such that
$$
\int_{M_1/Z_{M_1}} h(m_1)d_R m_1=C\int_{M_2/Z_{M_2}} h(m_2)d_Rm_2.
$$
\end{corollary}
\begin{proof}
By Lemma \ref{lemmeas}, there is $\phi\in C_c(\GL_2(F),\C)$, which satisfies $h(g)=\int_P\phi(bg)d_Lb$, for all $b\in P$ and $g\in \GL_2(F)$. This implies that 
$$
\int_{M_1/Z_{M_1}} h(m_1)d_R m_1=C'\int_{\GL_2(F)}\phi(g)dg=C\int_{M_2/Z_{M_2}} h(m_2)d_Rm_2,
$$
and the result follows.
\end{proof}

\section{Appendix 2: The $(\dG,O)$-module of discrete series}\label{DiscSer}

Let $\dG$ be the Lie algebra of $\GL_2(\R)$, let us consider the maximal compact subgroup $O=O(2)$, and let $H:=SO(2)\subset O$. Denote by $P$ the Borel subgroup
\[
P=\left\{u\left(\begin{array}{cc}y^{1/2}&xy^{-1/2}\\&y^{-1/2}\end{array}\right)\in\GL_2(\R)^+\right\}
\]
For our purposes, we will be interested in the $(\dG,O)$-module $\cD=\cD(2)$ of discrete series of weight 2. It can be described as a subrepresentation of an induced representation from a character of the Borel subgroup $P$. Indeed, since any $g\in\GL_2(\R)^+$ can be written uniquely as $g=u\cdot \tau(x,y)\cdot\kappa(\theta)$, where 
\[
u\in\R^+,\;\tau(x,y)=\left(\begin{array}{cc}y^{1/2}&xy^{-1/2}\\&y^{-1/2}\end{array}\right)\in P,\;\kappa(\theta)=\left(\begin{array}{cc}\cos\theta&\sin\theta\\ -\sin\theta&\cos\theta\end{array}\right)\in\SO(2),
\]
we write
\[
\tilde I=\{f\in C^\infty(\GL_2(\R)^+,\C),\mbox{ such that }f(u\cdot \tau(x,y)\cdot\kappa(\theta))=yf(\kappa(\theta))\}.
\]
Then $I$ is the $(\dG,H)$-module of admissible vectors in $\tilde I$, namely, the set of $f\in\tilde I$ such that the Fourier series of $f(\kappa(\theta))$ is finite. Thus, 
\[
I=\bigoplus_{k\in\Z}\C f_{2k};\qquad f_{2k}(u\cdot \tau(x,y)\cdot\kappa(\theta))=ye^{2ki\theta}.
\]
The $(\dG,H)$-module structure of $I$ can be described as follows: Let $L,R\in\dG$ be the \emph{Maass differential operators} defined in \cite[\S 2.2]{Bump}
\[
 L=e^{-2i\theta}\left(-iy\frac{\partial}{\partial x}+y\frac{\partial}{\partial y}-\frac{1}{2i}\frac{\partial}{\partial \theta}\right),\quad R=e^{2i\theta}\left(iy\frac{\partial}{\partial x}+y\frac{\partial}{\partial y}+\frac{1}{2i}\frac{\partial}{\partial \theta}\right).
\]
Then, the $(\dG,H)$-module $I$ is characterized by the relations:
\begin{eqnarray}
Rf_{2k}=(1+k)f_{2k+2};&\qquad& Lf_{2k}=(1-k)f_{2k-2}; \label{rel1}\\
\kappa(t)f_{2k}=e^{2kit}f_{2k};&\qquad&u f_{2k}=f_{2k},\label{rel2}
\end{eqnarray}
for any $\kappa(t)\in\SO(2)$ and $u\in \R^{+}\subset\GL_2(\R)^+$. Since $f_{2k}\in\C R^kf_0$, if $k>0$, and $f_{2k}\in\C L^{-k}f_0$, if $k<0$, we have that $I$ is generated by $f_0$ .

To provide structure of $(\dG,O)$-module, we have to define the action of $\omega=\left(\begin{array}{cc}-1&\\&1\end{array}\right)\in O(2)\setminus SO(2)$. Therefore, we have to define 
$\omega\in\End(I)$, such that 
\[
(i)\quad\omega f_{2k}\in\C f_{-2k};\qquad (ii)\quad\omega^2=1;\qquad (iii)\quad\omega R=L\omega.
\]
If we write $\omega f_{2k}=\lambda(k)f_{-2k}$, condition $(ii)$ implies that $\lambda(k)\lambda(-k)=1$. Moreover, condition $(iii)$ implies that $\lambda(k)=\lambda(k+1)$. %, if $k\neq 1$.Assuming that $\lambda(0)=1$,
We obtain two possible $(\dG,O)$-module structures for $I$: Letting $\lambda(k)=1$ for all $k\in\Z$, or letting $\lambda(k)=-1$ for all $k\in\Z$. We denote by $I^+$ and $I^-$ the corresponding $(\dG,O)$-module structures.

Note that we have a well defined morphism of $(\dG,H)$-modules 
\[
I\longrightarrow\C,\qquad f\mapsto\int_0^{\pi}f(\kappa(\theta))d\theta.
\]
Write $\C\simeq \C(+1)$ for the vector space $\C$ with trivial $\GL_2(\R)$-action, and 
$\C(-1)$ for the vector space $\C$ equipped with the action of $\GL_2(\R)$ given by the character $g\mapsto{\rm sign}(\det(g))$.
The above expresion defines a morphism of $(\dG,O)$-modules $pr^\pm:I^\pm\rightarrow\C(\pm 1)$. The kernels of both morphisms are isomorphic as $(\dG,O)$-modules by means of the isomorphism
\[
\ker(pr^+)\longrightarrow\ker(pr^-);\qquad f_{2k}\longmapsto {\rm sign}(k)f_{2k},
\] 
its isomorphism class is called \emph{discrete series representation} $\cD$. It is an irreducible $(\dG,O)$-module generated by $f_2$. Moreover, we have constructed two different extensions of $\cD$:
\begin{eqnarray}
0\longrightarrow \cD\longrightarrow &I^+&\longrightarrow\C\longrightarrow 0;\label{ext-seq-GK1}\\
0\longrightarrow \cD\longrightarrow &I^-&\longrightarrow\C(-1)\longrightarrow 0.\label{ext-seq-GK2}
\end{eqnarray}
%\[
%f_2\left(\begin{array}{cc}a&b\\c&d\end{array}\right)=(ad-bc)(ci+d)^{-2}, \quad \left(\begin{array}{cc}a&b\\c&d\end{array}\right)\in G(F_\sigma).
%\]

\bibliographystyle{plain}
\bibliography{Anticyclotomic}

\end{document}